\documentclass[final,onefignum,onetabnum]{siamart171218}
\usepackage{lineno}
\usepackage{xfp}

\newsavebox{\ReferenceTableBox}

\usepackage{lipsum}
\usepackage{amsfonts,amssymb}
\usepackage{graphicx}
\usepackage{epstopdf}
\usepackage{bm}
\usepackage{algorithmic}
\usepackage{graphicx}  
\usepackage{float}  
\usepackage{subfigure}  
\usepackage{mathrsfs}
\usepackage{booktabs}
\usepackage{multirow}
\usepackage{makecell}
\usepackage{comment}
\usepackage{lineno}
\AtBeginDocument{\nolinenumbers}

\ifpdf
  \DeclareGraphicsExtensions{.eps,.pdf,.png,.jpg}
\else
  \DeclareGraphicsExtensions{.eps}
\fi

\newsiamremark{remark}{Remark}
\newsiamremark{hypothesis}{Hypothesis}
\crefname{hypothesis}{Hypothesis}{Hypotheses}
\newsiamthm{claim}{Claim}
\theoremstyle{plain}
\theoremheaderfont{\normalfont\scshape}
\theorembodyfont{\normalfont\itshape}
\theoremseparator{.}
\theoremsymbol{}

\newtheorem{thm}{Theorem}[section]
\newtheorem{lem}{Lemma}[section]
\newtheorem{cor}{Corollary}[section]
\newtheorem{rmk}{Remark}[section]
\newtheorem{prop}{Proposition}[section]

\theoremstyle{plain}
\theoremheaderfont{\normalfont\itshape}
\theorembodyfont{\normalfont}
\theoremseparator{.}
\theoremsymbol{}

\definecolor{teal}{RGB}{0,128,128}
\newcommand{\wll}[1]{{\color{blue}#1}}
\headers{Birkhoff-Preconditioners}{S. Li, Z. Gao, Y. Jiao and L. Wang}

\title{Well-Conditioned Birkhoff-Collocation Methods for  Elliptic-type Problems in Multiple Dimensions\thanks{%
\textbf{Funding:} The research of the second and fourth authors is partially supported by  Singapore MOE AcRF Tier 2 Grant: MOE-T2EP20224-0012. The third author is partially supported by Nature Science Foundation of China grant No. 12271365. The first author would like to acknowledge the support of China Scholarship Council (CSC, No. 202508310113) to visit Nanyang Technological University, Singapore, and work on this topic.}}
\newcommand{\LiAuthorNotes}{%
  \thanks{Department of Mathematics, Shanghai Normal University, Shanghai 200234, China.}%
}
\newcommand{\GaoAuthorNotes}{%
  \thanks{Division of Mathematical Sciences, School of Physical and Mathematical Sciences, Nanyang Technological University, 637371, Singapore.}%
}

\author{%
Shunchang Li\LiAuthorNotes
\and Zixuan Gao\GaoAuthorNotes
\and Yujian Jiao\footnotemark[2]
\and Li-Lian Wang\footnotemark[3]\kern0.45em\thanks{Corresponding author. Email: lilian@ntu.edu.sg (L. Wang). The first two authors (S. Li and Z. Gao) contribute equally to this work.}
}

\usepackage{cite}
\usepackage{amsopn}

\newsavebox{\snapshotbox}
\makeatletter
\def\refstepcounter@optarg[#1]#2{%
  \cref@old@refstepcounter{#2}%
  \cref@constructprefix{#2}{\cref@result}%
  \@ifundefined{cref@#1@alias}%
    {\def\@tempa{#1}}%
    {\def\@tempa{\csname cref@#1@alias\endcsname}}%
  \protected@edef\cref@currentlabel{%
    [\@tempa][\arabic{#2}][\cref@result]%
    \csname p@#2\endcsname\csname the#2\endcsname}%
}
\makeatother
\ifpdf
\hypersetup{
  pdftitle={Well-Conditioned Birkhoff-Collocation Methods for Elliptic-type Problems in Multiple Dimensions},
  pdfauthor={Shunchang Li, Zixuan Gao, Yujian Jiao, Li-Lian Wang}
}
\fi

\begin{document}

\maketitle

\begin{abstract} Collocation methods based on Birkhoff interpolation at Gaussian-type points have been proven to be well-conditioned for initial and boundary value problems [Wang et al., {\em SIAM J. Sci. Comput.} 36 (2014)]. This development was regarded as revolutionary in pseudospectral (PS) optimal control, effectively rendering earlier PS methods using Lagrange interpolation ``obsolete'' [I. Ross, {\em J. Guid. Control Dyn.} 47 (2024)]. However, extending this one-dimensional construction, which relies on inverting the leading differential operator in the discrete sense, to multiple dimensions is not possible. 
This paper aims to resolve the long-standing issue of ill-conditioning in multidimensional collocation methods for second-order elliptic-type problems.  The key is to show that the second-order differentiation matrix (DM) and its inverse, namely, the PS integration matrix (PSIM) constructed from Birkhoff interpolation at Legendre-Gauss-Lobatto (LGL) points, are both similar to symmetric negative definite matrices. This, in turn, enables stable diagonalisation of the non-symmetric, dense and ill-conditioned  PSIM and  DM even with thousands of LGL points, and then construct optimal Birkhoff-preconditioners for well-conditioned collocation methods for elliptic problems. Notably, we find that it is essential to properly incorporate the variable coefficients into the diagonalisation and the construction of preconditioners, particularly for highly anisotropic elliptic problems involving high-contrast, oscillatory, or degenerating coefficients. We carry out  spectral analysis and present extensive numerical results demonstrating the well-conditioning and stability of the proposed Birkhoff-collocation methods for various challenging second-order elliptic-type problems, including large-scale problems in two and three dimensions. In a nutshell, the proposed Birkhoff-collocation schemes make collocation methods practical and establish them as compelling high-order solvers of choice.   
\end{abstract}



\section{Introduction}\label{sec: introduction}

The spectral collocation method can be viewed as a global finite-difference method on Gaussian-type points \cite{Fornberg1996Pseudospectral,Fornberg2025FiniteDifference}, with an appealing plug-and-play feature: derivatives at the collocation points are simply replaced by the corresponding differentiation matrices. This makes the method particularly attractive for variable-coefficient and nonlinear problems, and it has also proved advantageous as a local solver in hierarchical Poincar\'e-Steklov (HPS) spectral element methods \cite{Martinsson2013HPS,Fortunato2024SurfacePDE}. However, the resulting linear systems are typically dense and ill-conditioned, rendering high-order computations increasingly expensive and potentially unstable. Indeed, it was observed in \cite{Fortunato2024SurfacePDE} that collocation-based HPS spectral element methods become less efficient and even unstable as the number of collocation points within each element increases. Ill-conditioning thus remains a major obstacle to the practical use of spectral collocation methods.

This issue in one dimension was successfully  resolved by Wang et al. \cite{wang2014well}, where the key was to introduce suitable Birkhoff interpolation problems that led to the so-called Birkhoff basis polynomials for constructing perfect preconditioners. This work has had a far-reaching impact on the subsequent developments of well-conditioned spectral collocation and pseudospectral methods, with important applications to pseudospectral (PS) optimal control (see \cite{McCoidTrummer2018,Koeppen2019FastMesh,RossProulxBorges2023,Ross2024Birkhoff} and \cite[Chapter 4]{SongZhaoTheil2023}).
In particular, Ross \cite{Ross2024Birkhoff} described the work \cite{wang2014well} as ``pioneering'' and ``revolutionary'', effectively rendering earlier Lagrange-based PS methods ``obsolete''.
It is noteworthy that very recently, Javeed et al. \cite{JaveedKouriRidzalSteinman2025} extended this integration-preconditioning approach to an \(\mathcal{O}(n)\)-applicable preconditioner with $n+1$ collocation points for one-dimensional first-order variable-coefficient problems.

In the past decades, several attempts have been made to alleviate the ill-conditioning of multidimensional spectral collocation systems using finite-difference or finite-element preconditioners \cite{DevilleMund1985,KimParter1996,KimParter1997,Parter2001FD,Parter2001FE}.  In particular, Fang et al. \cite{fang2018preconditioning} developed efficient finite-difference preconditioners for self-adjoint elliptic problems with separable variable coefficients. 
However,  it is more natural and practical to consider the preconditioners based on consistent collocation discretisation, but
unfortunately, the one-dimensional  construction in \cite{wang2014well} cannot be generalized to multiple dimensions. In fact, a direct spectral analysis (see Remark~\ref{Rmk:spCI}) shows that tensor-product Birkhoff basis functions do not alleviate the ill-conditioning of the resulting collocation systems, but only alter the eigenvalue distributions of the usual Lagrange collocation system.  In this paper, we introduce new tools to address the long-standing ill-conditioning of multidimensional Legendre collocation methods.
The main idea and contributions are as follows.
\smallskip 
\begin{itemize}
    \item[(i)] We derive two important identities for the second-order DM $\mathbf D^{(2)}$ and its inverse PSIM $\mathbf B=(\mathbf D^{(2)})^{-1};$ see Lemma \ref{WinD2in_eqA}.  These identities allow us to  identify  a symmetric negative definite matrix $\mathbf S_c,$ similar to $\mathbf{B}_c:= \mathbf{C}^{-\frac{1}{2}}\mathbf{B}   \mathbf{C}^{-\frac{1}{2}},$ where the positive diagonal matrix $\mathbf  C$ is to incorporate the variable coefficients. The symmetric matrix $\mathbf S_c$ admits a stable orthogonal diagonalisation, which in turn yields stable diagonalisation of the non-symmetric matrices
    $\mathbf B$ and $\mathbf D^{(2)};$ see Theorem \ref{WinD2in_eq_D2inTWin}.
    \medskip
    
    \item[(ii)] The stable diagonalisation provides the key ingredients for the construction of multidimensional Birkhoff preconditioners. The basic idea is to approximate the principal part of the elliptic operator by a separable one. For example, in two dimensions, we consider
\begin{equation}\label{mathL0}
    {\mathcal L}_0
    :=-a_{11}(x,y)\partial_x^2-a_{22}(x,y)\partial_y^2
    \;\;\;\Rightarrow\;\;\;
    \widehat{\mathcal L}_0
    :=-a(x)\partial_x^2-b(y)\partial_y^2,
\end{equation}
on $\Omega=(-1,1)^2$, where
\begin{equation}\label{mathab}
    a(x)=\frac{1}{2}\int_{-1}^1 a_{11}(x,y)\,{\rm d}y,
    \qquad
    b(y)=\frac{1}{2}\int_{-1}^1 a_{22}(x,y)\,{\rm d}x.
\end{equation}
The explicit inversion of the Lagrange or Birkhoff collocation systems associated with $\widehat{\mathcal L}_0$, obtained through the stable diagonalisation in (i) with $c=a,b$, then yields the corresponding (right) Birkhoff preconditioners for ${\mathcal L}_0$, and more generally for second-order elliptic-type operators $\mathcal L$; see \eqref{defn38} and Theorem~\ref{prop:explicit_preconditioners}. We further carry out a spectral analysis of the preconditioned collocation matrices and establish eigenvalue clustering under suitable conditions. In particular, we prove the well-conditioning of the preconditioned collocation systems in some special setting.

    

    \medskip
\item[(iii)]  We apply the proposed collocation methods to a broad range of challenging elliptic-type problems in two and three dimensions, including (a) problems with highly oscillatory coefficients or strongly anisotropic operators, with coefficient contrasts up to \(O(10^{10})\); (b) elliptic problems on deformed triangles and tetrahedra obtained through singular Duffy transformations; and (c) elliptic problems on domains with curved boundaries or surfaces using Gordon-Hall transformations. In all these cases, we find that it is essential to properly incorporate the variable coefficients into the construction of the preconditioners, as outlined in (ii). The proposed preconditioners remain robust and stable across all the challenging test cases, yielding either uniformly bounded or substantially reduced condition numbers together with nearly \(N\)-independent GMRES iteration counts. By contrast, the original unpreconditioned collocation schemes fail to converge in most of these cases. As a by-product, the collocation system for the Helmholtz operator
\(
{\mathcal L}[u]=-\Delta u+\gamma u
\),
with a real constant \(\gamma\), on \((-1,1)^d\), \(d=2,3\), can be explicitly inverted through the diagonalisation developed in (i). As an application, we employ the  stable collocation solver for the spatial discretisation of the time-dependent Allen-Cahn equation with a very small phase-separation parameter. The method is competitive with existing adaptive spectral methods and, for complicated initial data, can outperform the adaptive approach. 
\end{itemize}


\medskip

The remainder of this paper is organized as follows. In
Section~\ref{sec: inv_diag}, we introduce and develop two key one-dimensional tools for  forthcoming 
multidimensional constructions in  Sections~\ref{sec: 2D_Preconditioner} and \ref{sec:extension}. In Section~\ref{sec: 2D_Preconditioner}, we present the detailed construction of the preconditioners in two dimensions, conduct the spectral analysis, and provide ample convincing numerical evidence.   
We then extend the framework to three dimensions
in Section~\ref{sec:extension}.
 Section~\ref{sec:conclusion} is 
 for conclusion and final remarks.

\smallskip

We conclude this introductory section with some notation to be used throughout this paper.  The scalars, vectors, and matrices are denoted by
italic, bold italic, and bold upright letters, respectively; for
example, \(f\), \(\bm{u}=(u_i)\), and
\(\mathbf{A}=(A_{ij})\).
Let \(\mathbb{P}_N\) denote the space of polynomials of degree at most
\(N\), and let \(L_k\) be the Legendre polynomial of degree \(k\).
We use \(\mathbf{I}_n\) for the \(n\times n\) identity matrix.
For \(\mathbf{M}\in\mathbb{R}^{n\times n}\),
\(\mathbf{M}^{\intercal}\), \(\|\mathbf{M}\|_2\), and
\(
\kappa_2(\mathbf{M})
=
\|\mathbf{M}\|_2\|\mathbf{M}^{-1}\|_2
\)
denote its transpose, Euclidean norm, and \(2\)-norm condition number,
respectively.
We also use \(\otimes\) for the Kronecker product and
\(\operatorname{vec}(\cdot)\) for the vectorization operator.

\section{Two key tools for constructing multidimensional preconditioners}\label{sec: inv_diag} 
In this section, we introduce two indispensable techniques for the construction of   multidimensional preconditioners in the forthcoming sections. The first is to compute the inverses of differentiation matrices by introducing suitable Birkhoff interpolation problems \cite{wang2014well}.
The second is to diagonalise the differentiation matrices and their inverse matrices in a stable manner. 


\subsection{Stable inversion of differential matrices  via Birkhoff interpolation}\label{subsec:PDM-unique}  
Let $\{x_j, \omega_j\}_{j=0}^{N}$ (with $x_0 = -1, x_N = 1$) be the Legendre-Gauss-Lobatto (LGL) nodes and quadrature weights. 
Then the LGL quadrature rule enjoys the exactness 
\begin{equation}\label{exactLGL}
    \int_{-1}^{1}\phi(x){\rm d}x = \sum_{j=0}^{N}\phi(x_j)\omega_j,\quad \forall \phi\in \mathbb{P}_{2N-1}.
\end{equation}
Let $\{\ell_j(x)\}_{j=0}^N$ be  the Lagrange interpolation basis polynomials such that 
$\ell_j(x)\in\mathbb{P}_N$ and $\ell_j(x_i)=\delta_{ij}$.  Define the corresponding  differentiation matrices of order $k\ge 1$ as 
\begin{equation}\label{Definition: Lagrange_Differentiation_Matrix}
    \mathbf{D}^{(k)} = \big(d_{ij}^{(k)}:=\ell_{j}^{(k)}(x_i)\big)_{1\le i,j \le N-1},\quad  \widetilde{\mathbf{D}}^{(k)} = \big(d_{ij}^{(k)}:=\ell_{j}^{(k)}(x_i)\big)_{0\le i,j\le N},
\end{equation}
where the former involves only the interior LGL points.
In particular, we denote $\mathbf{D} := \mathbf{D}^{(1)}$ and $\widetilde{\mathbf{D}}:=\widetilde{\mathbf{D}}^{(1)}.$ 
In what follows, we focus on $k=2.$
It is known that  the matrix $\mathbf{D}^{(2)}$ is nonsingular and its 
$(N-1)$ eigenvalues are all negative and distinct,  satisfying  (see e.g., \cite{boulmezaoud2007eigenvalues, welfert1994eigenvalues}): 
\begin{equation}\label{eig_bound_D2}
    -c_N \frac{N^4}{4\pi^2} = \lambda_{1}< \dots < \lambda_{N-1} < -\frac{\pi^2}{4} < 0,
\end{equation}
where $c_N \approx 1$ for sufficiently large $N$.  As a result, the condition number of $\mathbf{D}^{(2)}$ grows like $\mathcal O(N^4).$ 

Wang et al.~\cite{wang2014well} introduced a stable way to compute the inverse of $\mathbf{D}^{(2)}$ through suitable  Birkhoff interpolation~\cite{Lorentz1983}.  More precisely, we consider the Birkhoff interpolation problem: 
Find $p\in\mathbb{P}_N$ such that for any  $u\in C^2(-1,1),$
\begin{equation}\label{Birkhoff-2nd}
    p(-1) = u(-1); \quad p''(x_j)=u''(x_j),\quad 1\le j\le N-1; \quad p(1)=u(1).
\end{equation}
 According to \cite[Section 3]{wang2014well}, the interpolant  $p$  is  uniquely determined by
\begin{equation*}
    p(x)=u(-1)B_0(x)+\sum_{j=1}^{N-1}u''(x_j)B_j(x)+u(1)B_N(x),\quad x\in[-1,1],
\end{equation*}
where
\(
    B_0(x)=(1-x)/2, B_N(x)=(1+x)/2,
\)
and $B_j\in {\mathbb P}_N$ satisfy 
\begin{equation*}
    B_j(\pm 1)= 0,\quad B''_j(x_i) = \delta_{ij}, \quad 1\le i,j\le N-1.
\end{equation*}
This implies  $B_j(x)$ solve  the boundary value problems: 
\begin{equation}\label{BjODE}
B_j''(x)=\tilde \ell_j(x),\quad x\in (-1,1);\quad B_j(\pm 1)=0,
\end{equation}
where $\tilde \ell_j\in {\mathbb P}_{N-2}$ are the Lagrange interpolation polynomials only involving the interior LGL points. The relation \eqref{BjODE} is essential for stable computation of $B_j$  via integration in 
\cite{wang2014well}: 
\begin{equation}\label{B-formula}
B_j(x) = (\beta_{1j} - \beta_{0j})\frac{x+1}{2} + \sum_{k=0}^{N-2}\beta_{kj}\frac{\partial_{x}^{-2}L_k(x)}{\gamma_k},\quad 1\le j\le N-1,
\end{equation}
where $\gamma_k = 2/(2k+1),\, \partial_{x}^{-2}L_k(x)=\int_{-1}^ x \int_{-1}^t L_k(s){\rm d} s {\rm d} t,$  and
\begin{equation*}
    \beta_{kj} = \Big( L_k(x_j) - \frac{1-(-1)^{N+k}}{2}L_{N-1}(x_j) - \frac{1+(-1)^{N+k}}{2}L_N(x_j) \Big)\omega_j.
\end{equation*}
Correspondingly, we introduce the $k$th-order integration matrices as the counterparts of the differentiation matrices defined in \eqref{Definition: Lagrange_Differentiation_Matrix}: 
\begin{equation*}
    \mathbf B^{(k)}=(b_{ij}^{(k)}:=B_j^{(k)}(x_i))_{1\le i,j\le N-1},\quad \widetilde{\mathbf B}^{(k)}=(b_{ij}^{(k)}:=B_j^{(k)}(x_i))_{0\le i,j\le N}.
\end{equation*}
Remarkably, $\mathbf{B} = \mathbf{B}^{(0)}$ is the  inverse of $\mathbf D^{(2)}$ (see 
 \cite[Theorem 3.3]{wang2014well}):
\begin{equation}\label{D2B_eq_I}
    \mathbf{D}^{(2)}\mathbf{B}= \mathbf{B}\mathbf{D}^{(2)} = \mathbf{I}_{N-1}.
\end{equation}
In fact, the formula \eqref{B-formula} provides a very stable way to compute
$\mathbf{B}=(\mathbf{D}^{(2)})^{-1}$ for large $N.$  The implication of 
\eqref{D2B_eq_I} is two-fold \cite{wang2014well}: (i) it provides a perfect preconditioner for the ill-conditioned collocation method for the second-order BVPs; and (ii) it offers a basis for constructing well-conditioned collocation methods.  For example, consider the simple BVP: $-u''+\gamma\, u=f$ in $(-1,1)$ with $u(\pm 1)=h_{\pm},$ where $\gamma, h_{\pm}$ are given constants and $f\in C(-1,1).$ The coefficient matrix of the preconditioned  system reads
\begin{equation}\label{BD-1D-precond}
\mathbf B (-\mathbf D^{(2)}+\gamma \mathbf I_{N-1})=-\mathbf I_{N-1}+\gamma \mathbf B.
\end{equation}
We infer from \eqref{eig_bound_D2} and \eqref{D2B_eq_I} that its eigenvalues are $-1-\gamma/\lambda_j=\mathcal O(1),$ so it is well-conditioned. In fact,  $-\mathbf I_{N-1}+\gamma \mathbf B$ corresponds to the collocation scheme directly using $\{B_j\}$ as basis functions.

\subsection{Stable diagonalisation of differentiation and integration matrices}\label{subsection: Stable_Diagonalization}
As demonstrated in \eqref{BD-1D-precond}, the stable inversion of $\mathbf{D}^{(2)}$ yields optimal preconditioners for second-order BVPs in one dimension. However, this strategy cannot be extended  to multiple dimensions. Indeed, as to be shown in Remark \ref{Rmk:spCI}, the use of a Birkhoff interpolating basis does not improve the conditioning of the collocation system, but instead alters its eigenvalue distribution. 
In higher dimensions, the key lies in the stable diagonalisation of the dense, nonsymmetric, and ill-conditioned matrices $\mathbf{D}^{(2)}$ and $\mathbf{B}$, while properly incorporating the variable coefficients in the principal part of the elliptic operator.     

We begin by establishing the following key identities. 
\begin{lem}\label{WinD2in_eqA}
    Let $\{x_i,\omega_i\}_{i=0}^{N}$ be the Legendre-Gauss-Lobatto nodes and weights as before, and let 
    $\mathbf W={\rm diag}(\omega_1,\ldots, \omega_{N-1}).$ Then we have 
\begin{equation}\label{WD2-Id}
\mathbf{W}\mathbf{D}^{(2)} =\mathbf{D}^{(2)\intercal}\mathbf{W},\quad \mathbf{W}\mathbf{B} = \mathbf{B}^{\intercal}\mathbf{W}.
    \end{equation}
     \end{lem}
\begin{proof}
 For any vectors $\bm{u} = (u_1,\dots,u_{N-1})^\intercal\in \mathbb{R}^{N-1}$ and $\bm{v} = (v_1,\dots,v_{N-1})^\intercal\in \mathbb{R}^{N-1}$, let $P(x),Q(x)\in\mathbb{P}_N^0=\{\phi\in \mathbb P_N\,:\, \phi(\pm 1)=0\}$ be the corresponding Lagrange interpolating polynomials satisfying 
    \begin{equation*}
       P(x_i)= u_i,\quad  Q(x_i)=v_i,\quad 1\le i\le N-1.
    \end{equation*}
Then one verifies readily that 
    \begin{equation*}
    \mathbf{D}^{(2)}\bm{u} = \big(P''(x_1), \ldots, P''(x_{N-1})\big)^\intercal,\quad \mathbf{D}^{(2)}\bm{v} = \big(Q''(x_1), \ldots, Q''(x_{N-1})\big)^\intercal.
    \end{equation*}
Since $PQ'',\, P''Q  \in\mathbb{P}_{2N-2}$ and $P(\pm 1)=Q(\pm 1)=0$, we derive immediately from the exactness of the LGL quadrature \eqref{exactLGL} and integration by parts  that 
    \begin{equation}\label{uWDv}
    \begin{split}
        \bm{u}^\intercal \mathbf{W}(\mathbf{D}^{(2)}\bm v) & = \sum_{i=1}^{N-1}\omega_i P(x_i)Q''(x_i)= \sum_{i=0}^{N}\omega_i P(x_i)Q''(x_i)  \\
        &= \int_{-1}^{1} P(x)Q''(x)\,{\rm d}x =  \int_{-1}^{1} P''(x)Q(x)\,{\rm d}x\\
        & = \sum_{i=1}^{N-1}\omega_i P''(x_i)Q(x_i)= (\mathbf{D}^{(2)}\bm u)^\intercal \mathbf{W}\bm v= \bm{u}^\intercal \, \mathbf{D}^{(2)\intercal}\,\mathbf{W} \bm v,
        \end{split}
    \end{equation}
   for any $\bm u,\bm v\in\mathbb{R}^{N-1},$  which implies $\mathbf{W}\mathbf{D}^{(2)}=\mathbf{D}^{(2)\intercal}\,\mathbf{W}$.
In view of \eqref{D2B_eq_I}, we have
     $$
     \mathbf{W}\mathbf{B}^{-1} = \mathbf{B}^{-\intercal}\mathbf{W}\quad \Rightarrow \quad 
       \mathbf{B}^{\intercal}\mathbf{W}= \mathbf{W}\mathbf{B},
     $$ 
     which completes the proof. 
\end{proof}
 
\begin{rmk}\label{LGL}
  It follows from the derivations in \eqref{uWDv} that the identities in \eqref{WD2-Id}  hold only for the Legendre-Gauss-Lobatto case, since integration by parts may not extend to the general Jacobi case with a nonuniform weight function. \qed
\end{rmk}

Remarkably, the identities \eqref{WD2-Id} enable us to identify a similar, symmetric negative definite matrix  that admits stable diagonalisation of the differentiation and integration matrices, with a positive diagonal matrix naturally incorporated into the construction.    
\begin{thm}\label{WinD2in_eq_D2inTWin}
    Let 
    $\mathbf W={\rm diag}(\omega_1,\ldots, \omega_{N-1})$ be the same   as above, 
    and $\mathbf{C} = {\rm diag}(c_1,\ldots, c_{N-1})$ with all $c_i>0.$
Introduce the matrices 
\begin{equation}\label{Definition:2order-Bc}
 \mathbf{B}_c:= \mathbf{C}^{-\frac{1}{2}}\mathbf{B}  \mathbf{C}^{-\frac{1}{2}}, \quad \mathbf{S}_c := \mathbf{W}^{\frac{1}{2}} \mathbf{B}_c\mathbf{W}^{-\frac{1}{2}}.
 \end{equation}
 Then  the matrix $\mathbf{S}_c $ is  symmetric negative definite,
which admits  the orthogonal diagonalisation:
\begin{equation}\label{eq:Sc_eig_var_coeff}
\mathbf{S}_c=\mathbf{Q}_c\mathbf{\Sigma}_c\mathbf{Q}_c^{\intercal}, \;\;\; {\rm where} \;\;\; \mathbf{Q}_c^{\intercal}  \mathbf{Q}_c=\mathbf{I}_{N-1}, \;\;\mathbf{\Sigma}_c = {\rm diag}(\sigma_{c,1},\ldots, \sigma_{c,N-1}),
\end{equation}
leading to the diagonalisations:  
\begin{equation}\label{eq:DcD2_diag_var_coeff}
    \mathbf{B} \mathbf{C}^{-1} = \mathbf{V}_{c} \mathbf{\Sigma}_c \mathbf{V}_{c}^{-1}, \quad 
    \mathbf C\mathbf D^{(2)}
=
\mathbf V_c\mathbf\Sigma_c^{-1}\mathbf V_c^{-1},
\end{equation}
where
 \begin{equation}\label{VcVc}
  \mathbf{V}_{c} = \mathbf{C}^{\frac{1}{2}}\mathbf{W}^{-\frac{1}{2}} \mathbf{Q}_c, \quad   \mathbf{V}_c^{-1} = \mathbf{Q}_c^{\intercal}\mathbf{W}^{\frac{1}{2}}\mathbf{C}^{-\frac{1}{2}}.
  \end{equation}
\end{thm}
\begin{proof} 
  Using $\mathbf W\mathbf B=\mathbf B^\intercal\mathbf W$ from
Lemma~\ref{WinD2in_eqA}, we obtain
\begin{equation*}
\begin{split}
    \mathbf W\mathbf B_c
& =
\mathbf W\mathbf C^{-\frac 12}\mathbf B\mathbf C^{-\frac 12}
=
\mathbf C^{-\frac 12} \mathbf W\mathbf B\mathbf C^{-\frac 12}
=
\mathbf C^{-\frac 12}\mathbf B^\intercal\mathbf W\mathbf C^{-\frac 12}
\\&=
\mathbf C^{-\frac 12}\mathbf B^\intercal\mathbf C^{-\frac 12} \mathbf W
=
\mathbf B_c^\intercal\mathbf W ,
\end{split}
\end{equation*}
which implies 
\begin{equation*}
 \begin{split}
        \mathbf{S}^\intercal_c & = \mathbf{W}^{-\frac 12}\mathbf{B}^\intercal_c \mathbf{W}^{\frac{1}{2}} =  \mathbf{W}^{-\frac{1}{2}}(\mathbf{W}\mathbf{B}_c\mathbf{W}^{-1})\mathbf{W}^{\frac{1}{2}}
        = \mathbf{W}^{\frac 12}\mathbf{B}_c\mathbf{W}^{-\frac{1}{2}} = \mathbf{S}_c.
        \end{split}
    \end{equation*}
Thus, $\mathbf S_c$ is a symmetric matrix. As $\mathbf C$ is
positive diagonal,  $\mathbf B_c$ is negative definite, so is 
$\mathbf S_c$.
Hence it admits the orthogonal eigendecomposition \eqref{eq:Sc_eig_var_coeff}. 
Then by the definition \eqref{Definition:2order-Bc}, we find 
$$
\mathbf{B}_c = \mathbf{W}^{-\frac{1}{2}} \mathbf{S}_c\mathbf{W}^{\frac{1}{2}}=
\mathbf{W}^{-\frac{1}{2}} \mathbf{Q}_c\mathbf{\Sigma}_c\mathbf{Q}_c^{\intercal}\mathbf{W}^{\frac{1}{2}},
$$
so 
$$
\mathbf{B}\mathbf C^{-1} = \mathbf C^{\frac 1 2} \mathbf{W}^{-\frac{1}{2}} \mathbf{S}_c\mathbf{W}^{\frac{1}{2}}\mathbf{C}^{-\frac{1}{2}}=
\mathbf C^{\frac 1 2}  \mathbf{W}^{-\frac{1}{2}} \mathbf{Q}_c\mathbf{\Sigma}_c\mathbf{Q}_c^{\intercal}\mathbf{W}^{\frac{1}{2}} \mathbf C^{-\frac 1 2}= \mathbf{V}_{c} \mathbf{\Sigma}_c \mathbf{V}_{c}^{-1},
$$
where  $\mathbf{V}_{c}, \mathbf{V}_{c}^{-1}$ are given in \eqref{VcVc}.  
Therefore, the diagonalisation of $\mathbf C \mathbf D^{(2)}$ is a direct consequence of  \eqref{D2B_eq_I}.
\end{proof}

\begin{rmk}\label{C=IN} 
The matrix $\mathbf C$ in Theorem \ref{WinD2in_eq_D2inTWin} is to incorporate the variable coefficients in the diagonalisation, which turns to be essential for the highly anisotropic or degenerating  elliptic problems. 
\end{rmk}

It is clear that if $\mathbf C=\mathbf I_{N-1},$ we have the diagonalisation of  $\mathbf D^{(2)}$ and $\mathbf B$ from the above. In this case, we omit the subscript $c$  and simply denote $\mathbf S= \mathbf S_c, \mathbf Q=\mathbf Q_c$ and similarly for other ones.
According to the Davis-Kahan perturbation theorem
\cite{davis1970rotation}, the diagonalisation of $\mathbf B$ through the similar matrix $\mathbf S$ is nearly independent of the conditioning of $\mathbf B,$ but  
more on the perturbation of the
invariant eigenspaces governed by the size of
the matrix perturbation relative to the spectral gap. In Figure~\ref{fig:stable}, 
we depict the relative symmetry error
\(
\|\mathbf S-\mathbf S^{\intercal}\|_2/\|\mathbf S\|_2
\) and orthogonality error \(\|\mathbf Q^\intercal\mathbf Q-\mathbf I\|_2\)
for $N$ up to $16384,$ which shows this diagonalisation process is stable even for very large $N.$ 
\begin{figure}[htbp]
    \centering
\includegraphics[width=0.35\linewidth]{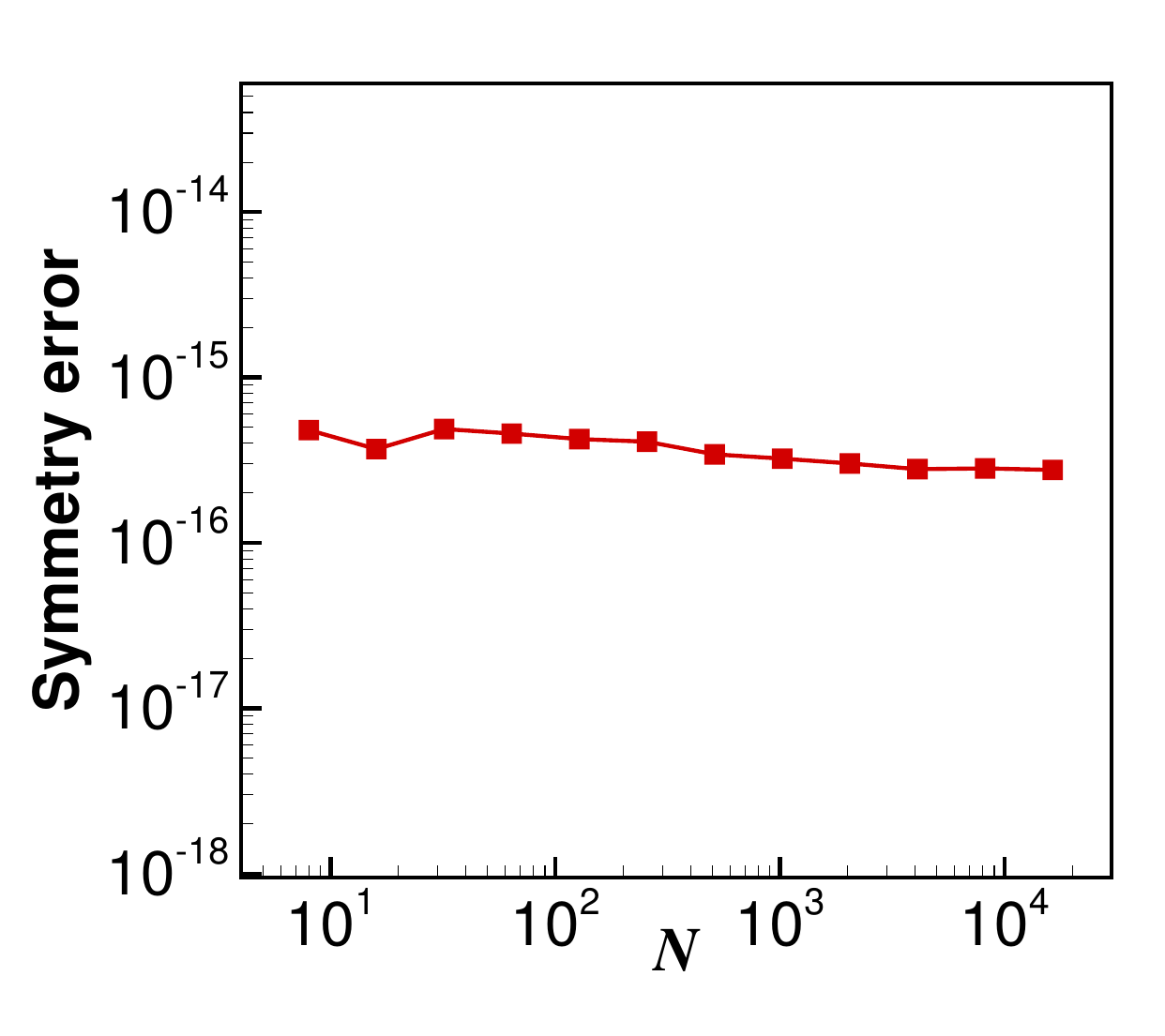}
    \qquad 
\includegraphics[width=0.35\linewidth]{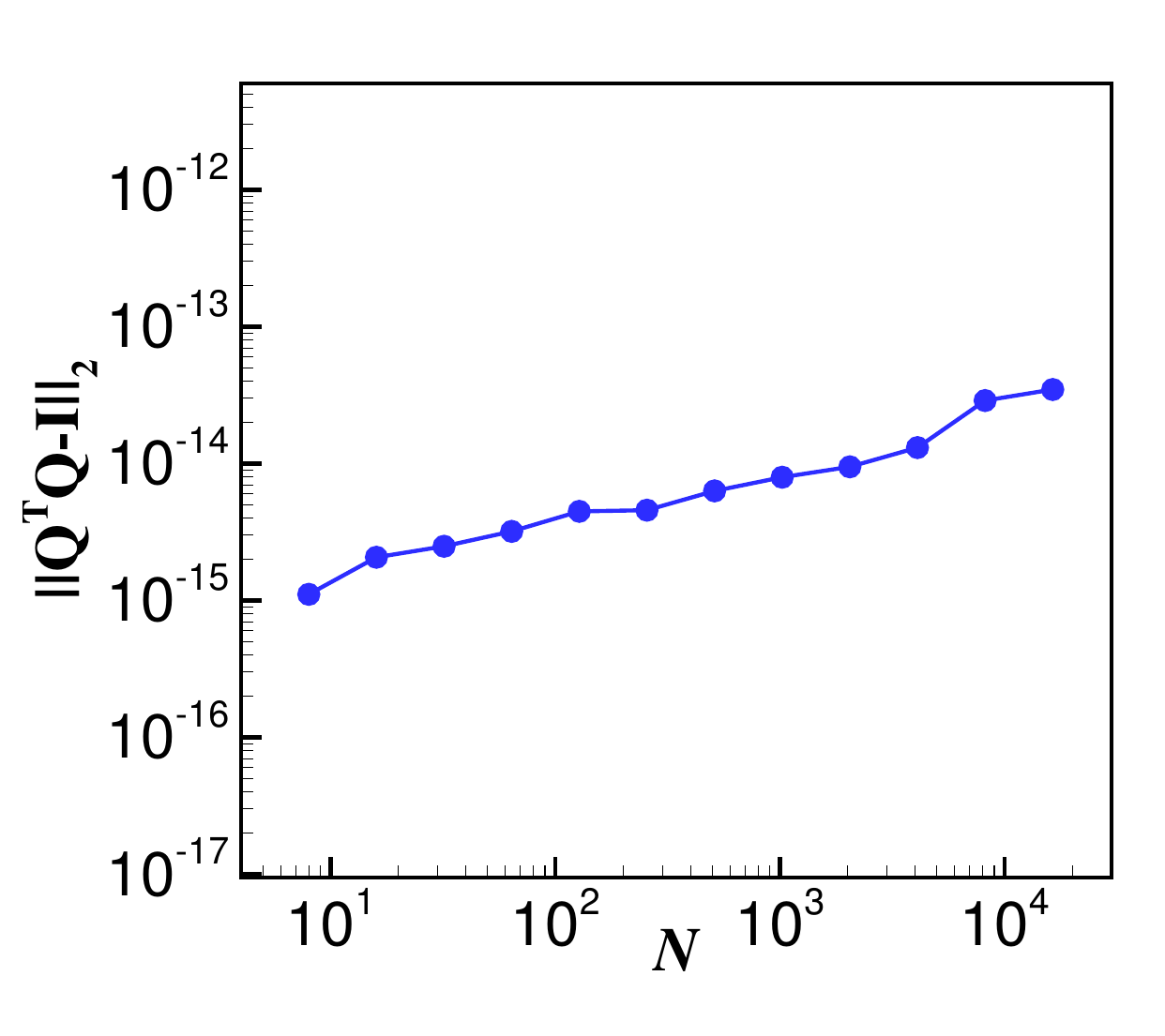}  
    \caption{
 Symmetry and orthogonality errors against  polynomial degree \(N\). Left:  \(
\|\mathbf S-\mathbf S^{\intercal}\|_2/\|\mathbf S\|_2
\). Right: \(\|\mathbf Q^\intercal\mathbf Q-\mathbf I\|_2\).}
    \label{fig:stable}
\end{figure}



We conclude this section with the following result on the conditioning of the  matrix $\mathbf V$ in the diagonalisation. 
\begin{cor}
    \label{Proposition:Cond_S_V}
    Let $\mathbf W$ and $\mathbf V$ be the matrices as defined in Theorem \ref{WinD2in_eq_D2inTWin} with $\mathbf C=\mathbf I_{N-1}$. Then their condition numbers grow like
    $\kappa_2(\mathbf{V}) = \sqrt{\kappa_2(\mathbf{W})}=\mathcal{O}(\sqrt{N}).$
\end{cor}
\begin{proof} 
    Since $\mathbf{Q}$ is orthogonal and
$\mathbf{V}=\mathbf{W}^{-\frac{1}{2}}\mathbf{Q}$, we have
    \(
        \|\mathbf{V}\|_2 = \|\mathbf{W}^{-\frac{1}{2}}\|_2\) and \( \|\mathbf{V}^{-1}\|_2 = \|\mathbf{W}^{\frac{1}{2}}\|_2.
    \)
    Hence,
    \begin{equation*}
        \kappa_2(\mathbf{V}) = \|\mathbf{V}\|_2\|\mathbf{V}^{-1}\|_2 = \|\mathbf{W}^{-\frac{1}{2}}\|_2\|\mathbf{W}^{\frac{1}{2}}\|_2 = \sqrt{\kappa_2(\mathbf{W})}.
    \end{equation*}
    Recall the following bounds for the LGL weights (see e.g. \cite{Szego1975}): 
    \begin{equation*}
        \frac{c_1}{N^2} \le \omega_j \le \frac{c_2}{N},\quad 0\le j\le N,
    \end{equation*}
    where $c_1$ and $c_2$ are positive constants independent of $N$. Hence, 
    $\kappa_2({\mathbf{W}})=\mathcal{O}(N).$
     This completes the proof.
     \end{proof}

\section{Birkhoff preconditioners in two dimensions}\label{sec: 2D_Preconditioner}

With the two building blocks in place, we are now ready to construct Birkhoff preconditioners in multiple dimensions.    To fix ideas, we focus on two dimensions in this section.


\subsection{General setup} We begin with a general setup, and consider
\begin{equation}\label{2order_operator}
\mathcal{L}u=f
\quad \text{in }\;  \Omega=(-1,1)^2;
\quad
u=g
\quad \text{on }\; \partial\Omega,
\end{equation}
where \(\mathcal{L}\) is a second-order differential operator. Since the
preconditioners are constructed from the equations at the interior
collocation points, we assume \(g=0\); nonhomogeneous boundary data can
be incorporated into the right-hand side.

Let 
\(\{x_i\}_{i=0}^{N}\) and \(\{y_j\}_{j=0}^{N}\) be the
LGL points in  two coordinate directions. Then it is standard to formulate the collocation scheme on the tensorial grids, and represent 
the numerical solution in terms of the tensorial Lagrange and
Birkhoff bases as
\begin{equation}\label{eq:2d_Lagrange_Birkhoff_expansions}
u_N(x,y)
=
\sum_{i,j=1}^{N-1}
u_{ij}\ell_i(x)\ell_j(y)
=
\sum_{i,j=1}^{N-1}
\widehat{u}_{ij}B_i(x)B_j(y).
\end{equation}
Let
$$
\bm{u}
=
\operatorname{vec}\bigl((u_{ij})_{i,j=1}^{N-1}\bigr),\quad 
\widehat{\bm{u}}
=
\operatorname{vec}\bigl((\widehat{u}_{ij})_{i,j=1}^{N-1}\bigr),\quad 
\bm{f}
=
\operatorname{vec}\bigl((f(x_i,y_j))_{i,j=1}^{N-1}\bigr).
$$
The corresponding Lagrange- and Birkhoff-collocation systems are
\begin{equation}\label{systems:2order_operator}
 \mathbf{A}_{\rm L}\bm{u}=\bm{f}\;\;\; \text{(LCOL)}; 
\quad
\mathbf{A}_{\rm B}\widehat{\bm{u}}=\bm{f}\;\;\; \text{(BCOL)},
\end{equation}
where \(\mathbf A_{\rm L}\) and \(\mathbf A_{\rm B}\) denote the
coefficient matrices of LCOL and BCOL, respectively. Note that we can recover the nodal values from  BCOL via $\bm u
=
(\mathbf B\otimes\mathbf B)\widehat{\bm u}.$
The main objective of this paper is to construct suitable right preconditioners: $\mathbf P_{\rm L}, \mathbf P_{\rm B}$ from the stable orthogonalisation in Theorem \ref{WinD2in_eq_D2inTWin}, leading to the preconditioned system:  
\begin{equation}\label{BP-systems} 
 \mathbf{A}_{\rm L} \mathbf{P}_{\rm L}\bm{v}=\bm{f}\;\;\; \text{(PLCOL)}; 
\quad
\mathbf{A}_{\rm B} \mathbf{P}_{\rm B}\widehat{\bm{v}}=\bm{f}\;\;\; \text{(PBCOL)}.
\end{equation}
Then we evaluate  
\begin{equation}\label{BP-systems-uu} 
 \bm u= \mathbf{P}_{\rm L}\bm{v}\;\;\; \text{(PLCOL)}; 
\quad
 \bm u
=
(\mathbf B\otimes\mathbf B)(\mathbf{P}_{\rm B}\widehat{\bm{v}})\;\;\; \text{(PBCOL)}.
\end{equation}

We shall present the explicit forms of the matrices and linear systems once the operator ${\mathcal L}$ is specified for the cases to be considered later.  


\subsection{Elliptic problems with separable variable coefficients}\label{subsec:var_coeff_sep_principal} 
We first consider  the operator
$\mathcal L$ \eqref{2order_operator} in the following separable form: 
\begin{equation}\label{var-coeff}
\mathcal{L}[u]=-a(x)\partial_{xx}u-b(y)\partial_{yy} u+c(x)d(y)u \quad\text{in }\; \Omega=(-1,1)^2,
\end{equation}
where \(a\), \(b\), \(c\), \(d\) are continuous functions, and there exists a constant $\beta$
such that
\begin{equation}\label{abeta}
0< a(x),\,b(y)\leq\beta,
\quad
c(x)d(y)\geq0
\quad\text{in }\; \Omega.
\end{equation}
Note that here the coefficients $a,b$ may be  degenerate, i.e., $a(\pm 1)=b(\pm 1)=0.$    
As highlighted in \eqref{mathL0}–\eqref{mathab}, our Birkhoff preconditioners for more general problems will be constructed by inverting the principal part, $-a(x)\partial_{xx}u-b(y)\partial_{yy}u$, using the stable diagonalisation established in Theorem \ref{WinD2in_eq_D2inTWin}. We therefore start with this motivating case in \eqref{var-coeff}-\eqref{abeta}.



Introduce the diagonal coefficient matrices
\begin{equation}\label{def:coeff_matrices}
\mathbf{C}_{a}
=
\operatorname{diag}(a(x_1),\ldots,a(x_{N-1})),
\quad
\mathbf{C}_{b}
=
\operatorname{diag}(b(y_1),\ldots,b(y_{N-1})),
\end{equation}
and likewise for \(\mathbf{C}_{c}\) and
\(\mathbf{C}_{d}\). 
The Lagrange and Birkhoff collocation matrices in
\eqref{systems:2order_operator} are respectively given by
\begin{subequations}\label{eq:var_coeff_matrices}
\begin{align}
\mathbf{A}_{\rm L}
&=-\mathbf{I}_{N-1}\otimes
(\mathbf{C}_{a}\mathbf{D}^{(2)})
-(\mathbf{C}_{b}\mathbf{D}^{(2)})
\otimes\mathbf{I}_{N-1}+\mathbf{C}_{d}\otimes\mathbf{C}_{c}=
\mathbf K_{\rm L}
+\mathbf{C}_{d}\otimes\mathbf{C}_{c},
\label{AL-var-coeff}\\[4pt]
\mathbf{A}_{\rm B}
&=-\mathbf{B}\otimes\mathbf{C}_{a}
-\mathbf{C}_{b}\otimes\mathbf{B} +(\mathbf{C}_{d}\mathbf{B})
\otimes
(\mathbf{C}_{c}\mathbf{B})=
\mathbf K_{\rm B}
+(\mathbf{C}_{d}\mathbf{B})
\otimes
(\mathbf{C}_{c}\mathbf{B}),
\label{AB-var-coeff}
\end{align}
\end{subequations}
where $\mathbf K_{\rm L}$ and $\mathbf K_{\rm B}$ denote the matrices of the
 corresponding principal parts. 


We define the inverses of
$\mathbf K_{\rm L}$ and $\mathbf K_{\rm B}$ as the right
preconditioners in \eqref{BP-systems}, that is, 
\begin{subequations}\label{defn38}
\begin{align}
\mathbf{P}_{\rm L}
&:= \mathbf K_{\rm L}^{-1} = 
-\big(
\mathbf{I}_{N-1}\otimes
(\mathbf{C}_{a}\mathbf{D}^{(2)})
+
(\mathbf{C}_{b}\mathbf{D}^{(2)})
\otimes\mathbf{I}_{N-1}\big)^{-1},\label{def:PL_PB_var_coeff}\\[4pt]
\mathbf{P}_{\rm B}
&:= \mathbf K_{\rm B}^{-1} =
-\big(
\mathbf{B}\otimes\mathbf{C}_{a}
+
\mathbf{C}_{b}\otimes\mathbf{B}\big)^{-1}.
\end{align}
\end{subequations}
Remarkably, they admit explicit tensor-product representations in terms of the matrices arising from the stable diagonalisation  in Theorem~\ref{WinD2in_eq_D2inTWin}. 
Such representations
can be implemented
without explicitly forming any two-dimensional Kronecker-product
matrices. By reshaping the input vector into an
\((N-1)^2\) array, the Kronecker-product transformations are
reduced to one-dimensional matrix multiplications, while the diagonal
factors are applied entry-wise. 
More precisely, we define the following matrices as in  \eqref{Definition:2order-Bc}:
\begin{equation}\label{BSnu}
 \mathbf B_\nu
=
\mathbf C_\nu^{-\frac{1}{2}}\mathbf B\mathbf C_\nu^{-\frac{1}{2}},
\quad
\mathbf S_\nu
=
\mathbf W^{\frac{1}{2}}\mathbf B_\nu\mathbf W^{-\frac{1}{2}},\quad \nu=a,b.  
\end{equation}
Then by Theorem~\ref{WinD2in_eq_D2inTWin}, we have the orthogonal diagonalisation:
$\mathbf S_\nu
=
\mathbf Q_\nu\mathbf\Sigma_\nu\mathbf Q_\nu^\intercal,$ where
\begin{equation}\label{QnuOrth}
   \mathbf Q_\nu^\intercal\mathbf Q_\nu
=
\mathbf I_{N-1},
\quad  
\mathbf\Sigma_\nu
=
\operatorname{diag}
(\sigma_{\nu,1},\ldots,\sigma_{\nu,N-1}), \quad  
\sigma_{\nu,k}<0.
\end{equation}
Then we have the following explicit representations of the Birkhoff preconditioners.
\begin{thm}\label{prop:explicit_preconditioners}
The preconditioners given in
\eqref{defn38} admit the explicit representations:
\begin{subequations}\label{defn38-inv}
\begin{align}
\mathbf P_{\rm L}
&=\mathbf K_{\rm L}^{-1}=
-(\mathbf V_b\otimes\mathbf V_a)\,
\mathbf\Lambda_{\rm L}\,
(\mathbf V_b^{-1}\otimes\mathbf V_a^{-1}),
\label{eqn39a}
\\[4pt]
\mathbf P_{\rm B}
&=\mathbf K_{\rm B}^{-1}=
-
(\mathbf C_b^{-1}\otimes\mathbf C_a^{-1})\, 
(\mathbf V_b\otimes\mathbf V_a)\,
\mathbf\Lambda_{\rm B}\,
(\mathbf V_b^{-1}\otimes\mathbf V_a^{-1}), \label{eqn39b}
\end{align}
\end{subequations}
where
\begin{equation}\label{LambdaLB}
    \mathbf\Lambda_{\rm L}
=
\operatorname{diag}
\big(
\tfrac{\sigma_{a,i}\sigma_{b,j}}
{\sigma_{a,i}+\sigma_{b,j}}
\big)_{i,j=1}^{N-1},  \quad 
\mathbf\Lambda_{\rm B}
=
\operatorname{diag}
\big(
\tfrac{1}
{\sigma_{a,i}+\sigma_{b,j}}
\big)_{i,j=1}^{N-1},
\end{equation}
and 
$$
\mathbf V_\nu
=
\mathbf C_\nu^{\frac{1}{2}}\mathbf W^{-\frac{1}{2}}\mathbf Q_\nu,\quad  \mathbf V_\nu^{-1}
=
\mathbf Q_\nu^\intercal\,\mathbf W^{\frac{1}{2}} \,\mathbf C_\nu^{-\frac{1}{2}},\quad \nu=a,b.
$$
\end{thm}
\begin{proof}
 From Theorem~\ref{WinD2in_eq_D2inTWin}, we have  
\begin{equation}\label{similarM}
    \mathbf B\mathbf C_\nu^{-1}
=
\mathbf V_\nu\mathbf\Sigma_\nu\mathbf V_\nu^{-1},
\quad
\mathbf C_\nu\mathbf D^{(2)}
=
\mathbf V_\nu\mathbf\Sigma_\nu^{-1}\mathbf V_\nu^{-1}.
\end{equation}
We first consider $\mathbf P_{\rm L}$. 
Recall that for any conformable matrices $\mathbf{M}_1$, $\mathbf{M}_2$, $\mathbf{N}_1$ and $\mathbf{N}_2$, we have
\begin{equation}\label{Kron-prop_MatMat}
    (\mathbf{M}_1\otimes \mathbf{N}_1)(\mathbf{M}_2\otimes \mathbf{N}_2) = (\mathbf{M}_1 \mathbf{M}_2)\otimes (\mathbf{N}_1\mathbf{N}_2).
\end{equation}
Thus, by \eqref{AL-var-coeff}, 
\[
\begin{aligned}
\mathbf K_{\rm L}
&=
-\mathbf I_{N-1}\otimes
\bigl(\mathbf V_a\mathbf\Sigma_a^{-1}\mathbf V_a^{-1}\bigr)
-
\bigl(\mathbf V_b\mathbf\Sigma_b^{-1}\mathbf V_b^{-1}\bigr)
\otimes\mathbf I_{N-1}
\\
&= -(\mathbf V_b\mathbf V_b^{-1}) 
\otimes
\bigl(\mathbf V_a\mathbf\Sigma_a^{-1}\mathbf V_a^{-1}\bigr)-
\bigl(\mathbf V_b\mathbf\Sigma_b^{-1}\mathbf V_b^{-1}\bigr)
\otimes(\mathbf V_a\mathbf V_a^{-1}) \\
&=
-(\mathbf V_b\otimes\mathbf V_a)
\big(
\mathbf\Sigma_b^{-1}\otimes\mathbf I_{N-1}
+
\mathbf I_{N-1}\otimes\mathbf\Sigma_a^{-1}
\big)
(\mathbf V_b^{-1}\otimes\mathbf V_a^{-1}).
\end{aligned}
\]
Hence, using the property:  
\begin{equation}\label{Kron-prop_MatMat2} 
(\mathbf M\otimes \mathbf N)^{-1}
=
\mathbf M^{-1}\otimes \mathbf N^{-1},
\end{equation}
we find readily that
\[
\mathbf P_{\rm L}=\mathbf K_{\rm L}^{-1}
=
-(\mathbf V_b\otimes\mathbf V_a)
\big(
\mathbf\Sigma_b^{-1}\otimes\mathbf I_{N-1}
+
\mathbf I_{N-1}\otimes\mathbf\Sigma_a^{-1}
\big)^{-1}
(\mathbf V_b^{-1}\otimes\mathbf V_a^{-1}).
\]
Since $\mathbf\Sigma_a$ and $\mathbf\Sigma_b$ are diagonal,
\[
\big(
\mathbf\Sigma_b^{-1}\otimes\mathbf I_{N-1}
+
\mathbf I_{N-1}\otimes\mathbf\Sigma_a^{-1}
\big)^{-1}
=
\mathbf\Lambda_{\rm L}.
\]
This gives the stated representation of $\mathbf P_{\rm L}$ in \eqref{eqn39a}.

We next consider $\mathbf P_{\rm B}$. Using
\(
\mathbf B\mathbf C_\nu^{-1}
=
\mathbf V_\nu\mathbf\Sigma_\nu\mathbf V_\nu^{-1},
\)
we obtain
\(
\mathbf B
=
\mathbf V_\nu\mathbf\Sigma_\nu\mathbf V_\nu^{-1}\mathbf C_\nu.
\)
Therefore, by \eqref{Kron-prop_MatMat}, 
\begin{equation}\label{KBLong}
\begin{aligned}
\mathbf K_{\rm B}
&=
-\mathbf B\otimes\mathbf C_a
-\mathbf C_b\otimes\mathbf B
=
-(\mathbf V_b\mathbf\Sigma_b\mathbf V_b^{-1}\mathbf C_b)
 \otimes\mathbf C_a
-\mathbf C_b\otimes
(\mathbf V_a\mathbf\Sigma_a\mathbf V_a^{-1}\mathbf C_a)
\\
&=
-(\mathbf V_b\otimes\mathbf V_a)
(
\mathbf\Sigma_b\otimes\mathbf I_{N-1}
+
\mathbf I_{N-1}\otimes\mathbf\Sigma_a
)
(
\mathbf V_b^{-1}\mathbf C_b
\otimes
\mathbf V_a^{-1}\mathbf C_a
)
\\
&=
-(\mathbf V_b\otimes\mathbf V_a)
(
\mathbf\Sigma_b\otimes\mathbf I_{N-1}
+
\mathbf I_{N-1}\otimes\mathbf\Sigma_a
)
(\mathbf V_b^{-1}\otimes\mathbf V_a^{-1})
(\mathbf C_b\otimes\mathbf C_a).
\end{aligned}
\end{equation}
Consequently, by \eqref{Kron-prop_MatMat2},
\begin{equation}\label{PBLong}
\mathbf P_{\rm B}
=
-(\mathbf C_b^{-1}\otimes\mathbf C_a^{-1})
(\mathbf V_b\otimes\mathbf V_a)
(
\mathbf\Sigma_b\otimes\mathbf I_{N-1}
+
\mathbf I_{N-1}\otimes\mathbf\Sigma_a
)^{-1}
(\mathbf V_b^{-1}\otimes\mathbf V_a^{-1}).
\end{equation}
Since
\(
(
\mathbf\Sigma_b\otimes\mathbf I_{N-1}
+
\mathbf I_{N-1}\otimes\mathbf\Sigma_a
)^{-1}
=
\mathbf\Lambda_{\rm B},
\)
the stated representation of $\mathbf P_{\rm B}$ follows.
\end{proof}
\begin{rmk}\label{DirectInv_constant} Following the lines in the proof of Theorem \ref{prop:explicit_preconditioners}, we can show that the collocation matrices $\mathbf{A}_{\rm L}$ and $\mathbf{A}_{\rm B} $ in \eqref{eq:var_coeff_matrices} for the simple operator: 
$\mathcal{L}[u] = -\Delta u + \gamma u$ with constant $\gamma,$  can be explicitly inverted through the diagonalisation in Theorem \ref{WinD2in_eq_D2inTWin} with $\mathbf C=\mathbf I_{N-1}$ and corresponding   
 $\mathbf{\Sigma} = \operatorname{diag}(\sigma_1,\ldots,\sigma_{N-1})$ and $\mathbf{V}$ in \eqref{VcVc}.
 Then we have
\[
\mathbf A_{\rm L}^{-1} = (\mathbf V\otimes\mathbf V) \mathbf{\Lambda}_{\rm L} (\mathbf V^{-1}\otimes\mathbf V^{-1}),
\quad
\mathbf A_{\rm B}^{-1} = (\mathbf V\otimes\mathbf V) \mathbf{\Lambda}_{\rm B} (\mathbf V^{-1}\otimes\mathbf V^{-1}).
\]
where 
\[
\mathbf{\Lambda}_{\rm L} = \operatorname{diag}\big( \tfrac{\sigma_i\sigma_j}{\gamma\sigma_i\sigma_j-\sigma_i-\sigma_j}\big)_{i,j=1}^{N-1},
\quad
\mathbf{\Lambda}_{\rm B} =
\operatorname{diag}\big( \tfrac{1}{\gamma\sigma_i\sigma_j-\sigma_i-\sigma_j}\big)_{i,j=1}^{N-1},
\]
if $\gamma\sigma_i\sigma_j\not=\sigma_i+\sigma_j$ for any pair $(i,j).$ Thus, PLCOL and PBCOL schemes reduce to the direct solvers.  
\end{rmk}

\subsection{Elliptic problems with general  variable coefficients}\label{sec:genvar}
 In general, we consider the variable coefficient  problem
\begin{equation}\label{eq:general_elliptic}
{\mathcal L}[u]:=-\nabla\cdot(\mathbf A(x,y)\nabla u)
+\mathbf r(x,y)\cdot\nabla u
+s(x,y)u
=
f(x,y),
\quad (x,y)\in\Omega,
\end{equation}
subject to homogeneous Dirichlet boundary conditions, where we assume 
\(\mathbf A=(a_{ij})_{i,j=1}^{2}\) is symmetric and positive definite (SPD) and $a_{ij}\in C^{1}(\Omega), r_i, s, f\in C(\Omega)$. Define
\[
q_1=r_1-\partial_xa_{11}-\partial_ya_{12},\quad 
q_2=r_2-\partial_xa_{12}-\partial_ya_{22}.
\]
The corresponding Lagrange  and Birkhoff collocation matrices 
in \eqref{systems:2order_operator} are
\begin{subequations}\label{eq:general_matrices}
\begin{align}
\mathbf A_{\rm L}
={}&
-\mathbf C_{11}
(\mathbf I_{N-1}\otimes\mathbf D^{(2)})
-2\mathbf C_{12}
(\mathbf D\otimes\mathbf D)
-\mathbf C_{22}
(\mathbf D^{(2)}\otimes\mathbf I_{N-1})
\notag\\
&+
\mathbf Q_1
(\mathbf I_{N-1}\otimes\mathbf D)
+
\mathbf Q_2
(\mathbf D\otimes\mathbf I_{N-1})
+\breve{\mathbf S},
\label{eq:general_AL}\\[4pt]
\mathbf A_{\rm B}
={}&
-\mathbf C_{11}
(\mathbf B\otimes\mathbf I_{N-1})
-2\mathbf C_{12}
(\mathbf D\mathbf B\otimes\mathbf D\mathbf B)
-\mathbf C_{22}
(\mathbf I_{N-1}\otimes\mathbf B)
\notag\\
&+
\mathbf Q_1
(\mathbf B\otimes\mathbf D\mathbf B)
+
\mathbf Q_2
(\mathbf D\mathbf B\otimes\mathbf B)
+
\breve{\mathbf S}
(\mathbf B\otimes\mathbf B),
\label{eq:general_AB}
\end{align}
\end{subequations}
where \(\mathbf C_{ij}\), \(\mathbf Q_i\), and \(\breve{\mathbf S}\) are the
diagonal matrices obtained by evaluating \(a_{ij}\), \(q_i\), and \(s\),
respectively, at the interior tensorial LGL points. Recall that 
$\mathbf D$ is the first-order DM defined in \eqref{Definition: Lagrange_Differentiation_Matrix}.

To  construct the preconditioners, we set 
\begin{equation}\label{eq:general_avg_coeff}
a(x):=\bar a_{11}(x)
=
\frac12\int_{-1}^{1}a_{11}(x,y)\,{\rm d}y,
\quad
b(y):=\bar a_{22}(y)
=
\frac12\int_{-1}^{1}a_{22}(x,y)\,{\rm d}x,
\end{equation}
and define the diagonal matrices  $\mathbf C_a, \mathbf C_b$ as in \eqref{def:coeff_matrices}.
Then the corresponding right preconditioners  \(\mathbf P_{\rm L}\) and \(\mathbf P_{\rm B}\) are given  in \eqref{defn38} but with these coefficient matrices.

\subsection{Eigenvalue analysis}\label{Subsec:eigenanl}

We now carry out spectral analysis of the preconditioned collocation matrices $\mathbf A_{\rm L}\mathbf P_{\rm L}$ and 
$\mathbf A_{\rm B}\mathbf P_{\rm B}$ in \eqref{BP-systems} for
both separable and general variable-coefficient problems. 

For any square matrix \(\mathbf A\), let \(\lambda(\mathbf A)\) denote
its spectrum. It is clear that by \eqref{eq:DcD2_diag_var_coeff},  \eqref{BSnu} and \eqref{similarM}, 
\begin{equation}\label{eigensBDS}
   \lambda(\mathbf B \mathbf C_\nu^{-1})=
   \{\sigma_{\nu,k}\}_{k=1}^{N-1};\quad    \lambda(\mathbf C_\nu \mathbf D^{(2)})=\{\sigma_{\nu,k}^{-1}\}_{k=1}^{N-1},\quad \nu=a,b.
\end{equation}
Moreover, by \eqref{eqn39a}, 
\begin{equation}\label{eigensKLPL}
   \lambda(\mathbf K_{\rm L})=
   \{-(\sigma_{a,i}^{-1}+\sigma_{b,j}^{-1})\}_{i,j=1}^{N-1};\quad    \lambda(\mathbf P_{\rm L})=
   \{-(\sigma_{a,i}^{-1}+\sigma_{b,j}^{-1})^{-1}\}_{i,j=1}^{N-1},
\end{equation}
and by \eqref{KBLong}-\eqref{PBLong},
\begin{equation}\label{eigensKBPB}
\begin{split}
   & \lambda(\mathbf K_{\rm B} (\mathbf C_b\otimes\mathbf C_a)^{-1})=
   \{-(\sigma_{a,i}+\sigma_{b,j})\}_{i,j=1}^{N-1};\quad  \lambda((\mathbf C_b\otimes\mathbf C_a)\mathbf P_{\rm B})=
   \{-(\sigma_{a,i}+\sigma_{b,j})^{-1}\}_{i,j=1}^{N-1}.
   \end{split}
\end{equation}

\begin{rmk}\label{Rmk:spCI}
If $\mathbf C_a=\mathbf C_b=\mathbf  I_{N-1},$ we find from 
    Theorem \ref{WinD2in_eq_D2inTWin} that $\sigma_{a,k}=\sigma_{b,k}=\sigma_k=\lambda_k^{-1}$ with $\{\lambda_k\}$ given by \eqref{eig_bound_D2}. Furthermore, if $cd=\gamma$ is a positive constant in  
    \eqref{var-coeff} {\rm(}i.e., ${\mathcal L}=-\Delta+\gamma${\rm)}, 
    then $\mathbf{A}_{\rm L}, \mathbf{A}_{\rm B}$ in
\eqref{eq:var_coeff_matrices} reduce to  
\begin{equation}\label{ALAB-Lap}
\mathbf{A}_{\rm L}= \mathbf{K}_{\rm L}+\gamma \mathbf{I}_{N-1}\otimes \mathbf{I}_{N-1}, \quad \mathbf{A}_{\rm B}= \mathbf{K}_{\rm B}+\gamma \mathbf{B}\otimes \mathbf{B},
\end{equation}
    whose spectra are respectively  
\begin{equation}\label{4matrices}
\begin{split}
   & \lambda(\mathbf A_{\rm L})=
   \big\{\gamma-(\lambda_{i}+\lambda_{j})\big\}_{i,j=1}^{N-1};\quad  \lambda(\mathbf A_{\rm B})=
   \big\{\tfrac \gamma {\lambda_i\lambda_j}-\tfrac1{\lambda_i}-\tfrac 1 {\lambda_j}\big\}_{i,j=1}^{N-1}.
   \end{split}
\end{equation}
Then we infer from \eqref{eig_bound_D2} that both coefficient matrices of the Lagrange and Birkhoff collocation schemes have the conditioning ${\mathcal O}(N^4).$  However, their eigenvalue distributions are nearly reversed:
the eigenvalues of $\mathbf{A}_{\rm L}$ range from $\mathcal O(1)$
to $\mathcal O(N^4)$, whereas those of $\mathbf{A}_{\rm B}$ range from
$\mathcal O(N^{-4})$ to $\mathcal O(1)$.  \qed 
\end{rmk}

\begin{prop}\label{prop:preconditioned_identity}
For both  separable  and  variable-coefficient
problems, the matrices of PLCOL and PBCOL in \eqref{eq:var_coeff_matrices} and \eqref{eq:general_matrices}  are identical:
\begin{equation}\label{eq:preconditioned_identity}
\mathbf A_{\rm L}\mathbf P_{\rm L}
=
\mathbf A_{\rm B}\mathbf P_{\rm B}.
\end{equation}
\end{prop}
\begin{proof}
Let
\(
\mathbf T=\mathbf B\otimes\mathbf B.
\)
Using
\(
\mathbf D^{(2)}\mathbf B=\mathbf I_{N-1}
\) and \eqref{Kron-prop_MatMat}-\eqref{Kron-prop_MatMat2},
one verifies readily from the corresponding collocation matrices that
\[
\mathbf A_{\rm B}
=
\mathbf A_{\rm L}\mathbf T,
\quad
\mathbf K_{\rm B}
=
\mathbf K_{\rm L}\mathbf T,
\]
where
\(
\mathbf K_{\rm L}=\mathbf P_{\rm L}^{-1}
\)
and
\(
\mathbf K_{\rm B}=\mathbf P_{\rm B}^{-1}.
\)
Therefore,
\(
\mathbf A_{\rm B}\mathbf P_{\rm B}
=
\mathbf A_{\rm L}\mathbf T
(\mathbf K_{\rm L}\mathbf T)^{-1}
=
\mathbf A_{\rm L}\mathbf P_{\rm L}.
\)
\end{proof}
\begin{rmk}\label{rmk:identityA}
    It is noteworthy  that, in exact arithmetic, the identity \eqref{eq:preconditioned_identity} holds, whereas in finite-precision arithmetic, computations involving $\mathbf  B$ and $\mathbf A_{\rm B}\mathbf P_{\rm B}$  are relatively stable when  $N$ is very large.
\end{rmk}

Importantly, we can estimate  eigenvalue distributions  of the  preconditioned collocation matrices. Again, we first consider the separable coefficient case.
\begin{thm}\label{thm:separable_spectral_bound}
Consider the separable variable-coefficient problem
\eqref{var-coeff}. If 
\begin{equation}\label{var-coeff-assumption}
a,b,c,d\in C[-1,1],\quad 
a(x),b(y)>0, \quad 
c(x)d(y)\ge 0,
\end{equation}
then all eigenvalues
\(
\lambda\in
\lambda(\mathbf A_{\rm L}\mathbf P_{\rm L})
=
\lambda(\mathbf A_{\rm B}\mathbf P_{\rm B})
\)
are real and $\lambda\ge 1.$  If in addition
\begin{equation}\label{var-coeff-assumption2}
0<\alpha\leq a(x),\,b(y)\leq\beta,
\quad
0\leq q_{\min}\leq c(x)d(y)\leq q_{\max},
\end{equation}
then we have the bounds
\begin{equation}\label{eq:separable_spectral_bound}
1+
\frac{2\pi^2 q_{\min}}
{\beta c_N N^4}
\leq
\lambda
<
1+
\frac{2q_{\max}}
{\alpha\pi^2},
\end{equation}
where $c_N$ is the constant appearing in the eigenvalue bound
\eqref{eig_bound_D2} for $\mathbf D^{(2)}$.
\end{thm}

\begin{proof}
By Proposition~\ref{prop:preconditioned_identity}, it suffices to consider
\(\mathbf A_{\rm L}\mathbf P_{\rm L}\). Define 
\(
\mathbf C=\mathbf C_d\otimes\mathbf C_c
\)
and
\begin{equation}\label{Hinner}
\mathbf H=(\mathbf W\mathbf C_b^{-1})\otimes
(\mathbf W\mathbf C_a^{-1}),
\quad
(\boldsymbol u,\boldsymbol v)_{\mathbf H}
=\boldsymbol v^\intercal\mathbf H\boldsymbol u,
\quad
\|\boldsymbol u\|_{\mathbf H}
=\sqrt{\boldsymbol u^\intercal\mathbf H\boldsymbol u}.
\end{equation}
By \eqref{var-coeff-assumption} and the positivity of the LGL quadrature
weights, \(\mathbf W\mathbf C_a^{-1}\) and
\(\mathbf W\mathbf C_b^{-1}\) are positive diagonal matrices, and hence
\(\mathbf H\) is positive definite. Let \((\lambda,\boldsymbol v)\) be an eigenpair of
\(\mathbf A_{\rm L}\mathbf P_{\rm L}\), and set
\(\boldsymbol u=\mathbf P_{\rm L}\boldsymbol v\). 
As \(\mathbf P_{\rm L}=\mathbf K_{\rm L}^{-1}\) given in \eqref{def:PL_PB_var_coeff}, one has
\[
\mathbf A_{\rm L}\bm u=\mathbf A_{\rm L}\mathbf P_{\rm L}\bm v=\lambda\bm v=\lambda \mathbf K_{\rm L}\bm u. 
\]
Using \eqref{AL-var-coeff}, one obtains \(\mathbf A_{\rm L}\boldsymbol u=\mathbf K_{\rm L}\boldsymbol u+\mathbf C\boldsymbol u\),
which implies 
\begin{equation}\label{defCu}
(\lambda-1)\mathbf K_{\rm L}\boldsymbol u
=\mathbf C\boldsymbol u, \quad (\lambda-1)(\mathbf H\mathbf K_{\rm L})\boldsymbol u
=
(\mathbf H\mathbf C)\boldsymbol u.
\end{equation}
Then by \eqref{AL-var-coeff} and \eqref{Kron-prop_MatMat}, we have 
\[
\mathbf H\mathbf K_{\rm L}
=
-(\mathbf W\mathbf C_b^{-1})\otimes(\mathbf W\mathbf D^{(2)})
-(\mathbf W\mathbf D^{(2)})\otimes(\mathbf W\mathbf C_a^{-1}).
\]
By Lemma~\ref{WinD2in_eqA},
\(\mathbf W\mathbf D^{(2)}\) is symmetric, and the 
identity \eqref{uWDv} shows that
\(-\mathbf W\mathbf D^{(2)}\) is positive definite.
Using  \eqref{var-coeff-assumption},
\(\mathbf W\mathbf C_a^{-1}\) and
\(\mathbf W\mathbf C_b^{-1}\) are positive diagonal matrices.
Therefore, both Kronecker-product terms above are symmetric positive
definite,   so is \(\mathbf H\mathbf K_{\rm L}\).
Moreover,  as \(\mathbf C\) is diagonal and positive semidefinite, the matrix \(\mathbf H\mathbf C\) is also symmetric
positive semidefinite. Let
\(
\boldsymbol w=(\mathbf H\mathbf K_{\rm L})^{\frac12}\boldsymbol u.
\)
Multiplying the  generalized eigenvalue equation in \eqref{defCu} by
\((\mathbf H\mathbf K_{\rm L})^{-\frac12}\), we obtain
\[
(\lambda-1)\boldsymbol w
=
(\mathbf H\mathbf K_{\rm L})^{-\frac12}
(\mathbf H\mathbf C)
(\mathbf H\mathbf K_{\rm L})^{-\frac12}\boldsymbol w.
\]
Note that the matrix on the right-hand side is symmetric positive semidefinite,
so 
\(\lambda\) is real and \(\lambda\geq1\).

We next derive bounds in \eqref{eq:separable_spectral_bound}.
Using the identity $\mathbf{W}\mathbf{D}^{(2)} =\mathbf{D}^{(2)\intercal}\mathbf{W}$ in  Lemma~\ref{WinD2in_eqA}, one verifies readily  that the matrix 
\(
\widehat{\mathbf S}:=
\mathbf W^{\frac12}\mathbf D^{(2)}\mathbf W^{-\frac12}
\) similar to \(\mathbf D^{(2)}\) 
is symmetric.  Hence,
\(\lambda(\widehat{\mathbf S})=\lambda(\mathbf D^{(2)})\). By
\eqref{eig_bound_D2},  the Rayleigh quotient of
\(-\widehat{\mathbf S}\) satisfies
\begin{equation*}\label{rayquo}
\frac{\pi^2}{4}
<
\frac{
\boldsymbol y^\intercal(-\widehat{\mathbf S})\boldsymbol y
}{
\boldsymbol y^\intercal\boldsymbol y
}
\leq
\frac{c_NN^4}{4\pi^2},\quad \forall\, \bm 0\not= 
\bm y \in \mathbb R^{N-1},
\end{equation*}
so taking
\(\boldsymbol y=\mathbf C^{\frac12}_\nu\boldsymbol z\) yields
\begin{equation}\label{CWDzeigen}
\frac{\pi^2}{4}
<
\frac{
-\boldsymbol z^\intercal
\mathbf C_\nu^{\frac12}
\mathbf W^{\frac12}\mathbf D^{(2)}\mathbf W^{-\frac12}
\mathbf C_\nu^{\frac12}
\boldsymbol z
}{
\boldsymbol z^\intercal\mathbf C_\nu\boldsymbol z
}
\leq
\frac{c_NN^4}{4\pi^2}.
\end{equation}
Since \(\mathbf W\) and \(\mathbf C_\nu\) are diagonal, we have
\begin{equation}\label{CWDidentity}
\mathbf C_\nu^{\frac12}\mathbf W^{\frac12}\mathbf D^{(2)}
\mathbf W^{-\frac12}\mathbf C_\nu^{\frac12}
=
\bigl(\mathbf W^{\frac12}\mathbf C_\nu^{-\frac12}\bigr)
\bigl(\mathbf C_\nu\mathbf D^{(2)}\bigr)
\bigl(\mathbf C_\nu^{\frac12}\mathbf W^{-\frac12}\bigr).
\end{equation}
From \eqref{var-coeff-assumption2}, we find  
\(
\alpha
\boldsymbol z^\intercal\boldsymbol z
\leq
\boldsymbol z^\intercal\mathbf C_\nu\boldsymbol z
\leq
\beta
\boldsymbol z^\intercal\boldsymbol z
\), 
which, together with \eqref{CWDzeigen}-\eqref{CWDidentity}, leads to 
\[
\frac{\alpha\pi^2}{4}
<
\frac{
-\boldsymbol z^\intercal
\bigl(\mathbf W^{\frac12}\mathbf C_\nu^{-\frac12}\bigr)
\bigl(\mathbf C_\nu\mathbf D^{(2)}\bigr)
\bigl(\mathbf C_\nu^{\frac12}\mathbf W^{-\frac12}\bigr)
\boldsymbol z
}{
\boldsymbol z^\intercal\boldsymbol z
}
\leq
\frac{\beta c_NN^4}{4\pi^2}.
\]
As the matrix  $-(\mathbf W^{\frac12}\mathbf C_\nu^{-\frac12})
(\mathbf C_\nu\mathbf D^{(2)})
(\mathbf C_\nu^{\frac12}\mathbf W^{-\frac12})$ is similar to $-\mathbf C_\nu\mathbf D^{(2)},$  we have  from 
\eqref{eigensBDS} that
\[
\lambda(-\mathbf C_\nu\mathbf D^{(2)})=\{-\sigma_{\nu,k}^{-1}\}_{k=1}^{N-1}
\subseteq
(
\tfrac{\alpha\pi^2}{4},
\tfrac{\beta c_NN^4}{4\pi^2}],
\quad \nu=a,b.
\]
Then by \eqref{eigensBDS}-\eqref{eigensKLPL},
\begin{equation}\label{spectralKL}
\lambda(\mathbf K_{\rm L})=\big\{-(\sigma_{a,i}^{-1}+\sigma_{b,j}^{-1})\big\}_{i,j=1}^{N-1} \subseteq
(
\tfrac{\alpha\pi^2}{2},
\tfrac{\beta c_NN^4}{2\pi^2}].
\end{equation}
Since \(\mathbf H\mathbf K_{\rm L}\) is SPD,
the matrix
\(
\widehat{\mathbf K}_{\rm L}
:=\mathbf H^{\frac12}\mathbf K_{\rm L}\mathbf H^{-\frac12}
\)
is SPD and has the same spectrum as
\(\mathbf K_{\rm L}\). Similarly, we define the symmetric positive semidefinite matrix $\widehat{\mathbf C}:=\mathbf H^{\frac12}\mathbf C\mathbf H^{-\frac12}$. Taking
\(\boldsymbol \xi=\mathbf H^{\frac12}\boldsymbol u\), we obtain from 
 \eqref{var-coeff-assumption2} and \eqref{spectralKL}   that
\begin{equation}\label{rayquo2}
\frac{\alpha\pi^2}{2}
<
\frac{
\boldsymbol \xi^\intercal
\widehat{\mathbf K}_{\rm L}
\boldsymbol \xi
}{
\boldsymbol \xi^\intercal\boldsymbol \xi
}
\leq
\frac{\beta c_NN^4}{2\pi^2},\quad 
q_{\min}
\leq
\frac{
\boldsymbol\xi^\intercal
\widehat{\mathbf C}
\boldsymbol\xi
}{
\boldsymbol\xi^\intercal\boldsymbol\xi
}
\leq
q_{\max}.
\end{equation}
By \eqref{defCu}, we have
\begin{equation}\label{rayquo3}
(\lambda-1)\boldsymbol \xi^\intercal
\widehat{\mathbf K}_{\rm L}
\boldsymbol \xi=(\lambda-1)\boldsymbol \xi^\intercal\mathbf H^{\frac12}\mathbf K_{\rm L}\bm u=\boldsymbol \xi^\intercal\mathbf H^{\frac12}\mathbf C\bm u=\boldsymbol\xi^\intercal
\widehat{\mathbf C}
\boldsymbol\xi.
\end{equation}
Combining \eqref{rayquo2}-\eqref{rayquo3} yields
\[
\frac{2\pi^2q_{\min}}{\beta c_NN^4}
\leq
\lambda-1=\frac{
\boldsymbol\xi^\intercal
\widehat{\mathbf C}
\boldsymbol\xi
}{
\boldsymbol\xi^\intercal\boldsymbol\xi
} \frac{
\boldsymbol \xi^\intercal\boldsymbol \xi
}{
\boldsymbol \xi^\intercal
\widehat{\mathbf K}_{\rm L}
\boldsymbol \xi
}
<
\frac{2q_{\max}}{\alpha\pi^2},
\]
which completes the proof of \eqref{eq:separable_spectral_bound}.
\end{proof}

As a direct consequence of Theorem \ref{thm:separable_spectral_bound}, the eigenvalues of the preconditioned matrices are confined in a finite interval for  the following special case, in sharp contrast to the eigenvalue distributions of the unpreconditioned collocation matrices described in Remark \ref{Rmk:spCI}. Indeed, in this setting, we recover the same perfect preconditioning property as in the one-dimensional case \eqref{BD-1D-precond}.
\begin{cor}\label{Laplace-Special} Consider the LCOL and BCOL collocation matrices $\mathbf A_{\rm L}, \mathbf A_{\rm B}$ in \eqref{ALAB-Lap}, and the corresponding Birkhoff preconditioners in \eqref{defn38} for ${\mathcal L}[u]=-\Delta u+\gamma u$ with constant $\gamma\ge 0$, that is, 
\begin{equation}\label{defn3822_1}
\mathbf{P}_{\rm L}
 = 
-(
\mathbf{I}_{N-1}\otimes
\mathbf{D}^{(2)}
+
\mathbf{D}^{(2)}
\otimes\mathbf{I}_{N-1})^{-1},\quad 
\mathbf{P}_{\rm B}
=
-(
\mathbf{B}\otimes\mathbf{I}_{N-1}
+
\mathbf{I}_{N-1}\otimes\mathbf{B})^{-1}.
\end{equation}
 Then all eigenvalues
\(
\lambda\in
\lambda(\mathbf A_{\rm L}\mathbf P_{\rm L})
=
\lambda(\mathbf A_{\rm B}\mathbf P_{\rm B})
\) are real and concentrated around $1:$ 
\begin{equation}\label{eq:separable-Lap-Eigen}
1+
\frac{2\pi^2\gamma}
{c_N N^4}
\leq
\lambda
<
1+
\frac{2\gamma}
{\pi^2},
\end{equation}
where $c_N\approx 1$ for $N\gg 1.$
\end{cor}

The eigenvalue analysis of the preconditioning for the general elliptic problem
\eqref{eq:general_elliptic} is, however, more involved than that for the separable variable-coefficient case. After all, the preconditioners are built upon certain separable forms. For completeness, we feel compelled to present the main result here, with the detailed proof given in Appendix~\ref{app:general_spectral_bound}.
Assume that the variable coefficients in 
\eqref{eq:general_elliptic} satisfy the following conditions: 
\begin{itemize}
    \item[{\rm (i)}] 
    $\mathbf{A}=(a_{ij})$ with $a_{ij}\in C^1(\overline\Omega)$ is 
    uniformly symmetric positive definite on $\overline\Omega,$ and $ \bar a_{11}(x),\bar a_{22}(y)$ in \eqref{eq:general_avg_coeff} are uniformly bounded: $0<\alpha_0
    \leq
    \bar a_{11}(x),\bar a_{22}(y)
    \leq
    \beta_0.$ Moreover, we have 
     \(s\in C(\overline{\Omega})\) and
    \(0<s_{\min}\leq s(x,y)\leq s_{\max}\).
    \item[{\rm (ii)}]  $f\in L^2(\Omega)$ and \(
\mathcal{L}[u]=f,
\) with homogeneous Dirichlet boundary conditions has the 
     $H^2$-solution such that   
\begin{equation}\label{eq:continuous_H2_stability}
\|u\|_{H^2(\Omega)}
\leq
C_A\|f\|_{L^2(\Omega)},
\end{equation}
where the constant  \(C_A>0\), independent of \(f\).
\end{itemize}
In general, we have the following uniform bounds 
for the preconditioned collocation matrices, and refer to Appendix~\ref{app:general_spectral_bound} for the detailed proof. 
\begin{thm}\label{thm:general_spectral_bound} Assume the elliptic problem
\eqref{eq:general_elliptic} satisfy the above conditions {\rm (i)}-{\rm (ii)}. Then
there exist constants \(C_1,C_2>0\), independent of \(N\), such that for every
\(
\lambda
\in
\lambda(\mathbf{A}_{\rm L}\mathbf{P}_{\rm L})
=
\lambda(\mathbf{A}_{\rm B}\mathbf{P}_{\rm B}),
\)
one has
\begin{equation}\label{eq:general_preconditioned_annular_bound}
0<C_1
\leq
|\lambda|
\leq
C_2.
\end{equation}
\end{thm}

\begin{rmk}\label{rmk:DiffC1}
Theorem~\ref{thm:separable_spectral_bound} 
is a special case of 
Theorem~\ref{thm:general_spectral_bound}, 
but in the former setting, we can better characterize the dependence of the bounds on the parameters with a simpler and different proof. 
\end{rmk}



It is noteworthy that eigenvalue clustering and \(N\)-independent bounds do not by themselves guarantee good conditioning for a nonsymmetric matrix, but they can  be indicative of rapid GMRES convergence, particularly when the matrix is nearly normal or possesses a well-conditioned eigenvector basis. Indeed, extensive numerical tests in Subsection~\ref{results} below demonstrate that the GMRES solver converges within only a small number of iterations. 

At this stage, we can only show the well-conditioning in a relatively simpler scenario. Its proof is given in Appendix~\ref{app:uniform_condition_number}.
\begin{thm}\label{thm:uniform_condition_number} Under the setting and conditions of Theorem \ref{thm:separable_spectral_bound}, we further assume  $c(x)d(y)\equiv s$ for some constant $s\ge 0.$   
Then we have
\begin{equation}\label{kapp2A}
\kappa_2(\mathbf A_{\rm L}\mathbf P_{\rm L})
=
\kappa_2(\mathbf A_{\rm B}\mathbf P_{\rm B})
={\mathcal O}(1),
\end{equation}
which holds uniformly for \(N\).
\end{thm}

In Figure~\ref{fig:spectrum_preconditioned}, we present numerical illustrations of the eigenvalue distributions and condition numbers of the original LCOL and BCOL systems and their corresponding preconditioned systems for two sets of variable coefficients. The first row corresponds to the separable problem \eqref{var-coeff} with
\[
a(x)=2-\sin x,\quad
b(y)=2+\cos y,\quad
c(x)=x^2+1,\quad
d(y)=y^2+2,
\]
whereas the second row corresponds to the general elliptic problem \eqref{eq:general_elliptic} with
\[
\mathbf A(x,y)=
\begin{bmatrix}
3+\sin(\pi x)\cos(\pi y) & 0.2x^2y^2\\
0.2x^2y^2 & 2+e^{xy}
\end{bmatrix},
\quad
\mathbf r(x,y)=
\begin{bmatrix}
-\sin(\pi y)\cos(\pi x)\\
\sin(\pi x)\cos(\pi y)
\end{bmatrix},
\]
and \(s(x,y)=1+x^2+y^2\). In both cases, the spectra of the preconditioned matrices remain uniformly bounded away from the origin. For the separable problem, the eigenvalues are real and tightly clustered, whereas for the general problem they may be complex, in agreement with Theorems~\ref{thm:separable_spectral_bound}-\ref{thm:general_spectral_bound}. For the separable problem in the first row, the coefficient assumptions satisfy the conditions of Theorem~\ref{thm:uniform_condition_number}, and the condition numbers of the preconditioned systems remain of order \(\mathcal O(1)\), as predicted by the theorem. For the general variable-coefficient problem in the second row, although the analysis for uniform well-conditioning analytically appears challenging and open, the numerical results likewise exhibit \(\mathcal O(1)\) condition-number behavior.


\begin{figure}[htbp]
    \centering
    \includegraphics[width=0.24\linewidth]{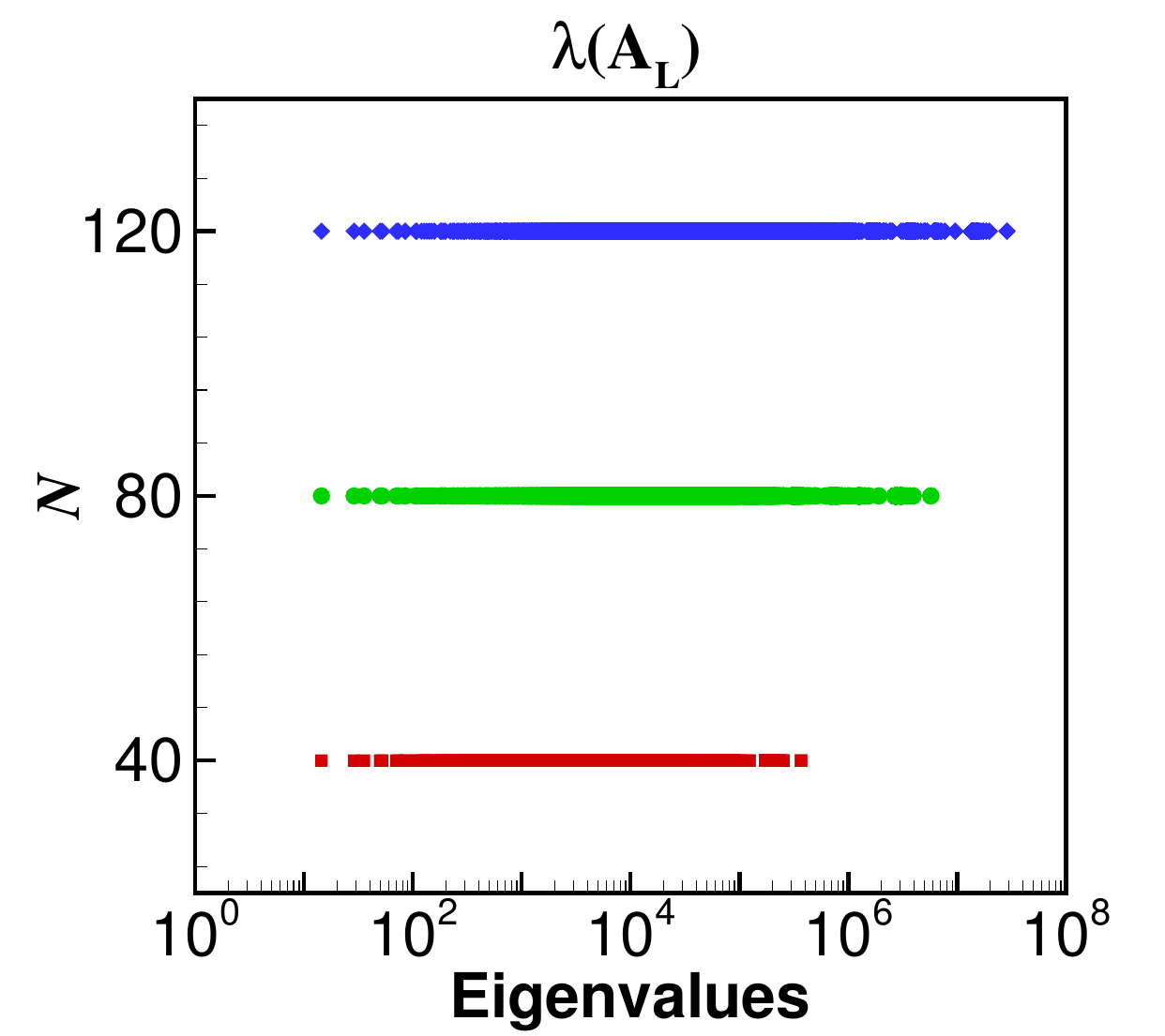} 
    \includegraphics[width=0.24\linewidth]{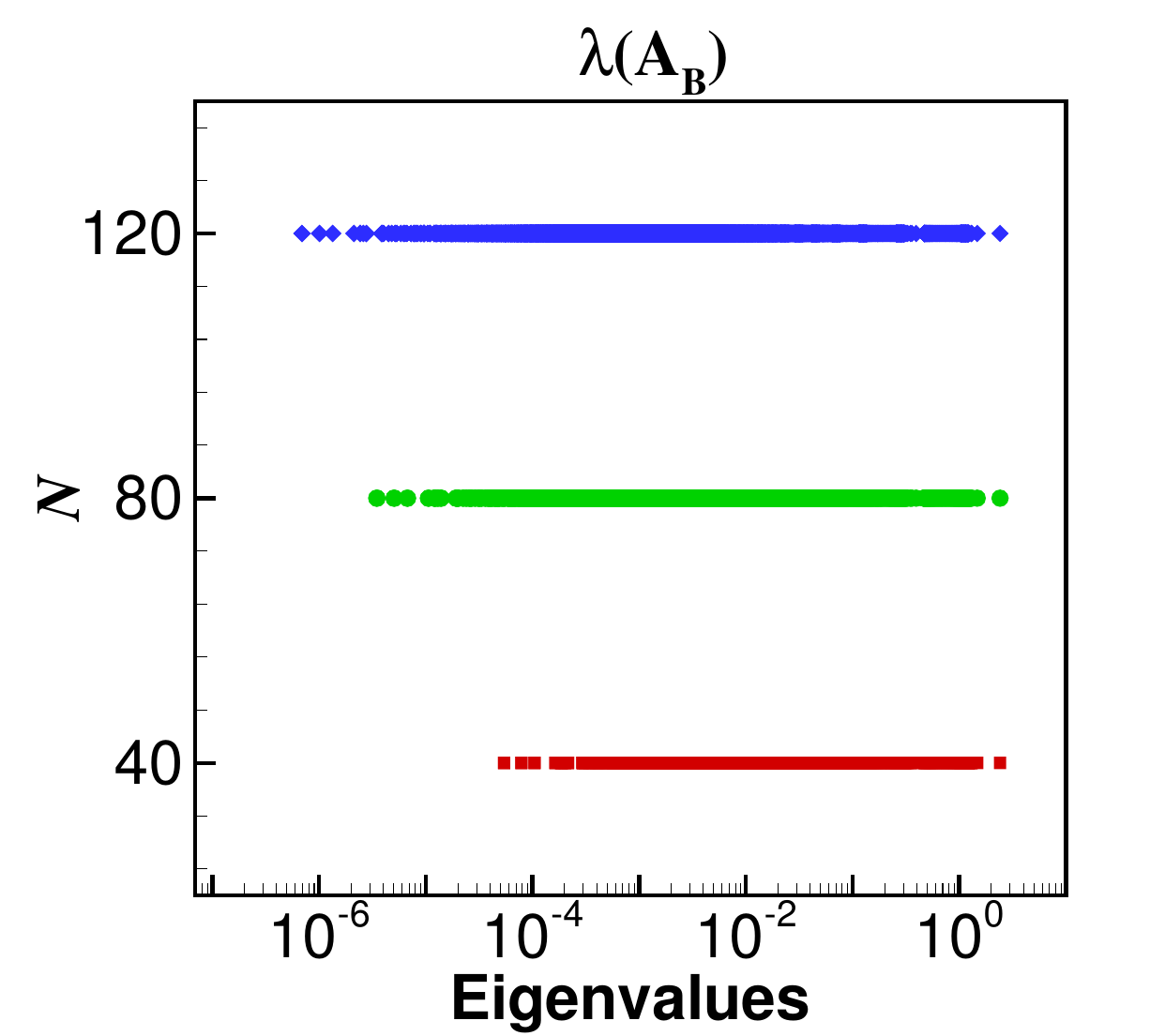} 
    \includegraphics[width=0.24\linewidth]{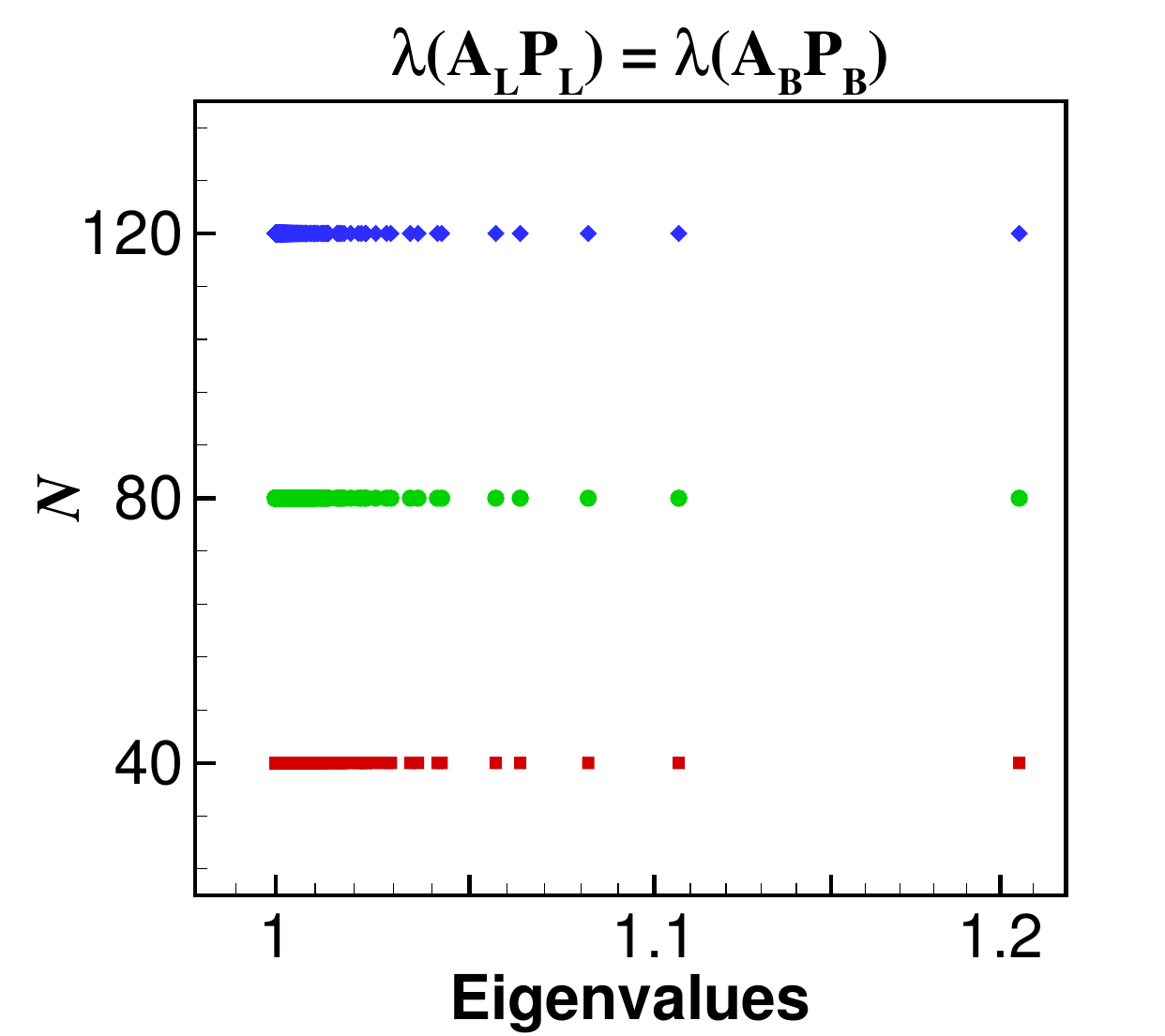}
    \includegraphics[width=0.24\linewidth]{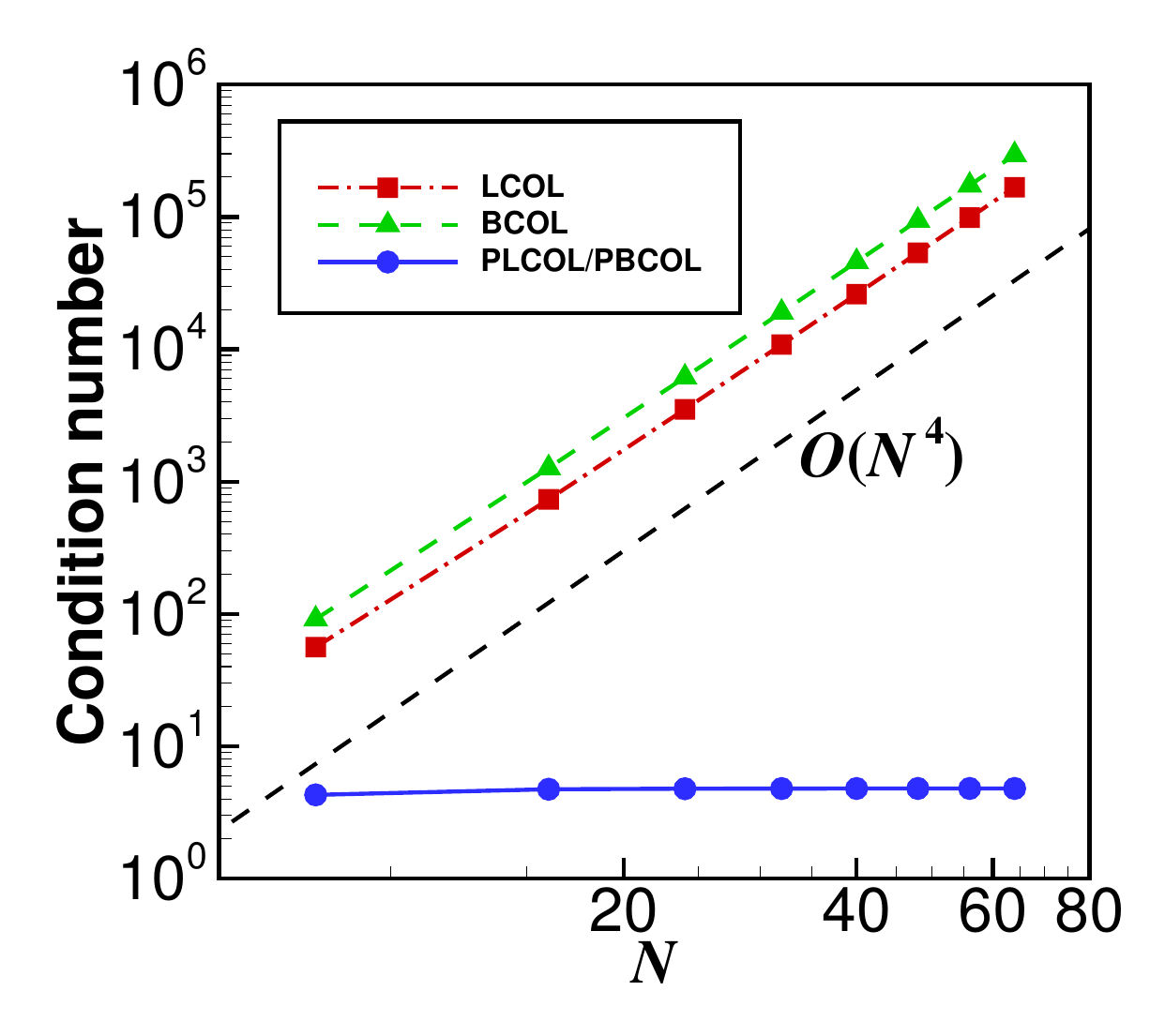}\\
    \includegraphics[width=0.24\linewidth]{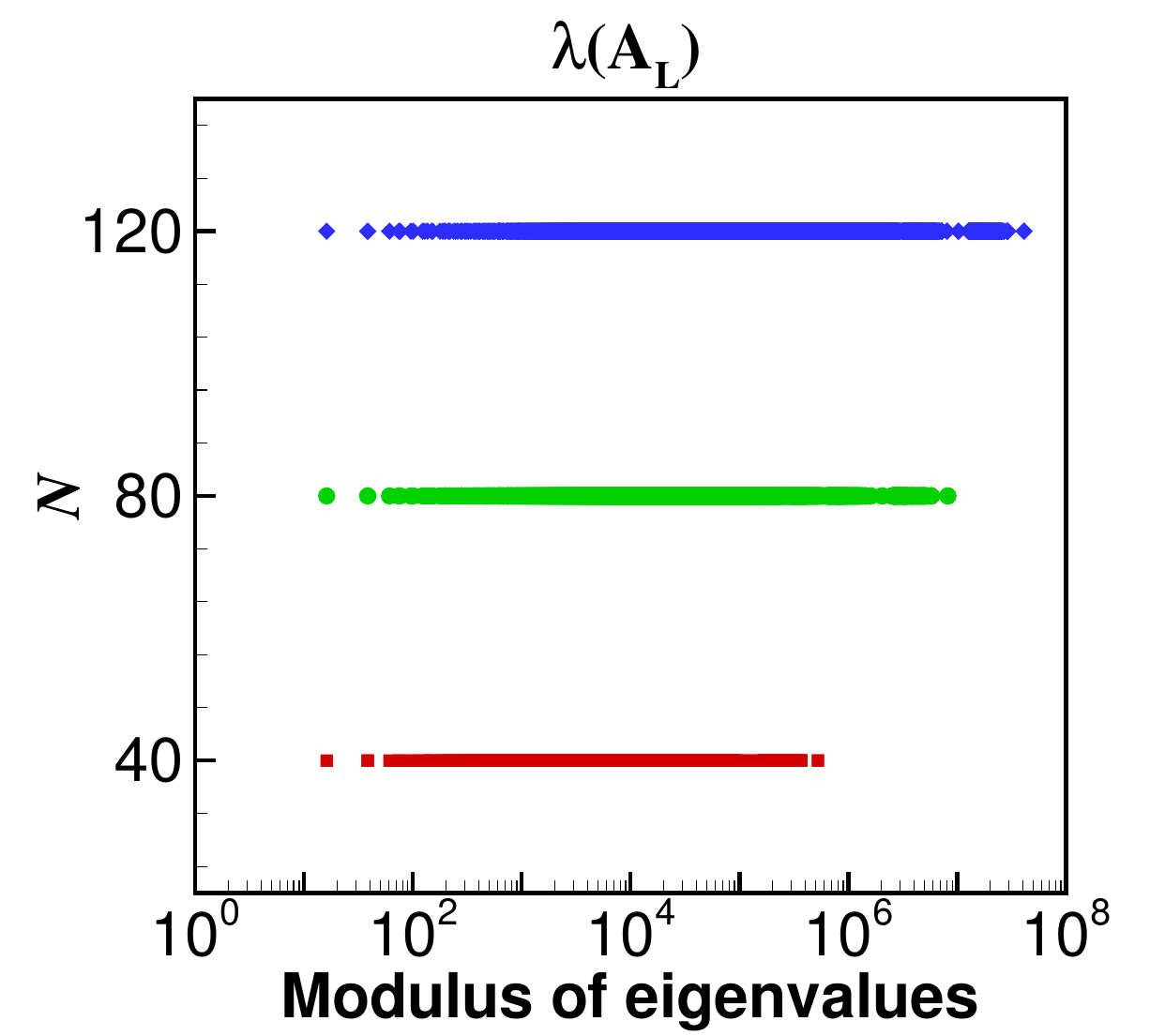} 
    \includegraphics[width=0.24\linewidth]{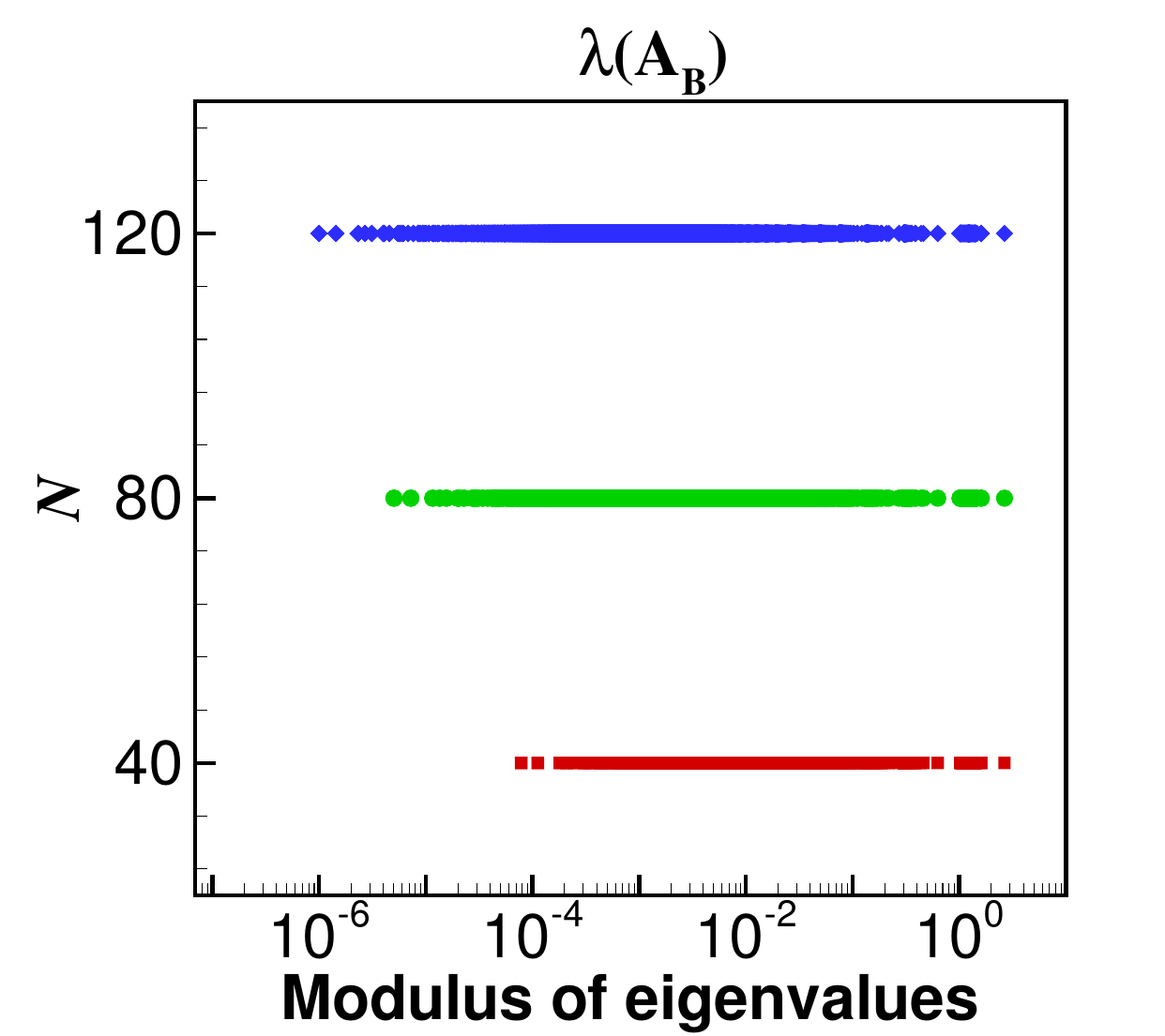} 
    \includegraphics[width=0.24\linewidth]{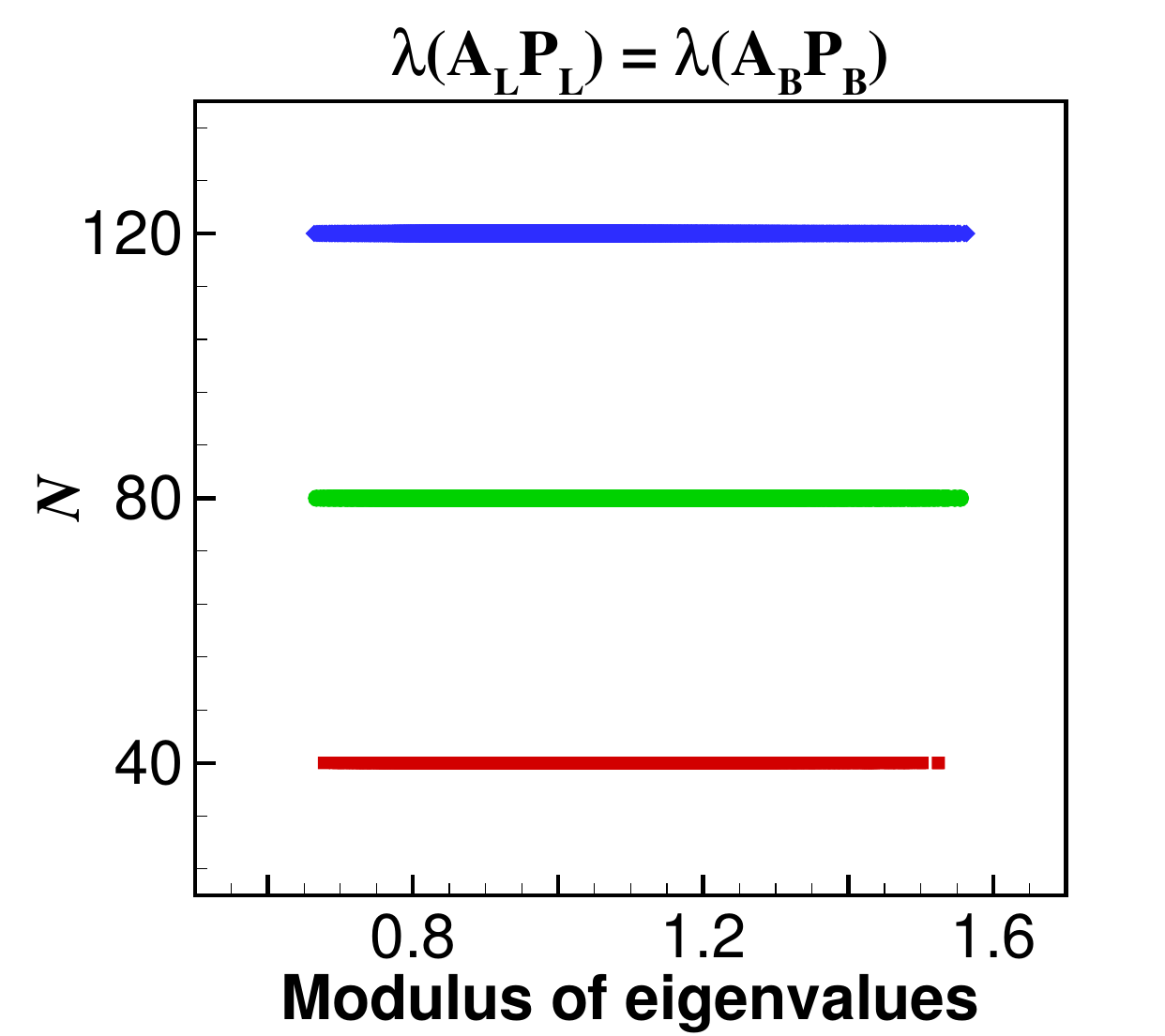}
    \includegraphics[width=0.24\linewidth]{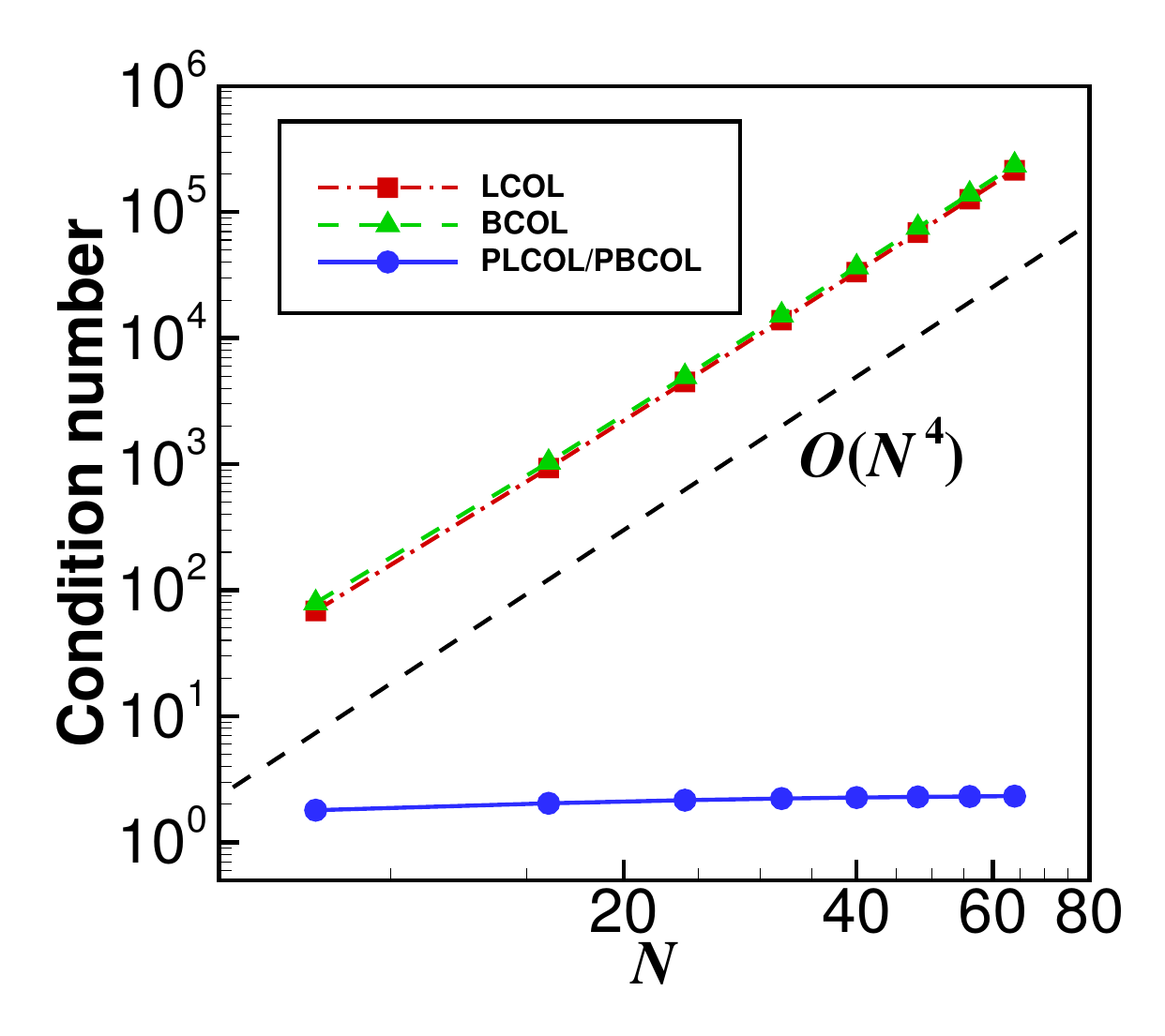}
    \caption{Eigenvalue distributions and condition numbers for the separable (top row) and general variable-coefficient (bottom row) problems. The first three columns show the eigenvalue distributions at \(N=40,80,120\), from left to right: \(\lambda(\mathbf A_{\rm L})\), \(\lambda(\mathbf A_{\rm B})\), and \(\lambda(\mathbf A_{\rm L}\mathbf P_{\rm L})=\lambda(\mathbf A_{\rm B}\mathbf P_{\rm B})\). The fourth column shows the corresponding Euclidean \(2\)-norm condition numbers against  \(N\).}
\label{fig:spectrum_preconditioned}
\end{figure}

\subsection{Numerical results}\label{results}
In what follows, we present numerical experiments to assess the accuracy and efficiency of the  preconditioned collocation schemes through comparison with unpreconditioned counterparts. We employ the GMRES method with a tolerance of $10^{-12}$ to solve all collocation systems. 
The tests cover two-dimensional separable problems with highly oscillatory and high-contrast coefficients, Duffy-transformed problems with geometric singularities, and a nonlinear Allen-Cahn problem, providing a range of tests for the robustness and effectiveness of the proposed preconditioners. All computations are performed in MATLAB on a desktop with an Intel(R) Core(TM) i7-11700F processor at 2.50 GHz and 32 GB RAM. Here, ``Error'' denotes the relative infinity-norm error, while ``/'' indicates failure to converge within 6000 iterations. These computational settings and conventions are used throughout all subsequent numerical experiments unless otherwise stated.

\subsubsection{Highly oscillatory and high contrast variable coefficients}
We consider \eqref{var-coeff} with
$f\equiv 1$ and homogeneous Dirichlet boundary conditions on $\partial \Omega$ with $\Omega=(-1,1)^2,$ and  
 take two sets of coefficients $a,b, c, d:$  (i):
\begin{align}
a_1(x)&=1+10e^{-\cos x}
\int_{-1}^{x}e^{\cos t}\sin(500t^2)\,{\rm d}t, &&b_1(y) = 0.002+0.001\cos(32\pi y),\nonumber\\
c_1(x) &= 5 + 4\cos(32\pi x), && d_1(y) = 5 + 4\sin(32\pi y),\nonumber
\end{align}
with highly oscillatory $a_1(x)$; and (ii): 
\(
a_2(x)=\exp(12x),b_2(y)=c_2(x)=d_2(y)=1,
\)
where the 
contrast between $a_2, b_2$ is
\(e^{24} \approx 2.7\times 10^{10}.\) 

We compare the performance of six collocation schemes: LCOL, BCOL, PLCOL, and PBCOL in \eqref{systems:2order_operator}-\eqref{BP-systems}, where the coefficient matrices ${\mathbf A}_{\rm L},\,{\mathbf A}_{\rm B}$ are defined in \eqref{eq:var_coeff_matrices} and the Birkhoff preconditioners ${\mathbf P}_{\rm L},\,{\mathbf P}_{\rm B}$ in \eqref{defn38}, respectively; the other two schemes are the Laplace-preconditioned variants LPLCOL and LPBCOL. In the latter two schemes, the variable coefficients $a_i,\,b_i$ ($i=1,2$) are not incorporated into the orthogonal diagonalisation in Theorem \ref{WinD2in_eq_D2inTWin}, so 
\begin{equation}\label{defn3822}
\begin{split}
\widetilde{\mathbf P}_{\rm L}
& = 
-\big(
\mathbf{I}_{N-1}\otimes
\mathbf{D}^{(2)}
+
\mathbf{D}^{(2)}
\otimes\mathbf{I}_{N-1}
\big)^{-1}, 
\quad 
\widetilde{\mathbf P}_{\rm B}
 =
-(
\mathbf{B}\otimes\mathbf{I}_{N-1}
+
\mathbf{I}_{N-1}\otimes\mathbf{B}
)^{-1}.
\end{split}
\end{equation}

\begin{figure}[h]
    \centering
    \includegraphics[width=0.26\linewidth]{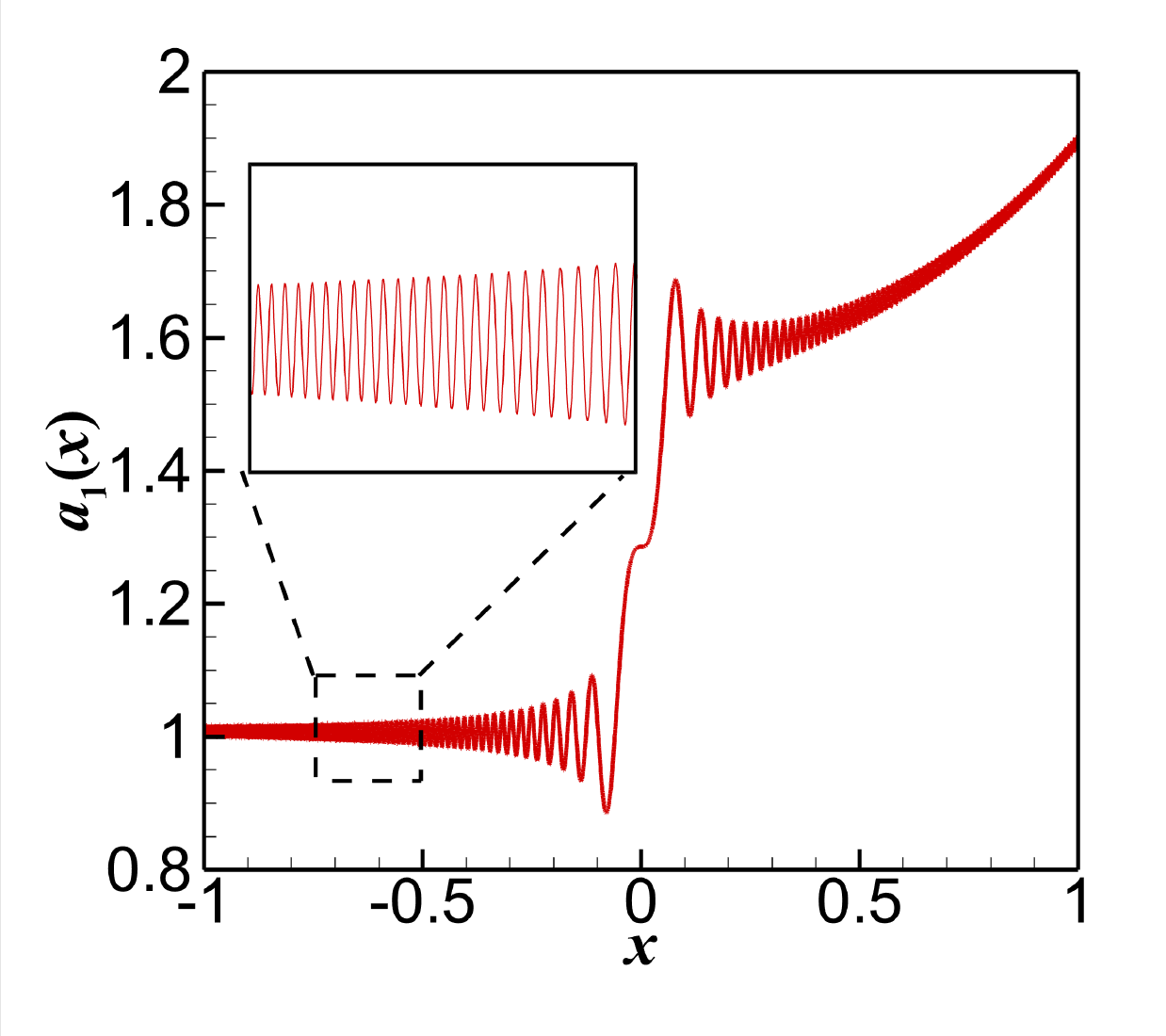} 
    \includegraphics[width=0.26\linewidth]{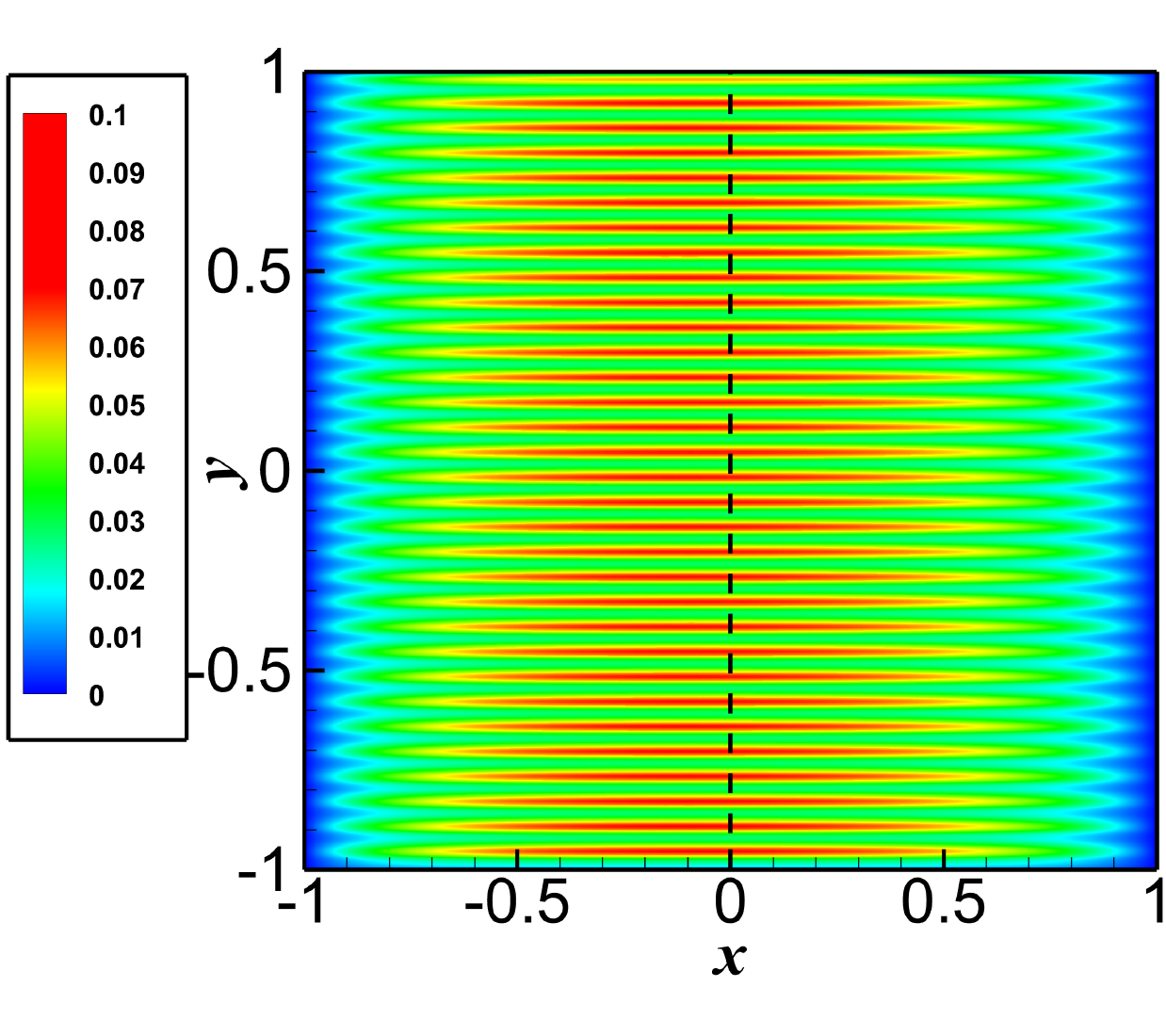} 
    \includegraphics[width=0.26\linewidth]{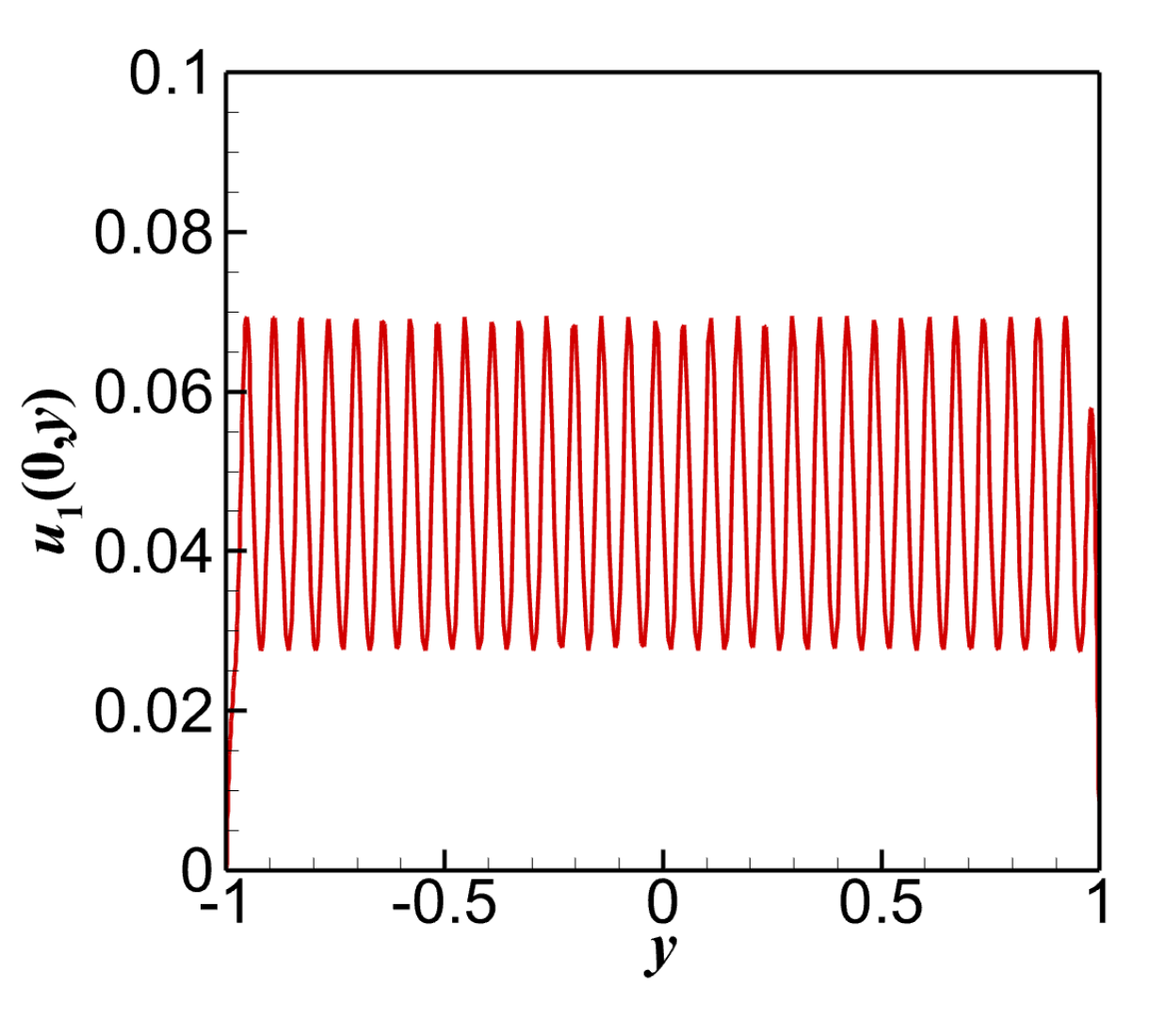} \\
    \includegraphics[width=0.26\linewidth]{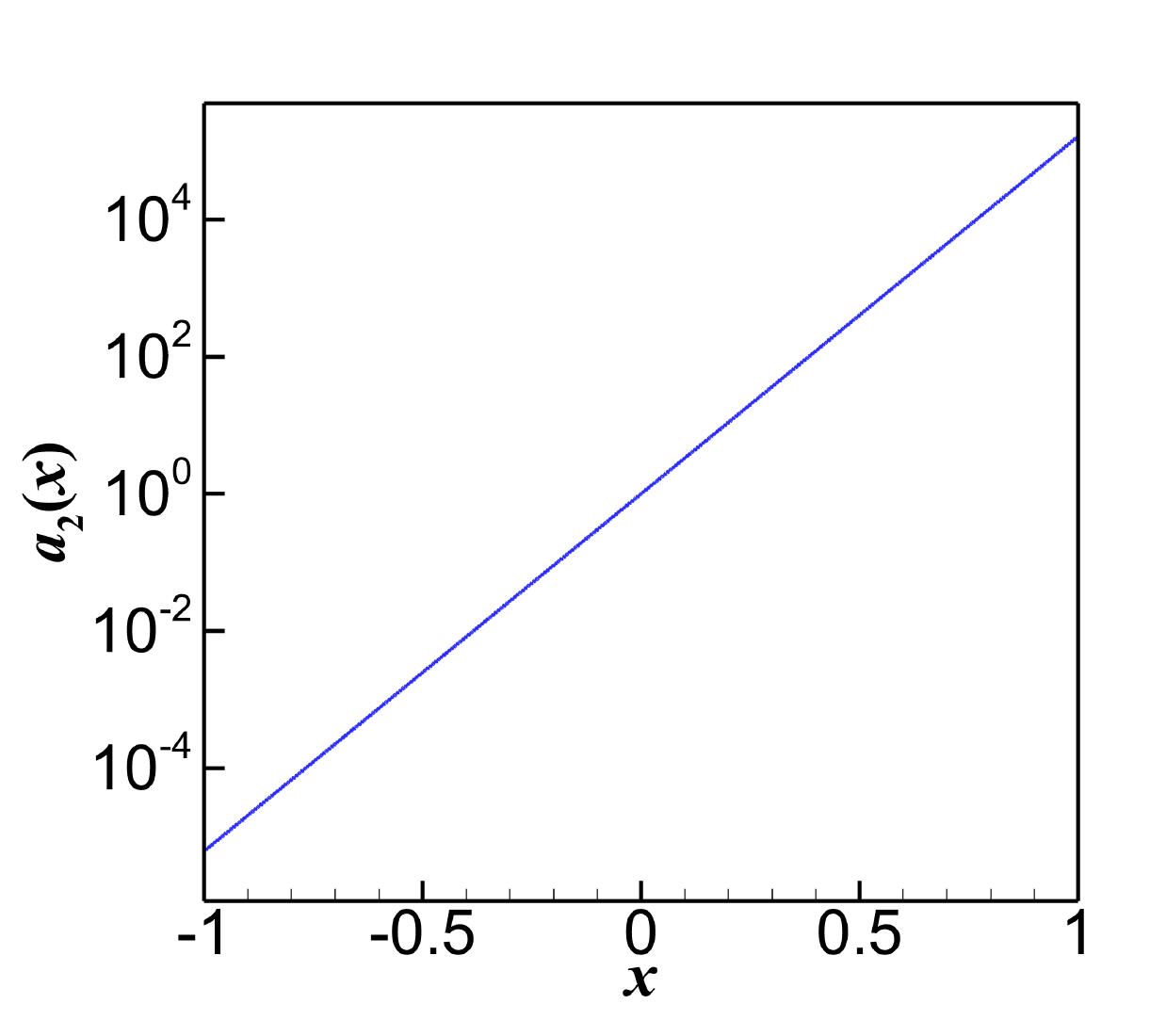} 
    \includegraphics[width=0.26\linewidth]{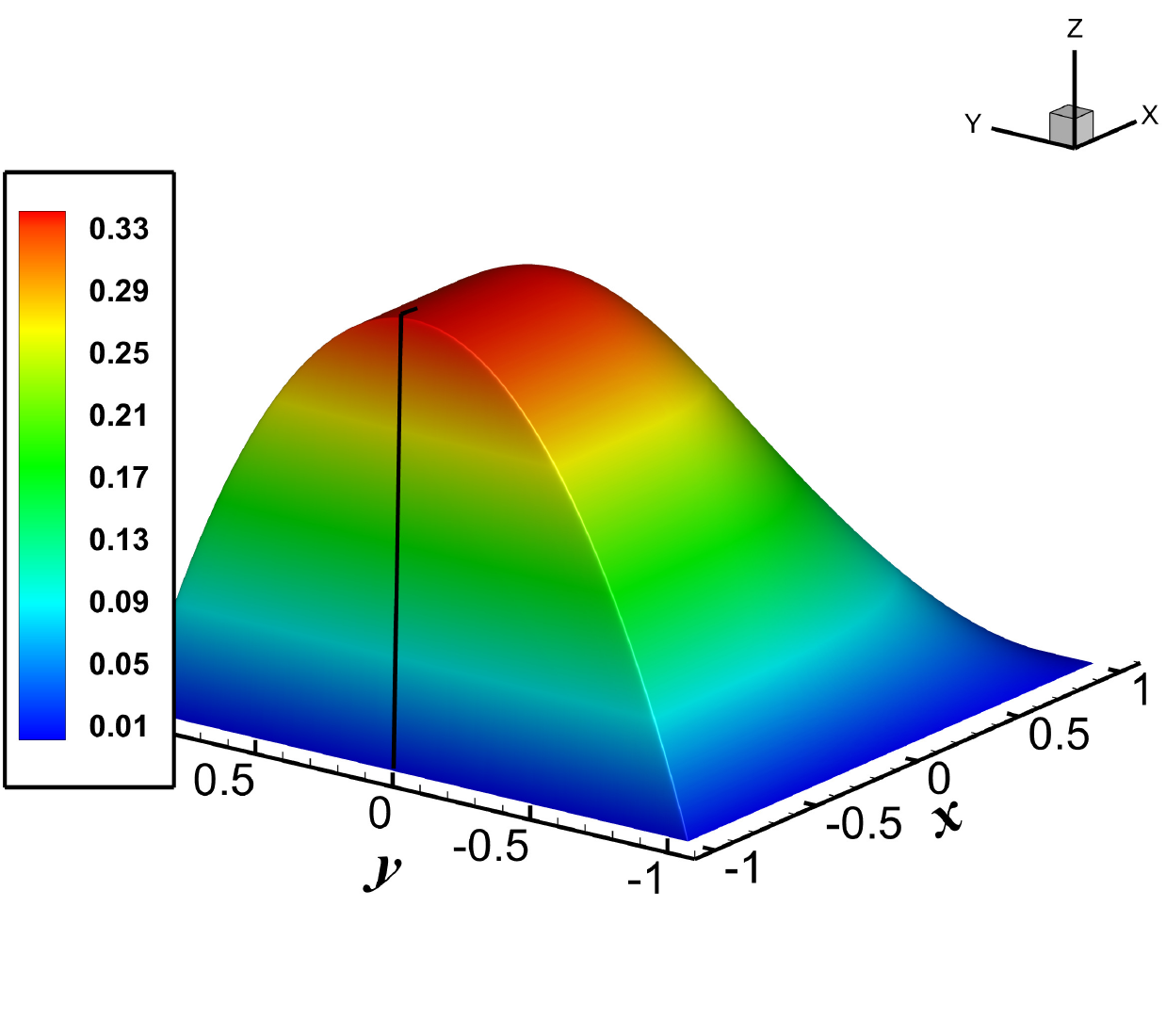} 
    \includegraphics[width=0.26\linewidth]{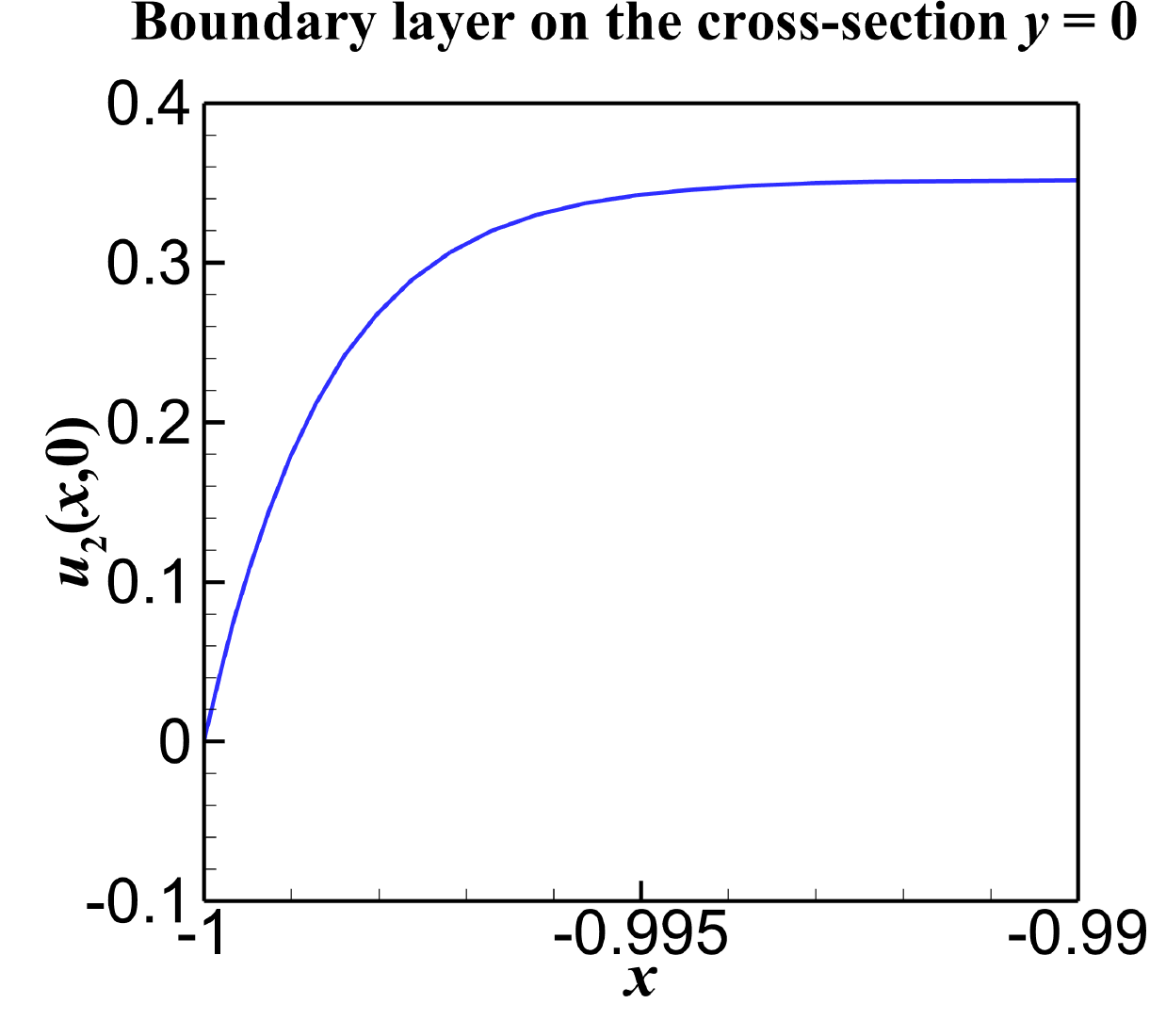} 
 \caption{
Representative features of the two separable test problems.
Top: \(a_1(x)\), \(u_1(x,y)\), and \(u_1(x,0)\).
Bottom: \(a_2(x)\), \(u_2(x,y)\), and \(u_2(x,0)\).
}
    \label{fig:sepproblem}
\end{figure}

As the exact solutions are unavailable, we compute reference solutions using PBCOL, the best-performing scheme, with \(N=4096\). In Figure~\ref{fig:sepproblem}, we display the coefficients \(a_1\) and \(a_2\), together with the corresponding reference solutions. The highly oscillatory coefficients induce pronounced small-scale oscillations in \(u_1\), whereas the large coefficient contrast in \(a_2(x)=\exp(12x)\) gives rise to a sharp boundary layer in \(u_2\) near \(x=-1\). Resolving these fine-scale structures requires stable computations at large values of \(N\).



\begin{figure}[htbp]
    \centering
    \includegraphics[width=0.3\linewidth]{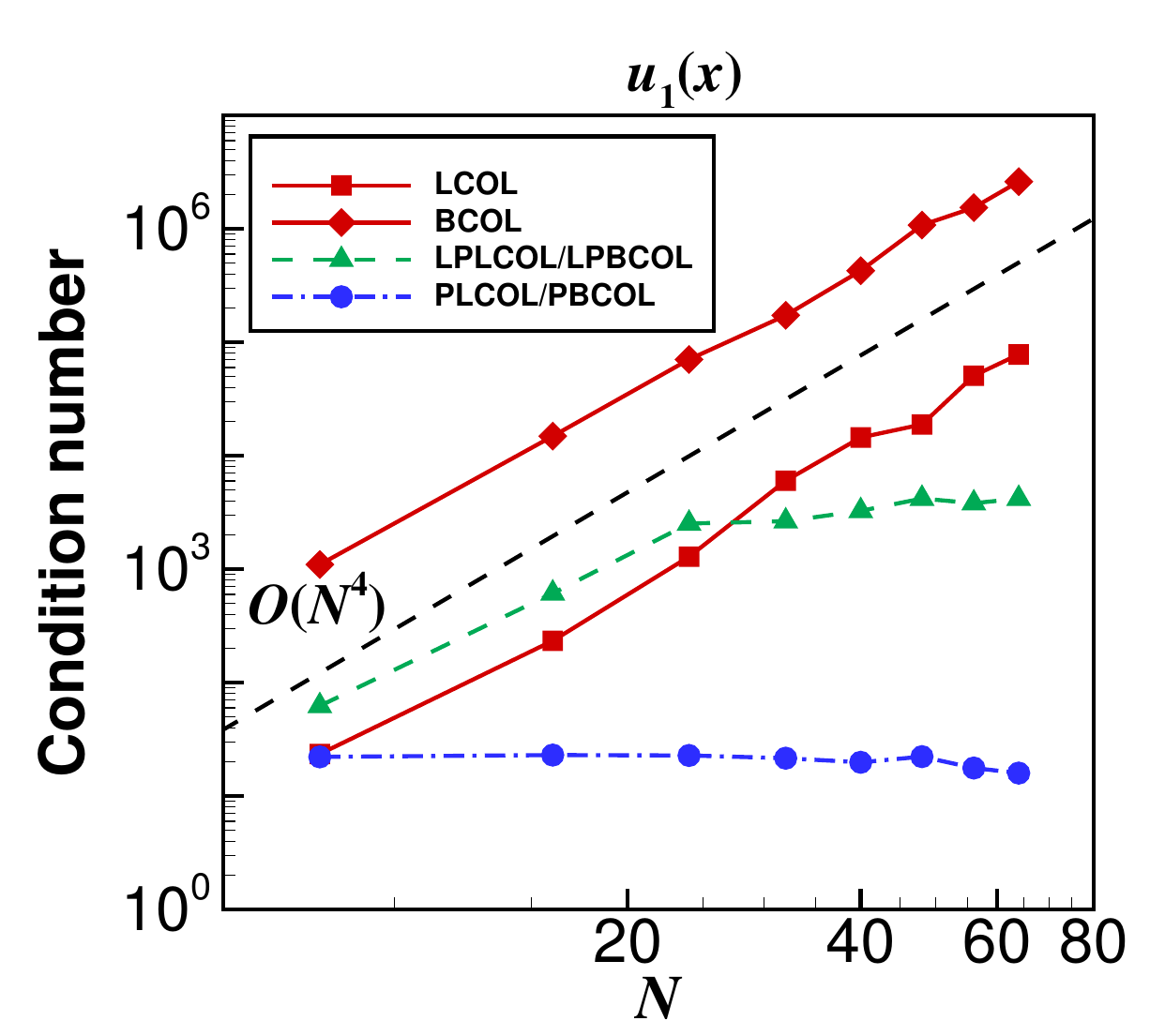} \quad 
    \includegraphics[width=0.3\linewidth]{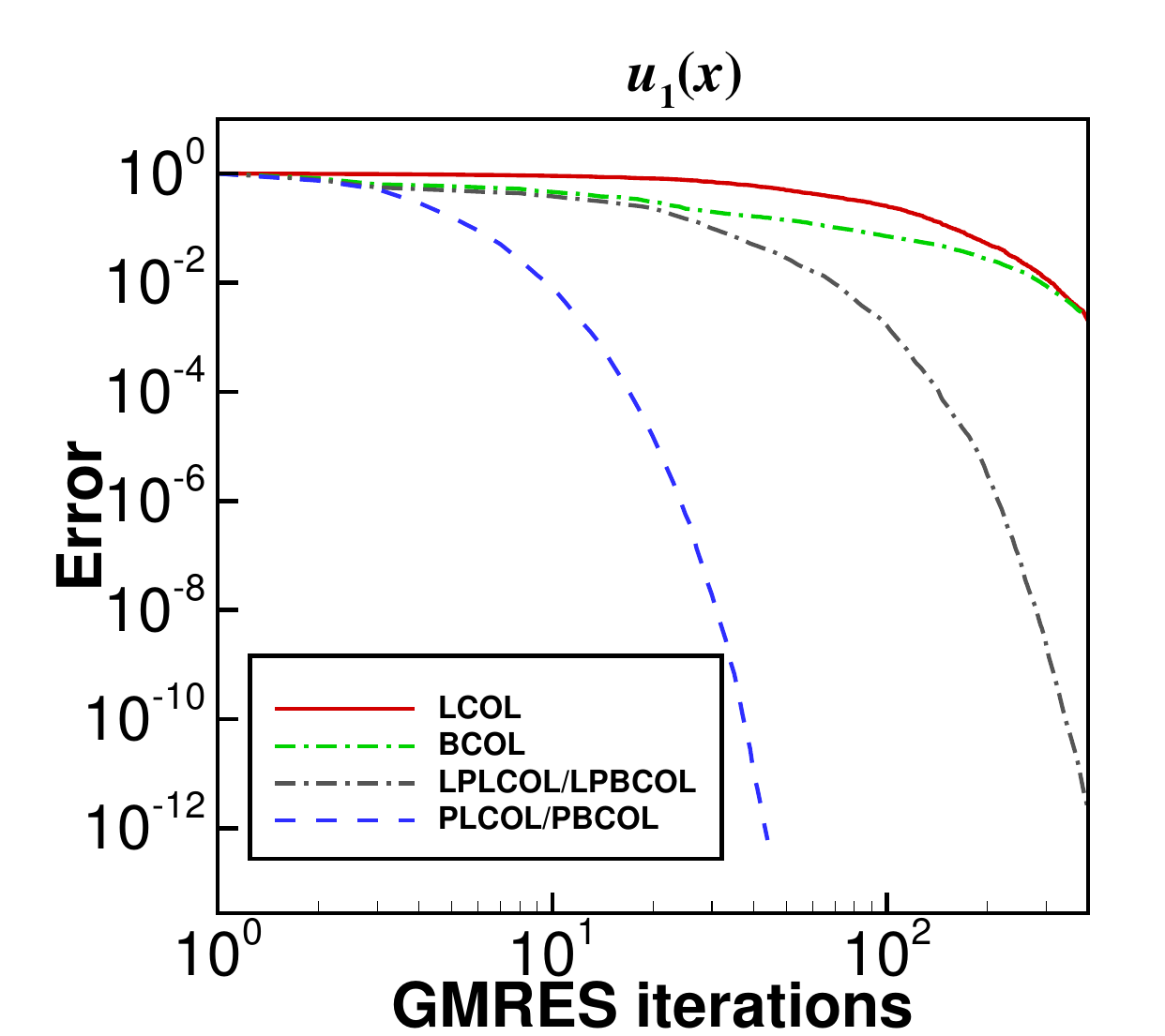}\\[6pt]
    \includegraphics[width=0.3\linewidth]{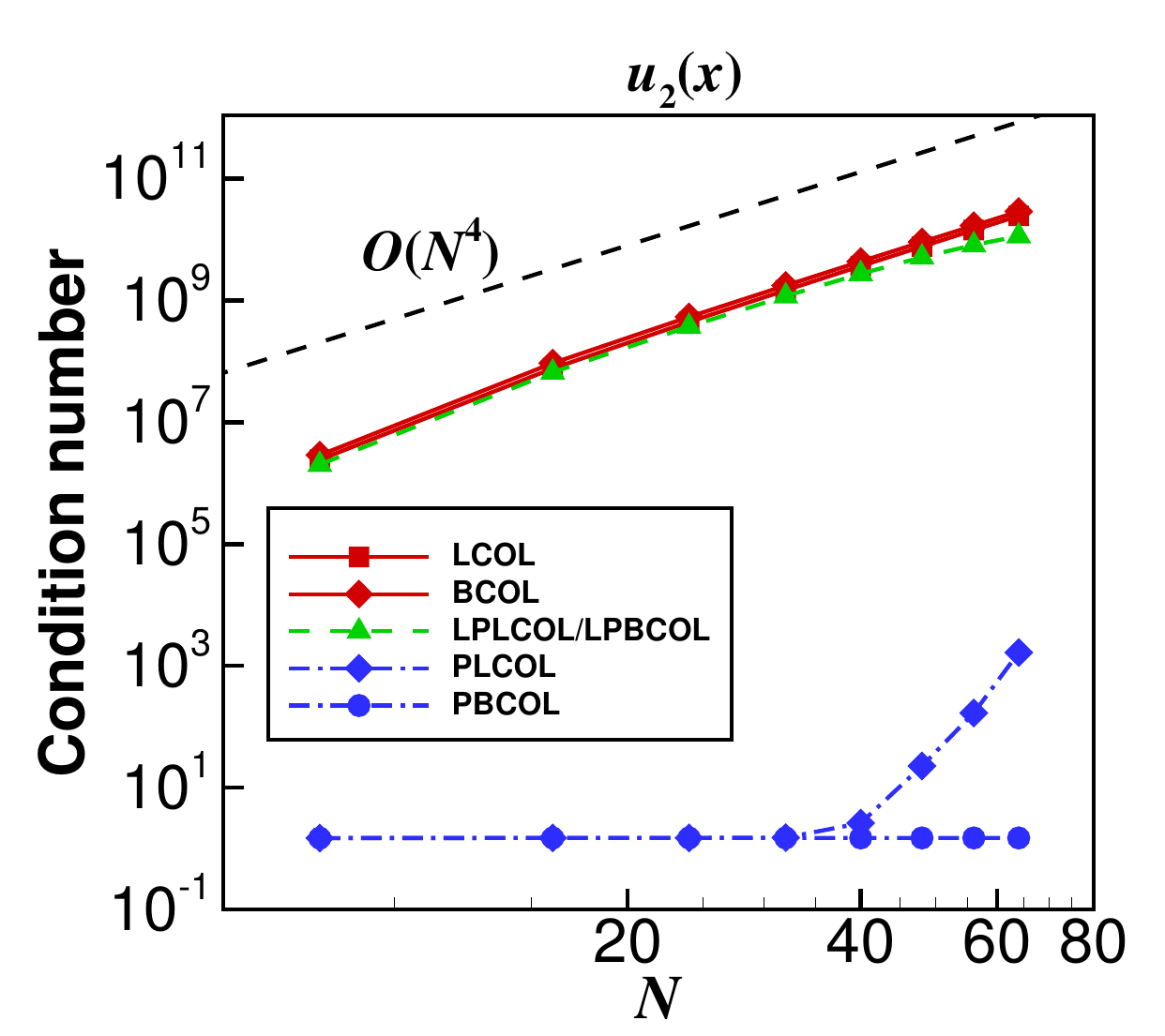} \quad 
    \includegraphics[width=0.3\linewidth]{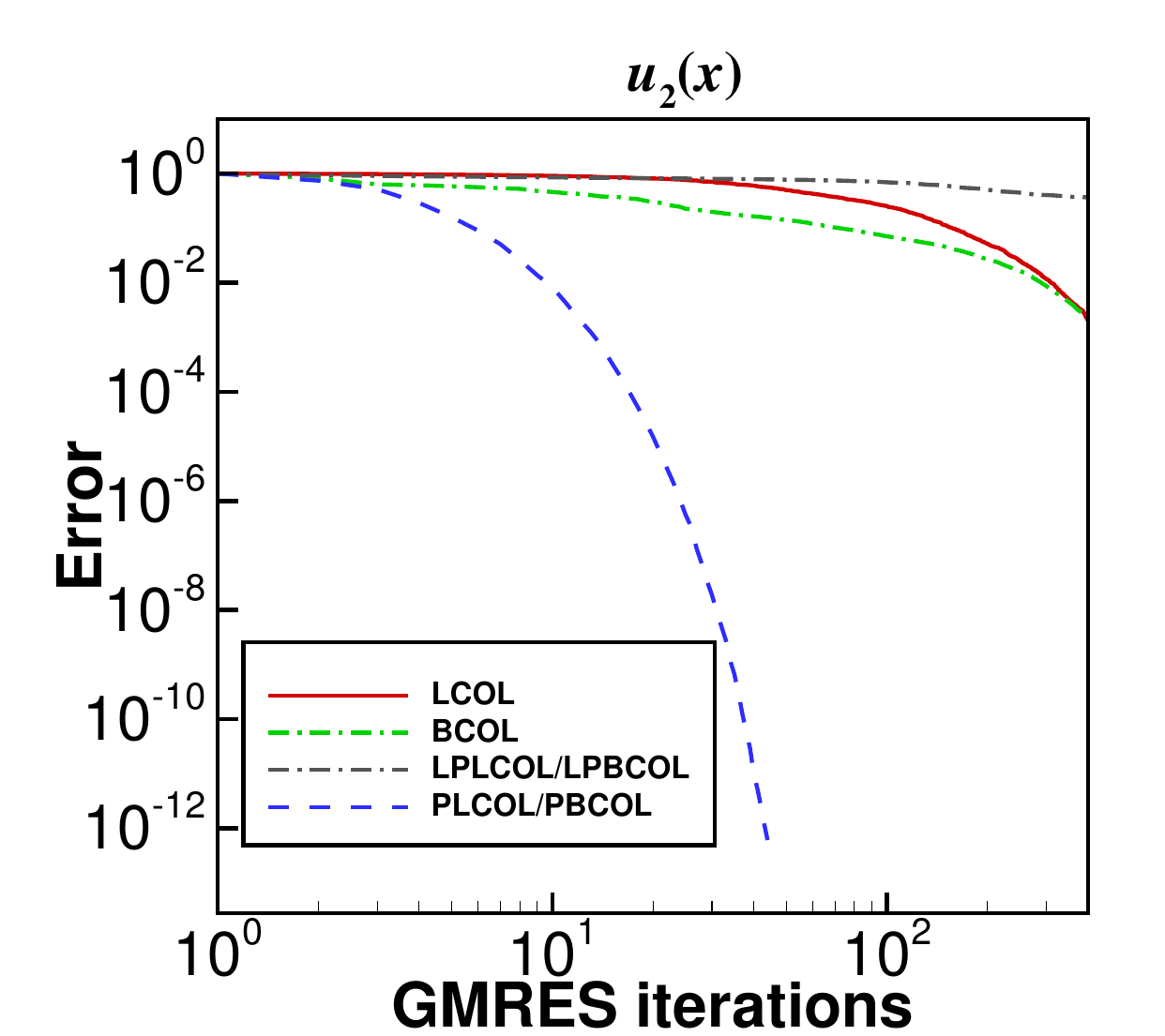}
    \caption{
Condition numbers (left column) and GMRES residual histories (right column)
for the highly oscillatory problem \(u_1\) (top row) and the
high-contrast problem \(u_2\) (bottom row).
}
    \label{fig:sep_cond_res}
\end{figure}

 In Figure \ref{fig:sep_cond_res}, we compare the conditioning and  decay of GMRES residuals against iteration numbers of six schemes for two examples. In both cases,
 the conventional Laplace preconditioners in \eqref{defn3822} fail to improve the conditioning of the resulting systems, which behave like the original unpreconditioned  LCOL and BCOL. 
In contrast, PBCOL and PLCOL exhibit nearly identical condition numbers, consistent with the algebraic equivalence established in Proposition~\ref{prop:preconditioned_identity}. PBCOL remains well conditioned, while PLCOL shows only a mild deterioration for large \(N\). The GMRES residual histories in Figure~\ref{fig:sep_cond_res} (right column) display the same overall trend. LCOL, BCOL, LPLCOL, and LPBCOL converge very slowly, whereas the residuals for PLCOL and PBCOL decay rapidly, with the two methods exhibiting nearly identical convergence behavior,  also illustrate that, for nonsymmetric systems, GMRES convergence is not strongly correlated with the matrix condition number alone; spectral distribution and nonnormality also matter.

In Table~\ref{tab:2d_separable_iter_cpu}, we compare the performance of six schemes in a more quantitative manner, where GMRES is terminated when
the stopping tolerance \(10^{-12}\) is reached or when the number of
iterations exceeds \(6000\). Again, we observe that the original LCOL and
BCOL schemes, as well as the Laplace-preconditioned LPLCOL and LPBCOL
schemes, fail to converge for moderately large \(N\), whereas PLCOL and
PBCOL remain convergent, with PBCOL exhibiting greater stability and the
best overall performance.  Indeed, for the first example, the PBCOL errors exhibit exponential decay once the oscillations are adequately resolved for \(N>64\).

We reiterate that properly incorporating the variable principal coefficients into the diagonalisation used to construct the preconditioners is essential for both good conditioning and rapid iterative convergence. 

\begin{table}[htbp]
\centering 
\caption{\footnotesize Errors, GMRES iteration counts, and CPU time (s) for the
two-dimensional separable problems using the six collocation methods.}
\label{tab:2d_separable_iter_cpu}
\footnotesize
\setlength{\tabcolsep}{4pt}
\begin{tabular}{
cc
cc cc cc cc cc ccc
}
\toprule
&
&
\multicolumn{2}{c}{LCOL}
& \multicolumn{2}{c}{LPLCOL}
& \multicolumn{2}{c}{PLCOL}
& \multicolumn{2}{c}{BCOL}
& \multicolumn{2}{c}{LPBCOL}
& \multicolumn{3}{c}{PBCOL}
\\

\cmidrule(lr){3-4}
\cmidrule(lr){5-6}
\cmidrule(lr){7-8}
\cmidrule(lr){9-10}
\cmidrule(lr){11-12}
\cmidrule(lr){13-15}

& $N$
& Iter & Time
& Iter & Time
& Iter & Time
& Iter & Time
& Iter & Time
& Iter & Time & Error
\\
\midrule

\multirow{8}{*}{$u_1$}
& 8
& 40   & 0.024
& 39   & 0.006
& 28   & 0.026
& 48   & 0.021
& 39   & 0.007
& 28   & 0.021
& 1.07e+00
\\

& 16
& 100  & 0.032
& 114  & 0.012
& 38   & 0.028
& 148  & 0.030
& 114  & 0.012
& 38   & 0.023
& 8.77e$-$01
\\

& 32
& 359  & 0.120
& 304  & 0.079
& 40   & 0.033
& 398  & 0.132
& 305  & 0.071
& 40   & 0.030
& 1.62e+00
\\

& 64
& 1020 & 5.852
& 408  & 0.961
& 43   & 0.078
& 1091 & 6.906
& 408  & 0.997
& 43   & 0.078
& 1.14e+00
\\

& 128
& 2633 & 180.9
& 497  & 6.605
& 44   & 0.253
& 3028 & 250.0
& 497  & 6.659
& 41   & 0.202
& 1.88e$-$01
\\

& 256
& / & /
& 540  & 22.60
& 38   & 0.833
& / & /
& 540  & 23.40
& 39   & 0.773
& 1.52e$-$02
\\

& 512
& / & /
& 549  & 55.04
& 37   & 2.584
& / & /
& 549  & 59.54
& 38   & 2.552
& 6.59e$-$05
\\

& 1024
& / & /
& 551  & 352.0
& 36   & 13.51
& / & /
& 551  & 379.0
& 37   & 13.78
& 1.82e$-$08
\\

\midrule

\multirow{8}{*}{$u_2$}
& 8
& 39   & 0.004
& 37   & 0.010
& 8    & 0.005
& 41   & 0.004
& 37   & 0.006
& 9    & 0.006
& 8.86e$-$01
\\

& 16
& 175  & 0.014
& 164  & 0.023
& 9    & 0.008
& 176  & 0.016
& 164  & 0.018
& 10   & 0.005
& 7.03e$-$01
\\

& 32
& 668  & 0.273
& 681  & 0.288
& 10   & 0.006
& 648  & 0.269
& 690  & 0.346
& 11   & 0.013
& 3.41e$-$01
\\

& 64
& 2517 & 32.65
& 2919 & 41.27
& 10   & 0.027
& 2329 & 25.03
& 2919 & 49.86
& 11   & 0.024
& 2.56e$-$02
\\

& 128
& / & /
& /    & /
& 11   & 0.033
& / & /
& /    & /
& 11   & 0.020
& 1.39e$-$03
\\

& 256
& / & /
& / & /
& 20   & 0.136
& / & /
& / & /
& 11   & 0.068
& 8.69e$-$05
\\

& 512
& / & /
& / & /
& 44   & 1.590
& / & /
& / & /
& 11   & 0.559
& 5.46e$-$06
\\

& 1024
& / & /
& / & /
& 162  & 44.90
& / & /
& / & /
& 11   & 3.904
& 3.40e$-$07
\\
\bottomrule
\end{tabular}
\end{table}

\subsubsection{Elliptic problems on triangles using Duffy transformations}\label{subsub:Duffy2D}
As an important testing case, we  consider the elliptic problem  on 
 a triangle $\mathcal T=\triangle ABC$ with vertices $\{(x_i,y_i)\}_{i=1}^3$: 
\begin{equation}\label{eq:PDEs_physicalDomain}
{\mathcal L}[u]:=-\nabla\cdot\bigl(\mathbf A\nabla u\bigr)
+\mathbf r\cdot\nabla u
+su
=
f
\quad\text{in }\mathcal T;
\quad
u=0
\quad\text{on }\partial\mathcal T,
\end{equation}
where
\(
\mathbf A
\)
is SPD,
\(
\mathbf r=(r_1,r_2)^{\intercal}
\)
is bounded,  and \(s\ge0\) that satisfy the conditions in Subsection \ref{sec:genvar}.  We adapt the Duffy transformation \cite{Duffy1982} to map \(\mathcal T\) onto \(\Omega=(-1,1)^2\), and then apply the proposed preconditioned collocation solver to the transformed problem corresponding to \eqref{eq:PDEs_physicalDomain}. This approach is analogous to the use of tensor-product Legendre polynomial bases on rectangles in \cite{shenwangli2009triangular}, rather than the orthogonal Koornwinder polynomials  employed in \cite{Dubiner1991}.  Such a singular collapsed-coordinate mapping leads to a clustering of collocation points near the singular vertex  (see Figure~\ref{fig:Duffy_transformations}). This substantially worsens the conditioning of the collocation system compared with that on a rectangular domain. We next demonstrate that the proposed preconditioners effectively overcome this difficulty and remain robust even for a highly deformed obtuse triangle with an interior angle of \(160^\circ\) (see Figure~\ref{fig:Duffy_transformations} (left)).

We first present some details on 
the transformed elliptic problems and form of preconditioners. 
The Duffy transformation takes the form:   
\begin{equation}\label{eq:Duffy_transformation}
    \mathbf x(\xi_1, \xi_2)=\begin{bmatrix}
        x(\xi_1, \xi_2)\\[6pt]
        y(\xi_1, \xi_2)
    \end{bmatrix} = \begin{bmatrix}
        x_1 + (x_2-x_1)\frac{(1+\xi_1)(1-\xi_2)}{4} + (x_3-x_1)\frac{1+\xi_2}{2} \\[4pt]
        y_1 + (y_2-y_1)\frac{(1+\xi_1)(1-\xi_2)}{4} + (y_3-y_1)\frac{1+\xi_2}{2}
    \end{bmatrix},
\end{equation}
which collapses the edge $\xi_2=1$ of $\Omega$ to the vertex $C$ of $\mathcal T$. In fact, it is a composition of Duffy transformation: $\Omega\to \mathcal T_{\rm ref}$ (with vertices $(0,0), (1,0), (0,1)$) and an affine mapping: $\mathcal T_{\rm ref}\to \mathcal T,$ 
as shown in Figure~\ref{fig:Duffy_transformations}.
Accordingly, the Jacobian matrix can be factorized as
\[
\mathbf J=\mathbf J_{\rm a}\,\mathbf J_{\rm r}
=\begin{bmatrix}
x_2-x_1 & x_3-x_1\\
y_2-y_1 & y_3-y_1
\end{bmatrix}
\begin{bmatrix}
\tfrac{1-\xi_2}{4} & -\tfrac{1+\xi_1}{4}\\[2pt]
0 & \tfrac12
\end{bmatrix}. 
\]
Thus, one verifies readily that
\begin{equation}\label{JinvA}
\mathbf J^{-1}
=\mathbf J_{\rm r}^{-1}\,\mathbf J_{\rm a}^{-1}=
\frac{1}{\mathbb J_{\rm r}} 
\begin{bmatrix}
\tfrac12 & \tfrac{1+\xi_1}{4}\\[2mm]
0 & \tfrac{1-\xi_2}{4}
\end{bmatrix}\times \frac{1}{\mathbb J_{\rm a}} 
\begin{bmatrix}
y_3-y_1 & x_1-x_3\\
y_1-y_2 & x_2-x_1
\end{bmatrix}, 
\end{equation}
where  the Jacobian  determinants are
\begin{equation}\label{JinvADet}
\mathbb J_{\rm r}= \tfrac {1-\xi_2} 8,\quad \mathbb J_{\rm a}= 2 {|\mathcal T|},
\quad \mathbb J=\mathbb J_{\rm a}\, \mathbb J_{\rm r}=\tfrac{(1-\xi_2) |\mathcal T|}{4},
\end{equation}
with $|\mathcal T|$ being the area of $\mathcal T.$ 

\begin{figure}[htbp]
    \centering
    \includegraphics[width=1.0\linewidth]{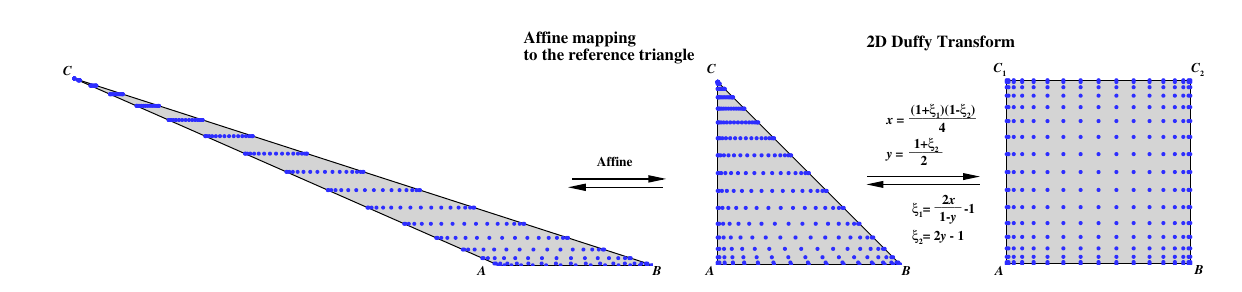}  
    \caption{Affine and Duffy transformations from a general triangular domain
to the reference square. The vertex \(C\) is mapped to the edge
\(C_1C_2\) under the inverse Duffy transformation. The obtuse triangle with $\angle A=160^{\circ}$ is of interest in the test.}
    \label{fig:Duffy_transformations}
\end{figure} 

For clarity, we denote $\hat u(\boldsymbol\xi)
=
u(\mathbf x)$ and likewise for $\hat f$ and $\hat s$. 
Also denote
\(
\widehat{\nabla}
=
(\partial_{\xi_1},\partial_{\xi_2})^{\intercal}.
\)
Then the  problem
\eqref{eq:PDEs_physicalDomain} on $\mathcal T$ is converted to the following problem on $\Omega:$ 
\begin{equation}\label{eq:Transformed_PDEs_referenceDomain}
\hat {\mathcal L}[\hat u]:=-\tfrac{1}{\mathbb J}
\widehat{\nabla}\cdot
\big(
\mathbb J\hat{\mathbf A}
\widehat{\nabla}\hat u
\big)
+
\hat{\mathbf r}\cdot\widehat{\nabla}\hat u
+
\hat s\,\hat u
=
\hat f \quad\text{in }\Omega;
\quad
\hat u=0
\quad\text{on }\partial\Omega,
\end{equation}
where
\(\hat{\mathbf A}
=
\mathbf J^{-1}
\mathbf A\bigl(\mathbf x(\boldsymbol\xi)\bigr)
\mathbf J^{-\intercal}\) and \(
\hat{\mathbf r}
=
\mathbf J^{-1}
\mathbf r\bigl(\mathbf x(\boldsymbol\xi)\bigr).
\)
In view of the singular factors, we multiply
\eqref{eq:Transformed_PDEs_referenceDomain} by \(\mathbb J^2\), and formulate it  in the non-divergence form:
\begin{equation}\label{eq:scaled_transformed_PDE}
{\mathbb J}^2\hat {\mathcal L}[\hat u]=-\sum_{i,j=1}^{2}
\tilde a_{ij}
\partial_{\xi_i\xi_j}\hat u
+
\mathbf b
\cdot\widehat{\nabla}\hat u
+
\mathbb J^2\hat s\,\hat u
=\mathbb J^2 \hat f,
\end{equation}
where
\begin{equation}\label{eq:scaled_coefficients}
\begin{aligned}
&\tilde {\mathbf A}
=\mathbb J^2 \hat  {\mathbf A}=
\mathbb J^2
\mathbf J^{-1}
\mathbf A(\mathbf x(\boldsymbol\xi))
\mathbf J^{-\intercal}:=
(\tilde a_{ij}(\bm \xi))_{i,j=1}^{2},\quad \mathbf b=
\mathbb J^2\hat{\mathbf r}
-
\mathbb J\,
\widehat{\nabla}\cdot
\bigl(\mathbb J\hat{\mathbf A}\bigr),
\end{aligned}
\end{equation}
With the separation  in \eqref{JinvA}-\eqref{JinvADet}, we can rewrite 
$$
\tilde {\mathbf A}=\begin{bmatrix}
\tfrac12 & \tfrac{1+\xi_1}{4}\\[2mm]
0 & \tfrac{1-\xi_2}{4}
\end{bmatrix}  \begin{bmatrix}
\breve a_{11} & \breve a_{12}\\[2mm]
\breve a_{12} & \breve a_{22}
\end{bmatrix} \begin{bmatrix}
\tfrac12 & 0\\[2mm]
 \tfrac{1+\xi_1}{4} & \tfrac{1-\xi_2}{4}
\end{bmatrix}, \quad \breve {\mathbf A}:=
\mathbb J^2_{\rm a}
\mathbf J^{-1}_{\rm a}
\mathbf A
\mathbf J^{-\intercal}_{\rm a},
$$
and direct calculation leads to 
\begin{equation}\label{breveA}
\begin{split}
\tilde a_{11}
&=
\tfrac{1}{16}
\left[
4\breve a_{11}
+4(1+\xi_1)\breve a_{12}
+(1+\xi_1)^2\breve a_{22}
\right],\\[1mm]
\tilde a_{12}
&=
\tfrac{1-\xi_2}{16}
\left[
2\breve a_{12}
+(1+\xi_1)\breve a_{22}
\right],\quad 
\tilde a_{22}
=
\tfrac{(1-\xi_2)^2}{16}\breve a_{22}.
\end{split}
\end{equation}
It is important to note that $\breve {\mathbf A}$ is SPD, which does not involve the degenerate factor $1-\xi_2.$ Moreover, we can verify similarly  that there is no negative power of \(1-\xi_2\) in
\(\mathbb J\widehat{\nabla}\cdot(\mathbb J\hat{\mathbf A})\) and $\mathbf b.$ 

With the above preparations, 
we are now ready to construct  the preconditioners  in \eqref{defn38} for the corresponding Lagrange  and Birkhoff collocation systems arising from \eqref{eq:scaled_transformed_PDE}-\eqref{eq:scaled_coefficients}.
Define the averaged coefficients to be incorporated in the diagonalisation:
\begin{subequations}\label{eq:general_geometry_averages}
\begin{align}
& \bar a_{11}(\xi_1)
:=
\frac1{32}\int_{-1}^{1}
\left[
4\breve a_{11}
+4(1+\xi_1)\breve a_{12}
+(1+\xi_1)^2\breve a_{22}
\right]\,{\rm d}\xi_2,
\\
&\bar a_{22}(\xi_2)
:=
\frac{(1-\xi_2)^2}{32} \int_{-1}^1 \breve a_{22} \,{\rm d}\xi_1,
\end{align}
\end{subequations}
so we derive  the preconditioners \eqref{defn38}
by setting $a=\bar a_{11}$ and 
$b=\bar a_{22}.$

We now compare the collocation schemes for 
solving \eqref{eq:PDEs_physicalDomain} on the triangle $\mathcal T$  with vertices $A(0,0)$, $B(1,0),$ and $C(\cot\omega, 1),$ where $\omega=\angle A$ and the area $|\mathcal T|=1/2.$ 
We consider the problem \eqref{eq:PDEs_physicalDomain} 
 with a smooth source term:
\(
f(x,y)
=
10^4
\exp(-\tfrac{x^2+y^2}{0.05^2}),
\)
and two sets of coefficients:  Case (i):
\(
\mathbf A=\mathbf I_2,
\mathbf r=\mathbf 0,
s=100; 
\)
and Case (ii):
\begin{equation*}
\mathbf A
=
\begin{bmatrix}
3+\sin(\pi x)\cos(\pi y) & 0.2x^2y^2\\[4pt]
0.2x^2y^2 & 2+e^{xy}
\end{bmatrix},
\;\;
\mathbf r
=
\begin{bmatrix}
-\sin(\pi y)\cos(\pi x)\\[4pt]
\sin(\pi x)\cos(\pi y)
\end{bmatrix},
\;\;
s=1+x^2+y^2.
\end{equation*}
The reference solution is computed by the PBCOL method with \(N=4096\).

In the comparison, the collocation schemes:
LCOL$^{\star}$ and BCOL$^{\star}$ refer to the Lagrange and Birkhoff collocation methods for \eqref{eq:Transformed_PDEs_referenceDomain} (without multiplying the Jacobian $\mathbb J^2$), while the other four schemes (i.e., LCOL, BCOL, PLCOL, and PBCOL) are for  \eqref{eq:scaled_transformed_PDE}.

\begin{figure}[!h]
    \centering
\includegraphics[width=0.3\linewidth]{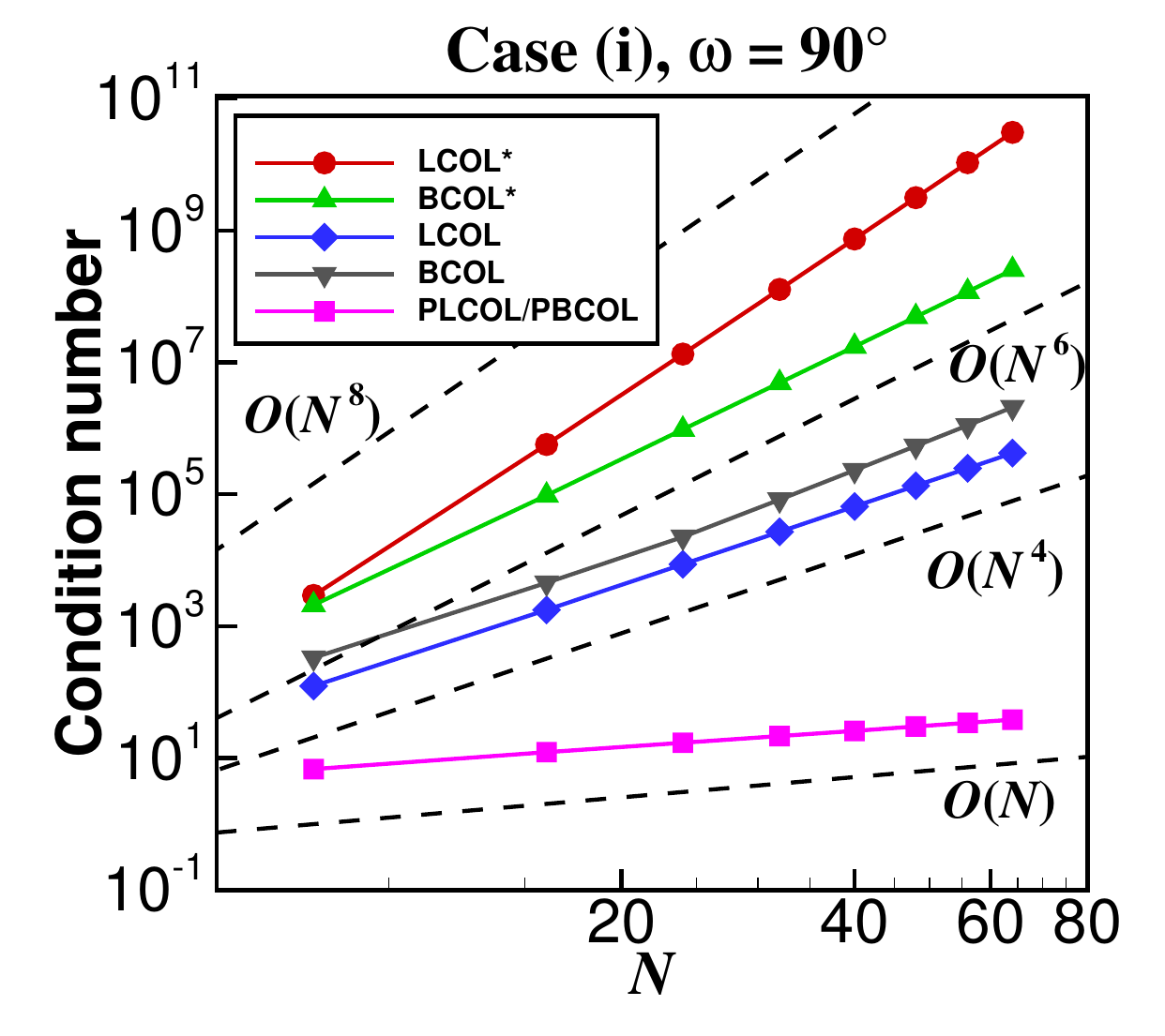} \quad 
\includegraphics[width=0.3\linewidth]{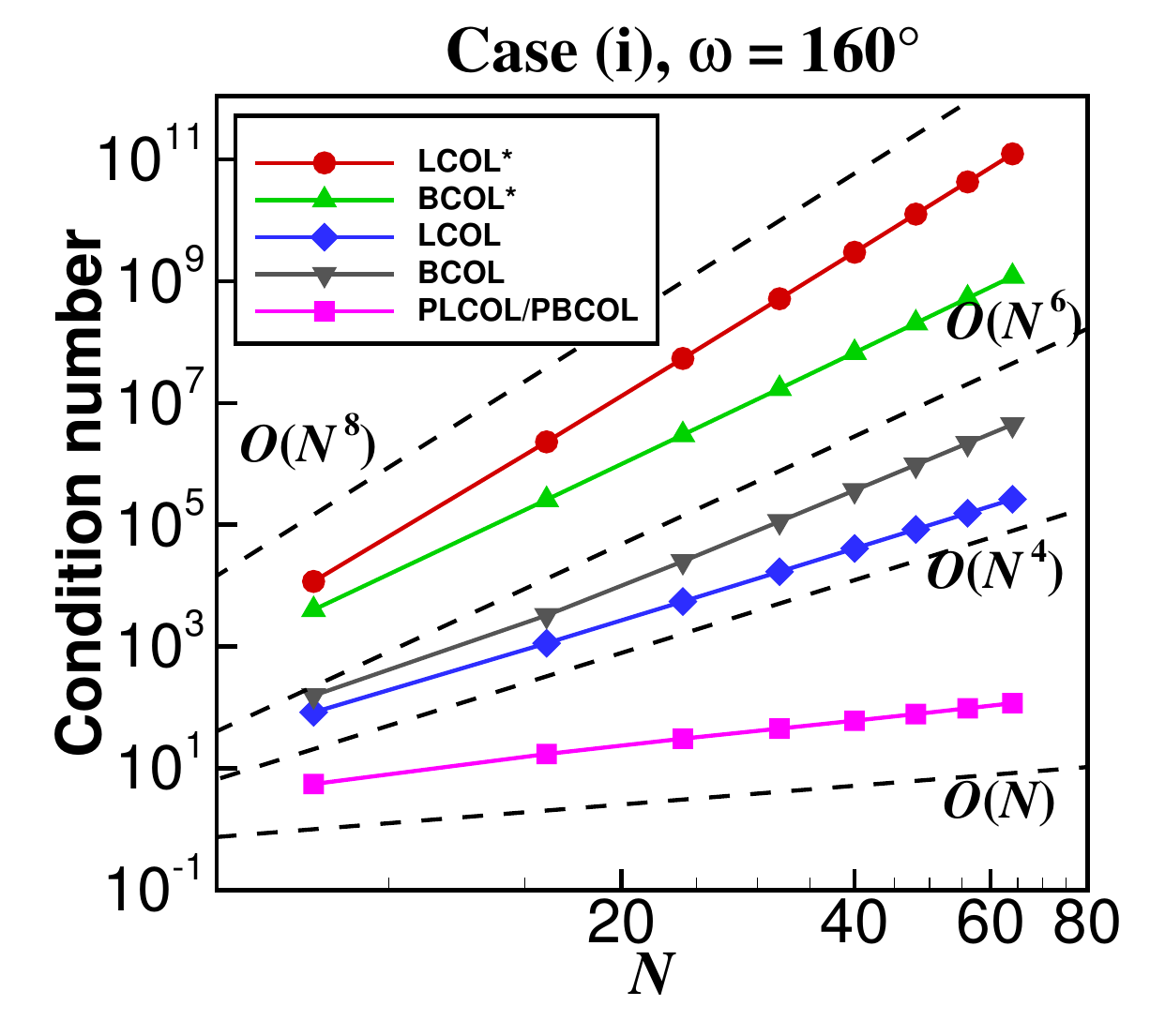}\\[6pt]
\includegraphics[width=0.3\linewidth]{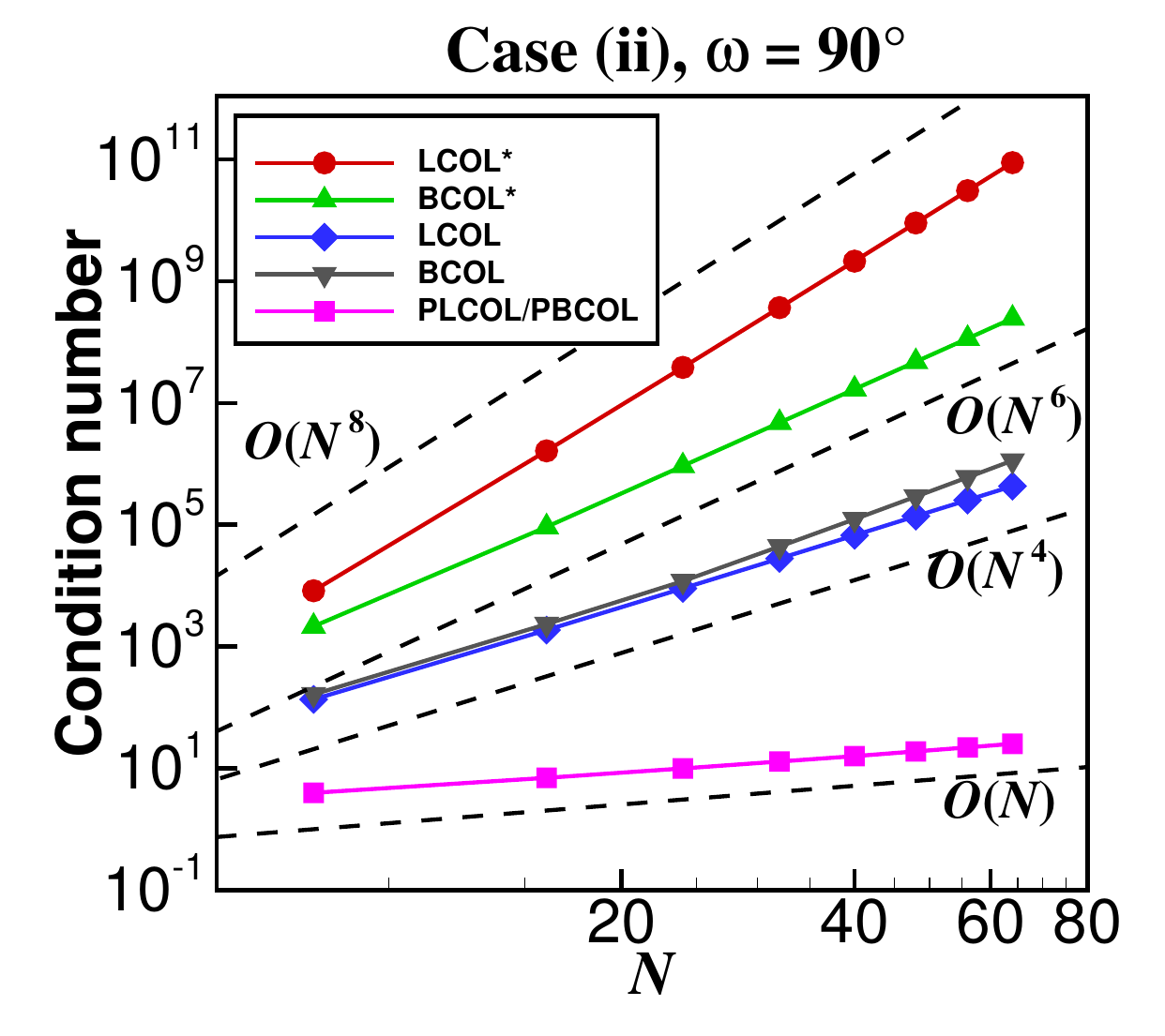} \quad 
\includegraphics[width=0.3\linewidth]{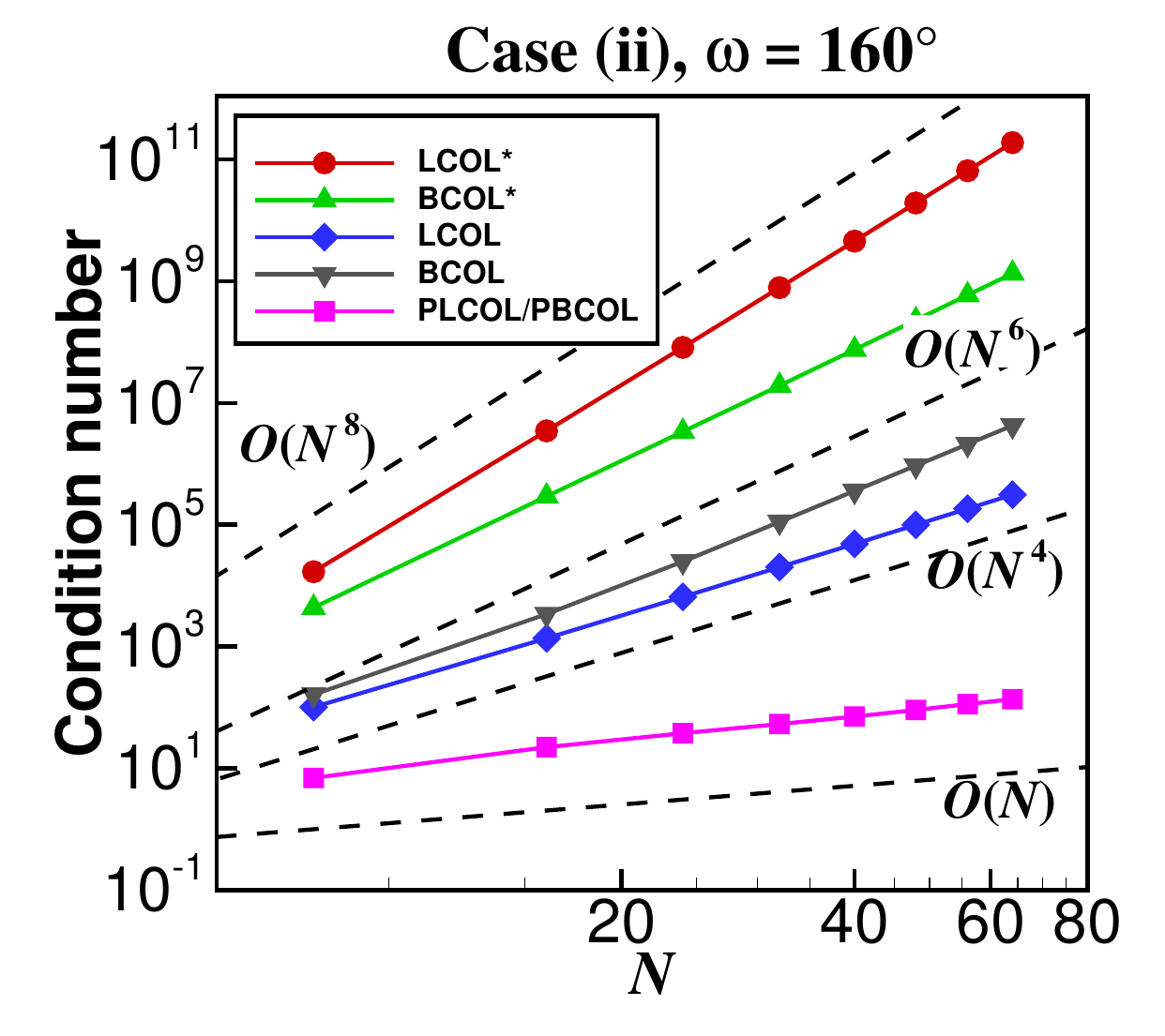}
    \caption{Condition numbers of the collocation matrices for problem \eqref{eq:PDEs_physicalDomain} on the triangle  with one interior angle $\omega = 90^\circ$ and $\omega = 160^\circ$. Left: the constant coefficient Case {\rm (i)}. Right: the variable-coefficient Case {\rm (ii)}.}
   \label{fig:Cond2_Triangle_omega160}
\end{figure}

 We compare the condition numbers of the coefficient matrices, and
 GMRES iteration counts, computational time (in seconds) for  six schemes for two sets of coefficients and two choices of $\omega=90^\circ, 160^\circ.$ When $\omega=160^\circ,$  the obtuse triangle $\mathcal T$ has the interior angles: $\angle A=160^\circ, \angle B\approx 15^\circ,$ and $\angle C\approx 5^{\circ},$ so $\mathcal T$ is highly skewed and strongly deformed. 
 This provides a stringent test of the robustness of the collocation schemes with respect to geometric distortion. Moreover, even for smooth $f$, the solution may have limited regularity due to corner singularities. Near $A=(0,0)$, the leading singular behavior is $u(r,\theta)\sim r^{\pi/\omega}\sin(\pi\theta/\omega)$, where $(r,\theta)$ are local polar coordinates. Hence, for $\omega=160^\circ$, the singular exponent is $\pi/\omega=9/8$, while for $\omega=90^\circ$, a resonant term of the form $r^2\log r$ may arise (see e.g., \cite{Dauge1988,FornbergFlyer2011}).

In Figure~\ref{fig:Cond2_Triangle_omega160}, we depict the condition numbers of six schemes, and observe that condition numbers of  PLCOL and PBCOL grow like $\mathcal{O}(N),$ while the other two pairs  behave like $\mathcal{O}(N^4),$ $\mathcal{O}(N^6)$ or $ \mathcal{O}(N^8)$ for all cases. Indeed, this shows our Birkhoff preconditioners remain highly effective even for this challenging, strongly deformed geometry and the associated singular transformation. 
The results in 
Table \ref{tab:2d_triangle_omega90} (for $\omega=90^\circ$) and   Table \ref{tab:2d_triangle_omega160} 
(for $\omega=160^\circ$) clearly indicate the necessity of preconditioning.


\begin{table}[htbp]
\centering
\caption{\footnotesize  GMRES iteration counts, CPU time (s) and relative $L^\infty$-errors/convergence order of PBCOL for \eqref{eq:PDEs_physicalDomain} on $\mathcal T$ with $\omega = 90^\circ$}
\label{tab:2d_triangle_omega90}
\footnotesize
\setlength{\tabcolsep}{3.0pt}

\begin{tabular}{
cc
cc cc cc cc cc ccc c
}
\toprule
&
&
\multicolumn{2}{c}{LCOL$^{\star}$}
& \multicolumn{2}{c}{BCOL$^{\star}$}
& \multicolumn{2}{c}{LCOL}
& \multicolumn{2}{c}{BCOL}
& \multicolumn{2}{c}{PLCOL}
& \multicolumn{4}{c}{PBCOL}
\\

\cmidrule(lr){3-4}
\cmidrule(lr){5-6}
\cmidrule(lr){7-8}
\cmidrule(lr){9-10}
\cmidrule(lr){11-12}
\cmidrule(lr){13-16}

& $N$
& Iter & Time
& Iter & Time
& Iter & Time
& Iter & Time
& Iter & Time
& Iter & Time & Error
& Rate
\\
\midrule

\multirow{6}{*}{Case (i)}
& 64
& 1666 & 15.14
& 1938 & 20.56
& 1317 & 9.673
& 937  & 5.007
& 23   & 0.036
& 23   & 0.044
& 8.78e$-$05
& -
\\

& 128
& 5367 & 725.7
& 5771 & 846.1
& 3789 & 360.2
& 2537 & 166.6
& 23   & 0.066
& 23   & 0.069
& 5.57e$-$06
& 3.98
\\

& 256
& / & /
& / & /
& / & /
& / & /
& 23   & 0.386
& 23   & 0.502
& 3.50e$-$07
& 3.99
\\

& 512
& / & /
& / & /
& / & /
& / & /
& 22   & 1.332
& 23   & 1.402
& 2.15e$-$08
& 4.02
\\

& 1024
& / & /
& / & /
& / & /
& / & /
& 22   & 8.054
& 23   & 9.208
& 1.31e$-$09
& 4.04
\\

& 2048
& / & /
& / & /
& / & /
& / & /
& 23   & 62.15
& 24   & 73.58
& 8.22e$-$11
& 3.99
\\

\midrule

\multirow{6}{*}{Case (ii)}
& 64
& 1856 & 18.81
& 1935 & 20.42
& 1498 & 10.83
& 941  & 6.567
& 21   & 0.036
& 21   & 0.020
& 8.44e$-$05
& -
\\

& 128
& 5987 & 905.8
& 5768 & 844.3
& 4295 & 444.9
& 2556 & 197.8
& 22   & 0.058
& 22   & 0.066
& 5.35e$-$06
& 3.98
\\

& 256
& / & /
& / & /
& / & /
& / & /
& 23   & 0.383
& 22   & 0.410
& 3.36e$-$07
& 3.99
\\

& 512
& / & /
& / & /
& / & /
& / & /
& 22   & 1.302
& 23   & 1.239
& 2.07e$-$08
& 4.03
\\

& 1024
& / & /
& / & /
& / & /
& / & /
& 23   & 8.044
& 24   & 8.442
& 1.26e$-$09
& 4.04
\\

& 2048
& / & /
& / & /
& / & /
& / & /
& 24   & 58.52
& 25   & 64.27
& 7.89e$-$11
& 3.99
\\

\bottomrule
\end{tabular}
\end{table}

\begin{table}[!htbp]
\centering
\caption{\footnotesize GMRES iteration counts, CPU time (s) and  relative $L^\infty$-errors/convergence order of PBCOL for \eqref{eq:PDEs_physicalDomain} on $\mathcal T$ with $\omega = 160^\circ$}
\label{tab:2d_triangle_omega160}
\footnotesize
\setlength{\tabcolsep}{3.0pt}

\begin{tabular}{
cc
cc cc cc cc cc ccc c
}
\toprule
&
&
\multicolumn{2}{c}{LCOL$^{\star}$}
& \multicolumn{2}{c}{BCOL$^{\star}$}
& \multicolumn{2}{c}{LCOL}
& \multicolumn{2}{c}{BCOL}
& \multicolumn{2}{c}{PLCOL}
& \multicolumn{4}{c}{PBCOL}
\\

\cmidrule(lr){3-4}
\cmidrule(lr){5-6}
\cmidrule(lr){7-8}
\cmidrule(lr){9-10}
\cmidrule(lr){11-12}
\cmidrule(lr){13-16}

& $N$
& Iter & Time
& Iter & Time
& Iter & Time
& Iter & Time
& Iter & Time
& Iter & Time & Error
& Rate
\\
\midrule

\multirow{6}{*}{Case (i)}
& 64
& 2304 & 25.34
& 2463 & 32.56
& 1341 & 8.811
& 1141 & 7.472
& 88   & 0.094
& 88   & 0.098
& 4.65e$-$03
& -
\\

& 128
& / & /
& / & /
& 3851 & 356.7
& 3225 & 253.2
& 92   & 0.436
& 93   & 0.352
& 9.85e$-$04
& 2.24
\\

& 256
& / & /
& / & /
& / & /
& / & /
& 95   & 2.040
& 96   & 1.839
& 2.08e$-$04
& 2.24
\\

& 512
& / & /
& / & /
& / & /
& / & /
& 98   & 4.839
& 99   & 4.606
& 4.37e$-$05
& 2.25
\\

& 1024
& / & /
& / & /
& / & /
& / & /
& 100  & 33.90
& 100  & 37.41
& 8.78e$-$06
& 2.32
\\

& 2048
& / & /
& / & /
& / & /
& / & /
& 102  & 234.5
& 102  & 274.1
& 1.54e$-$06
& 2.52
\\

\midrule

\multirow{6}{*}{Case (ii)}
& 64
& 2399 & 31.79
& 2414 & 29.37
& 1489 & 12.17
& 1116 & 7.042
& 90   & 0.101
& 90   & 0.139
& 4.34e$-$03
& -
\\

& 128
& / & /
& / & /
& 4256 & 447.2
& 3163 & 260.1
& 94   & 0.380
& 95   & 0.351
& 9.20e$-$04
& 2.24
\\

& 256
& / & /
& / & /
& / & /
& / & /
& 96   & 1.939
& 97   & 1.848
& 1.94e$-$04
& 2.24
\\

& 512
& / & /
& / & /
& / & /
& / & /
& 99   & 4.701
& 100  & 4.820
& 4.09e$-$05
& 2.25
\\

& 1024
& / & /
& / & /
& / & /
& / & /
& 101  & 30.69
& 101  & 37.46
& 8.21e$-$06
& 2.32
\\

& 2048
& / & /
& / & /
& / & /
& / & /
& 103  & 237.3
& 103  & 258.2
& 1.43e$-$06
& 2.52
\\
\bottomrule
\end{tabular}
\end{table}

We also observe from the last column of two tables that the convergence rates of PBCOL are roughly $\mathcal{O}(N^{-4})$ and $\mathcal{O}(N^{-2.25})$ for $\omega=90^\circ, 160^\circ,$ respectively. Under the transformation \eqref{eq:Duffy_transformation}, the solution of \eqref{eq:scaled_transformed_PDE}: $\hat u\sim (1-\xi_2)^{\frac \pi \omega}$ on $(-1,1)^2,$ so the tensorial Legendre polynomial approximation has a convergence $O(N^{-\frac {2\pi} \omega})$ in $L^\infty$-norm (see e.g., \cite{Sidi2009}), i.e., the convergence order for $\omega=160^\circ$ is about $\mathcal{O}(N^{-\frac {9} 4})$ as shown in Table  \ref{tab:2d_triangle_omega160}.  For $\omega=90^\circ,$ we have logarithmic singularity $\hat u\sim (1-\xi_2)^2\log (1-\xi_2),$ so the best possible convergence order in $L^\infty$-norm is $\mathcal{O}(N^{-4}),$ according to 
\cite{Sidi2009}, which has a good agreement with the last column in Table  \ref{tab:2d_triangle_omega90}.

\subsubsection{Application to Allen-Cahn equation} 
It is known that  time discretisation of many time-dependent problems by implicit-explicit (IMEX) schemes requires the solution of a second-order elliptic problem at each time step. The proposed approach makes collocation methods practical and compelling alternatives for such applications. As an illustration, we consider the Allen-Cahn equation and demonstrate that it  is competitive with the adaptive spectral methods in \cite{ShenYang2009}, while offering clear advantages for more complex initial data.

Consider the Allen-Cahn equation  as in \cite{ShenYang2009}:
\begin{equation}\label{eq:AC}
    u_t
    =
    \gamma\big(
    \Delta u
    -\tfrac{1}{\eta^2}f(u)
    \big),
    \quad
    f(u)=u^3-u,
    \quad
    (\bm{x},t)\in\Omega\times(0,T],
\end{equation}
where $\Omega=(-1,1)^2$. The equation is supplemented with the homogeneous Dirichlet boundary condition and the initial condition
\(
    u(\bm{x},0)=u_0(\bm{x}).
\)
Here $\gamma>0$ is a time-scaling parameter, and the parameter $\eta>0$ controls the diffuse-interface thickness.

Let $t_n=n\tau$ and denote by $u^n$ the approximation to $u(\cdot,t_n)$. We adopt the same temporal discretisation as in \cite{ShenYang2009}:
\begin{equation}\label{eq:AC-BDF2}
\begin{aligned}
    \tfrac{3u^{n+1}-4u^n+u^{n-1}}{2\tau}
    &+
    \tfrac{\gamma S}{\eta^2}
    (
        u^{n+1}-2u^n+u^{n-1}
    )
    -\gamma\Delta u^{n+1}
    =
    -\tfrac{\gamma}{\eta^2}
    [
        2f(u^n)-f(u^{n-1})
    ],
\end{aligned}
\end{equation}
 where $S>0$ is the stabilization parameter. 
Thus, at each time step, we solve
\begin{equation}\label{eq:AC-Helmholtz}
    (
        \tfrac{3}{2\tau}
        +\tfrac{\gamma S}{\eta^2}
        -\gamma\Delta
    )u^{n+1}
    =
    F^n,
\end{equation}
where
\[
    F^n
    =
    (
        \tfrac{2}{\tau}
        +\tfrac{2\gamma S}{\eta^2})u^n
    -
    (
        \tfrac{1}{2\tau}
        +\tfrac{\gamma S}{\eta^2}
    )u^{n-1}
    -
    \tfrac{\gamma}{\eta^2}
    [
        2f(u^n)-f(u^{n-1})
    ].
\]
Here, we use the Birkhoff collocation scheme for \eqref{eq:AC-Helmholtz}. According to Remark \ref{DirectInv_constant},  we can explicitly invert the system using the stable diagonalisation, precomputed only once. As a result, PBCOL becomes a direct solver, and 
the time-marching is very  efficient and stable for large $N$.

We first test the shrinkage of a circular domain under the same setting 
as  the moving-mesh spectral method in \cite[Subsection 4.1]{ShenYang2009}, where the
physical domain is $[0,256]^2;$  the initial radius is $R_0=100,$ and $\gamma=\eta=1.$ Note that the exact radius evolves as $R^2(t)=R_0^2-2t$. After mapping the physical domain to $[-1,1]^2$, the parameters become  $\gamma=\tfrac{1}{128^2} \approx 6.10351\times10^{-5}$ and 
$\eta=\tfrac 1{128}\approx 0.0078$. Similar to \cite{ShenYang2009}, we take $S=1, \tau=10^{-2}$ and the initial data to be 
\[
u_0(x,y) = \tanh \big(\!-\tfrac{\sqrt{x^2 + y^2} - R_0/128}{\sqrt{2}\eta}\big)(1-x^4)(1-y^4).
\]
In Figure~\ref{fig:singlebubble}, we compare  the shrinkage of the circular domain between PBCOL 
  (with a 
$300^2$ grid) and  the moving mesh method (with a 
$65^2$ grid) in \cite{ShenYang2009} in comparable  computational time up to $T=5000$. Our approach takes about $2748$ seconds, while the moving mesh takes about $3000$ seconds. Moreover, we observe from Figure~\ref{fig:singlebubble} (middle) that PBCOL perfectly fits the exact radius 
$R^2(t)=R_0^2-2t,$ so it is more accurate. It is also evident that the collocation approach is much easier to implement.


\begin{figure}[htbp]
    \centering

\begin{minipage}[c]{0.32\textwidth}
    \centering
    \includegraphics[width=0.48\linewidth]{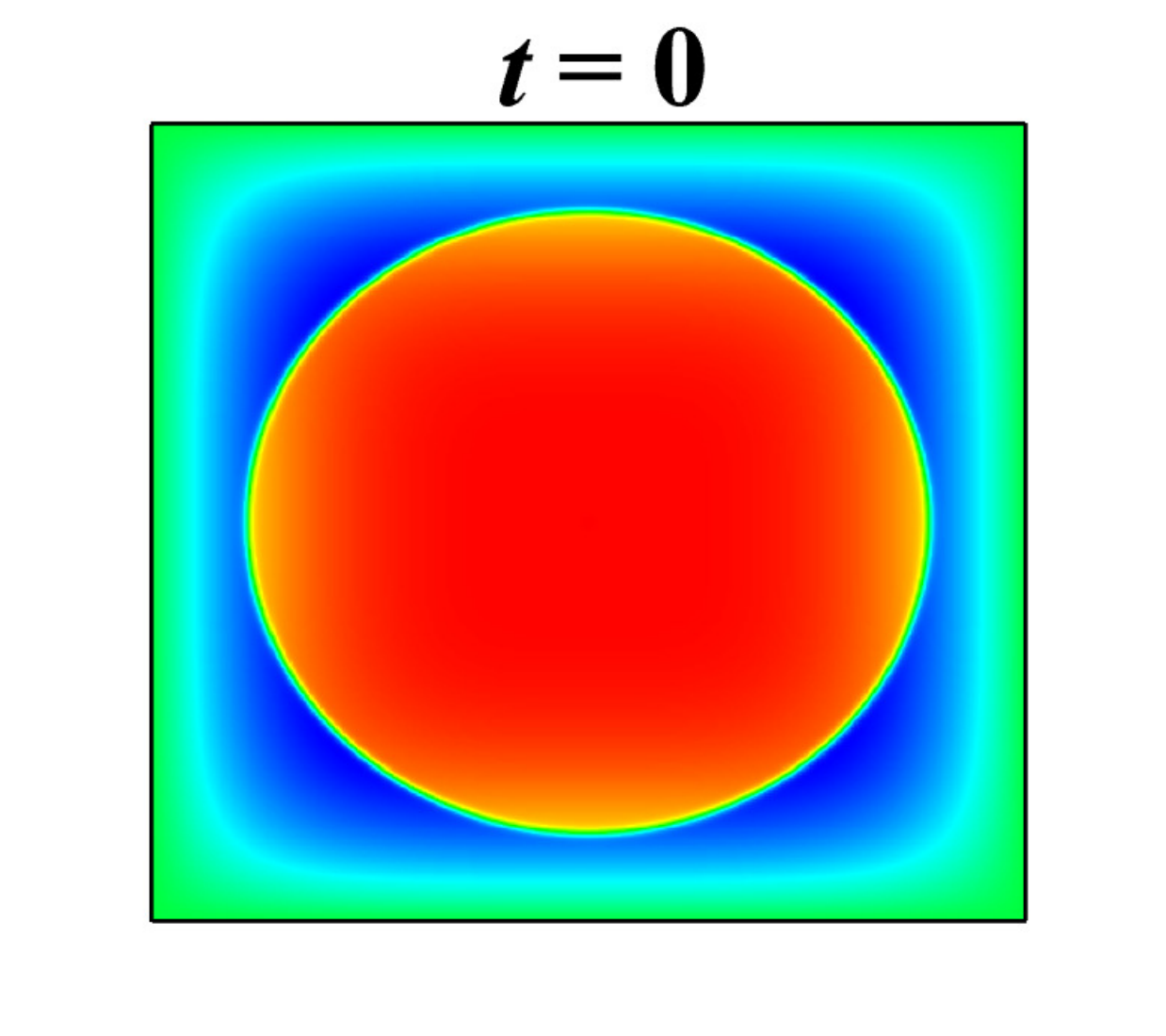}
    \includegraphics[width=0.48\linewidth]{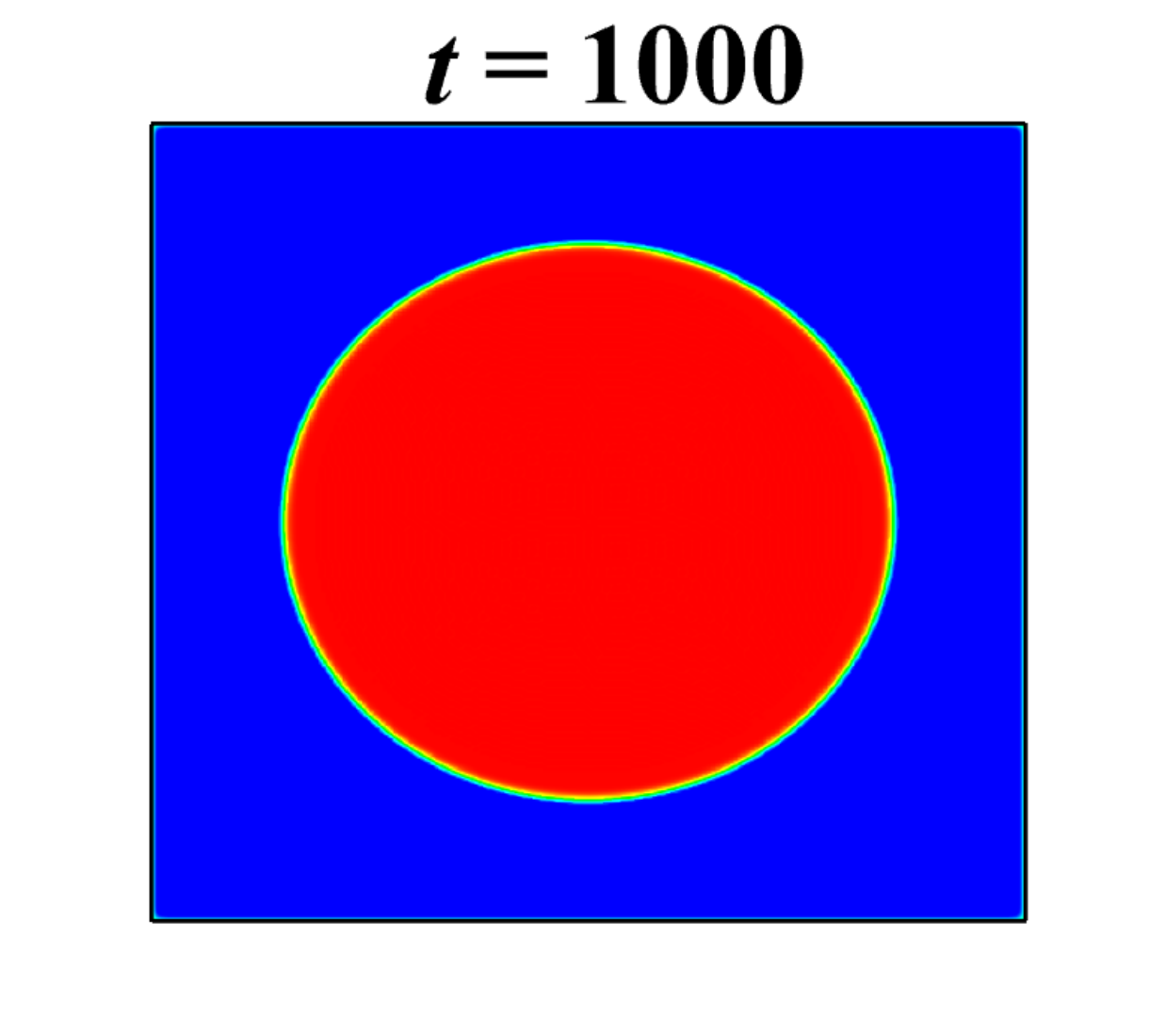}\\
    \includegraphics[width=0.48\linewidth]{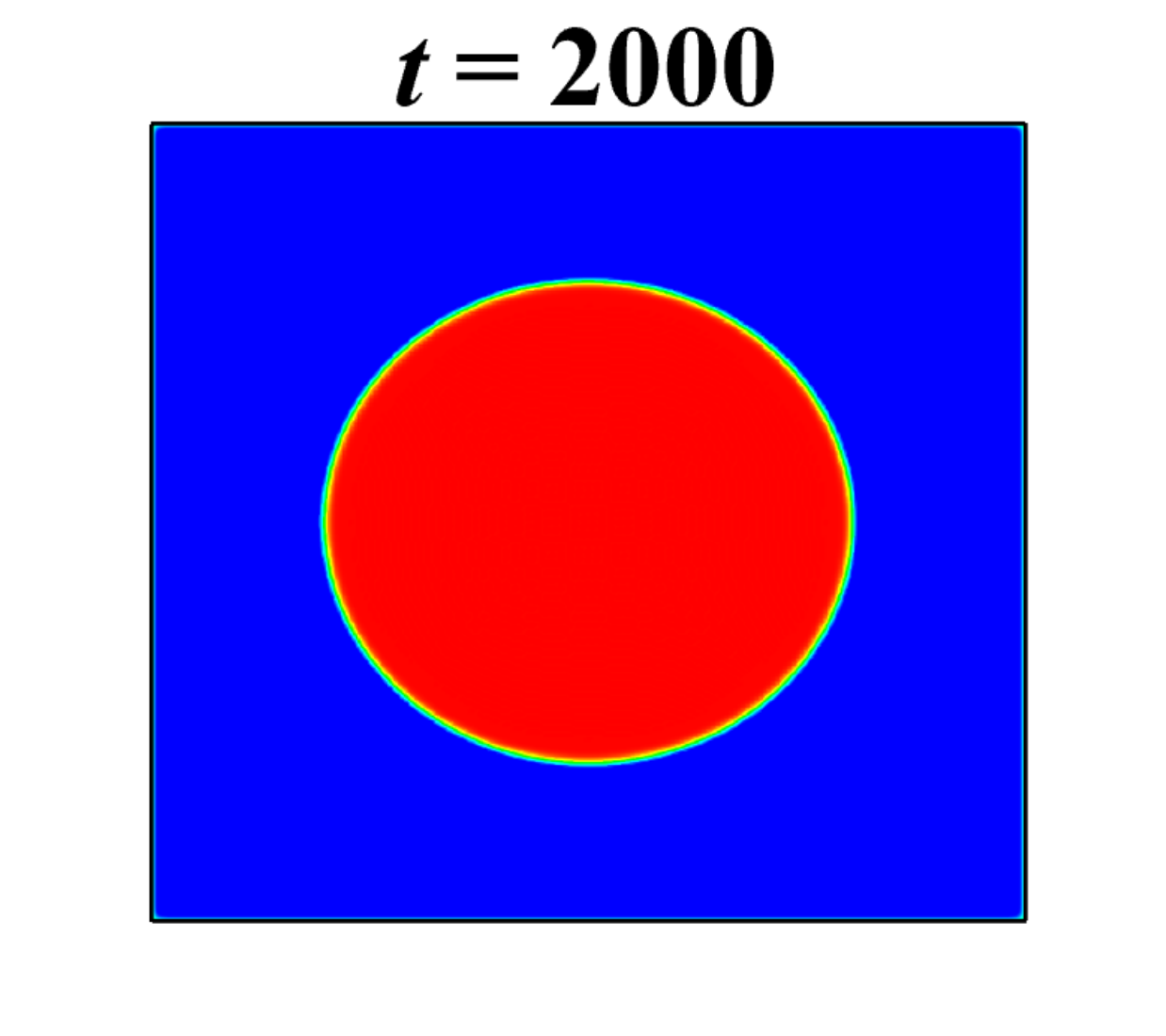}
    \includegraphics[width=0.48\linewidth]{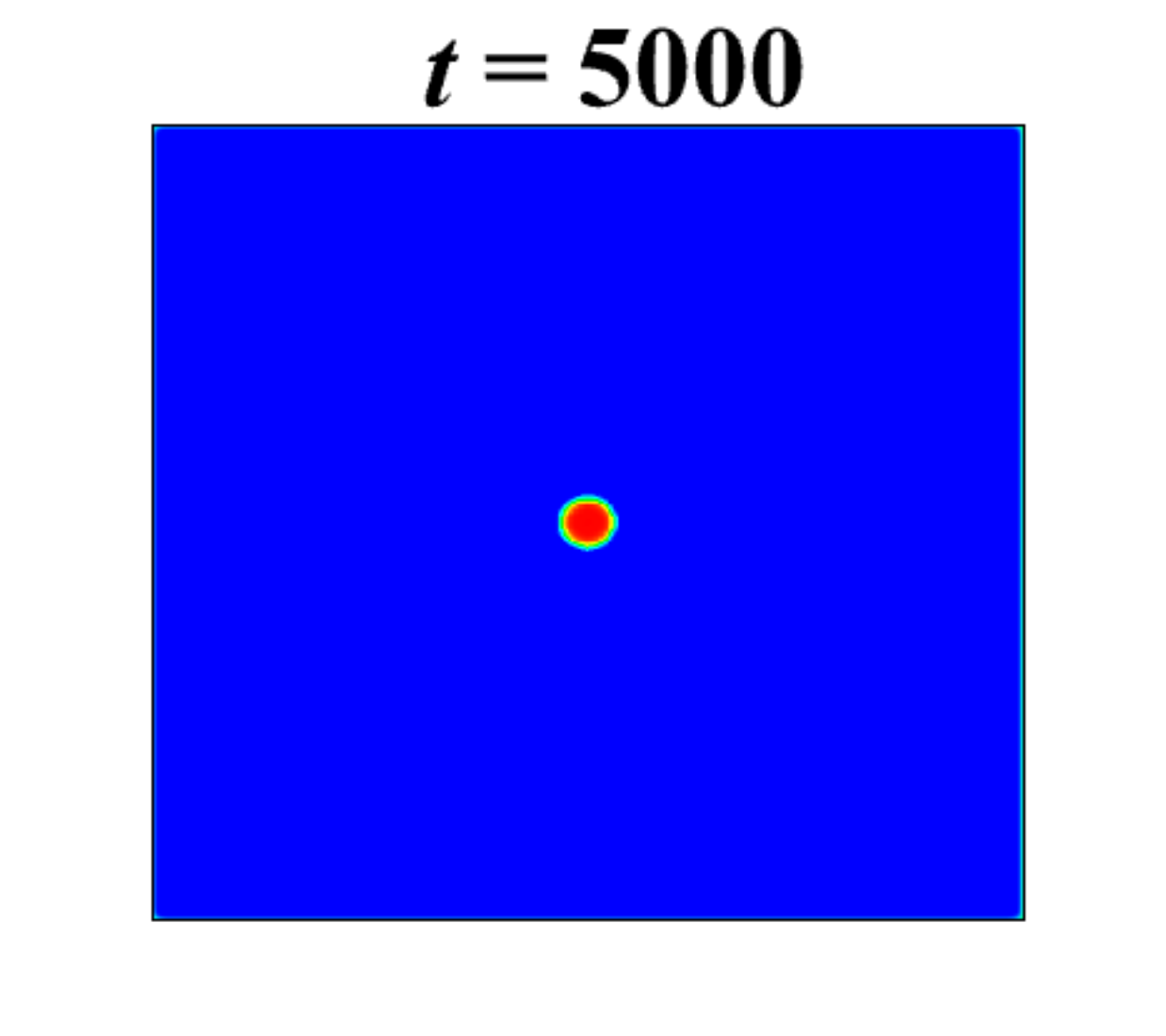}
\end{minipage}%
\begin{minipage}[c]{0.33\textwidth}
    \centering
    \includegraphics[width=\linewidth]{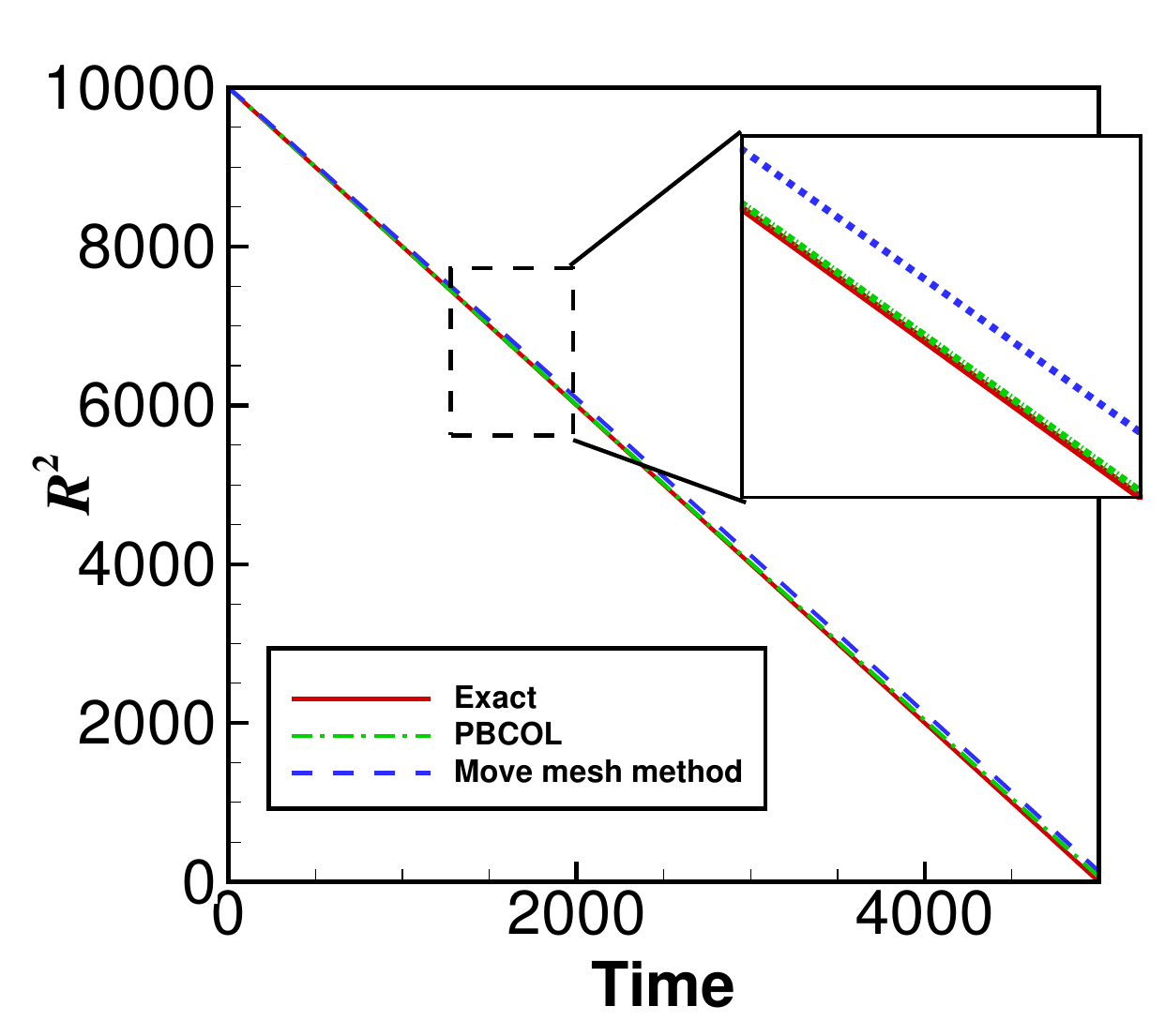}
\end{minipage}%
\begin{minipage}[c]{0.32\textwidth}
    \centering
    \includegraphics[width=0.48\linewidth]{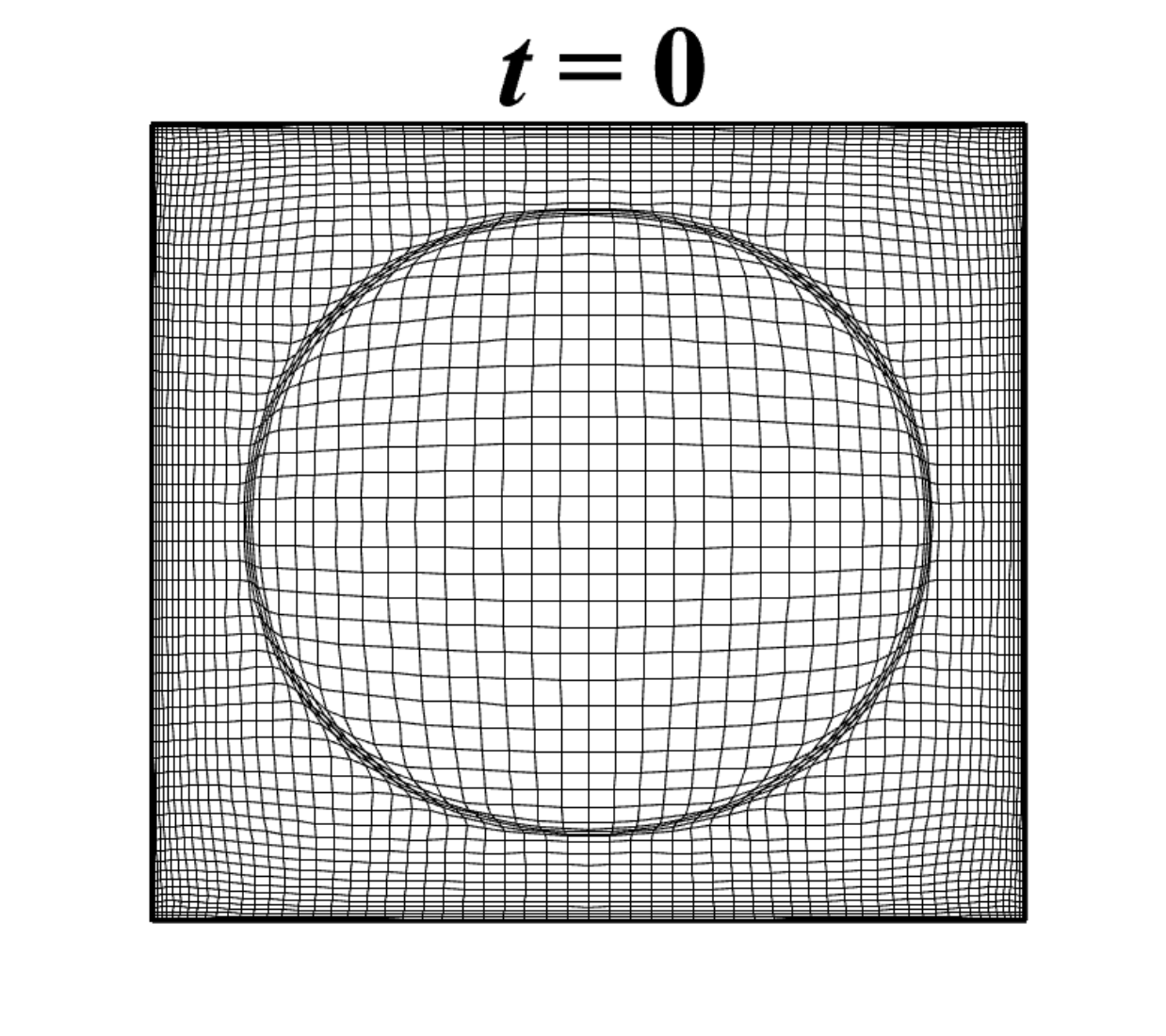}
    \includegraphics[width=0.48\linewidth]{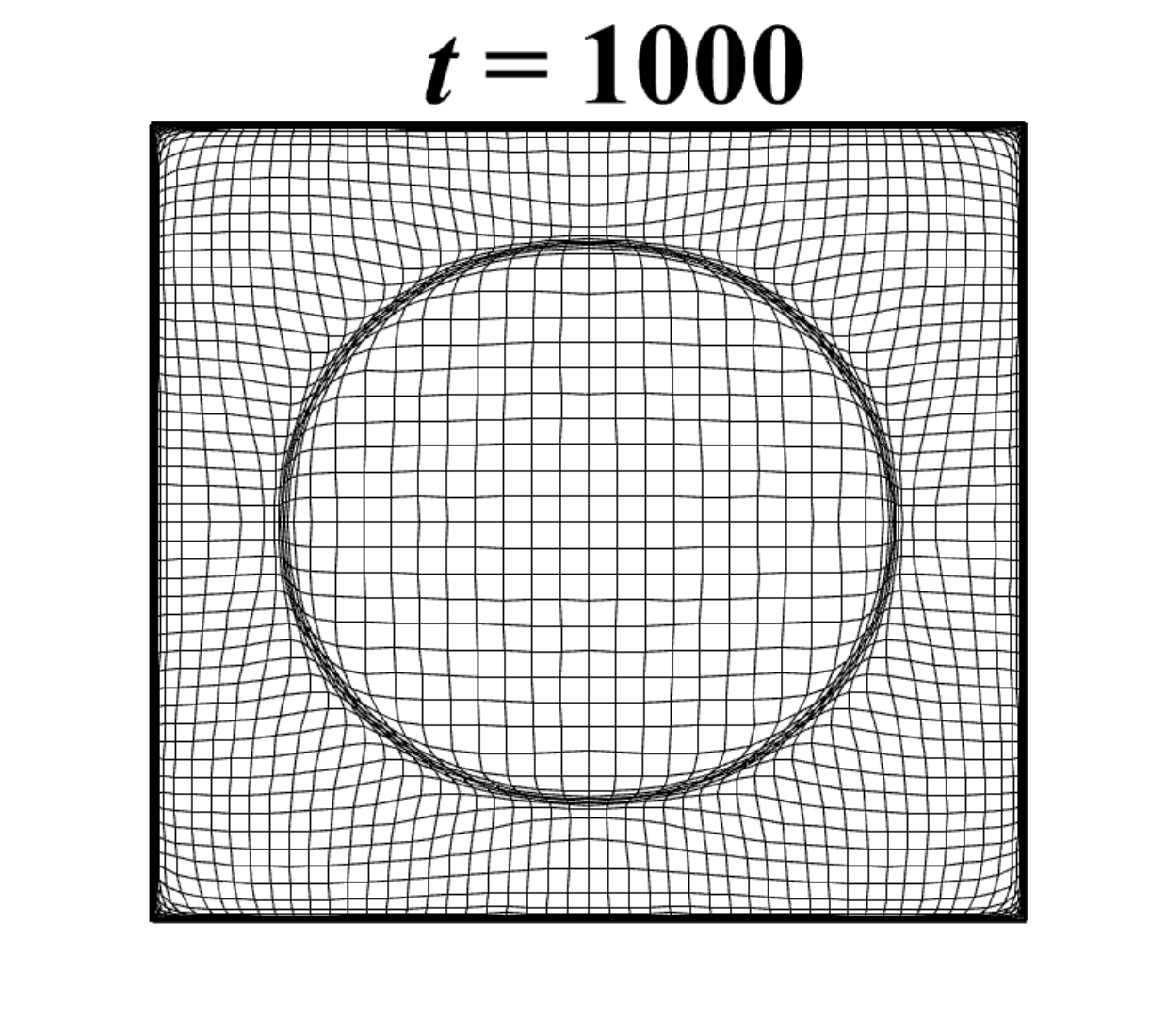}\\
    \includegraphics[width=0.48\linewidth]{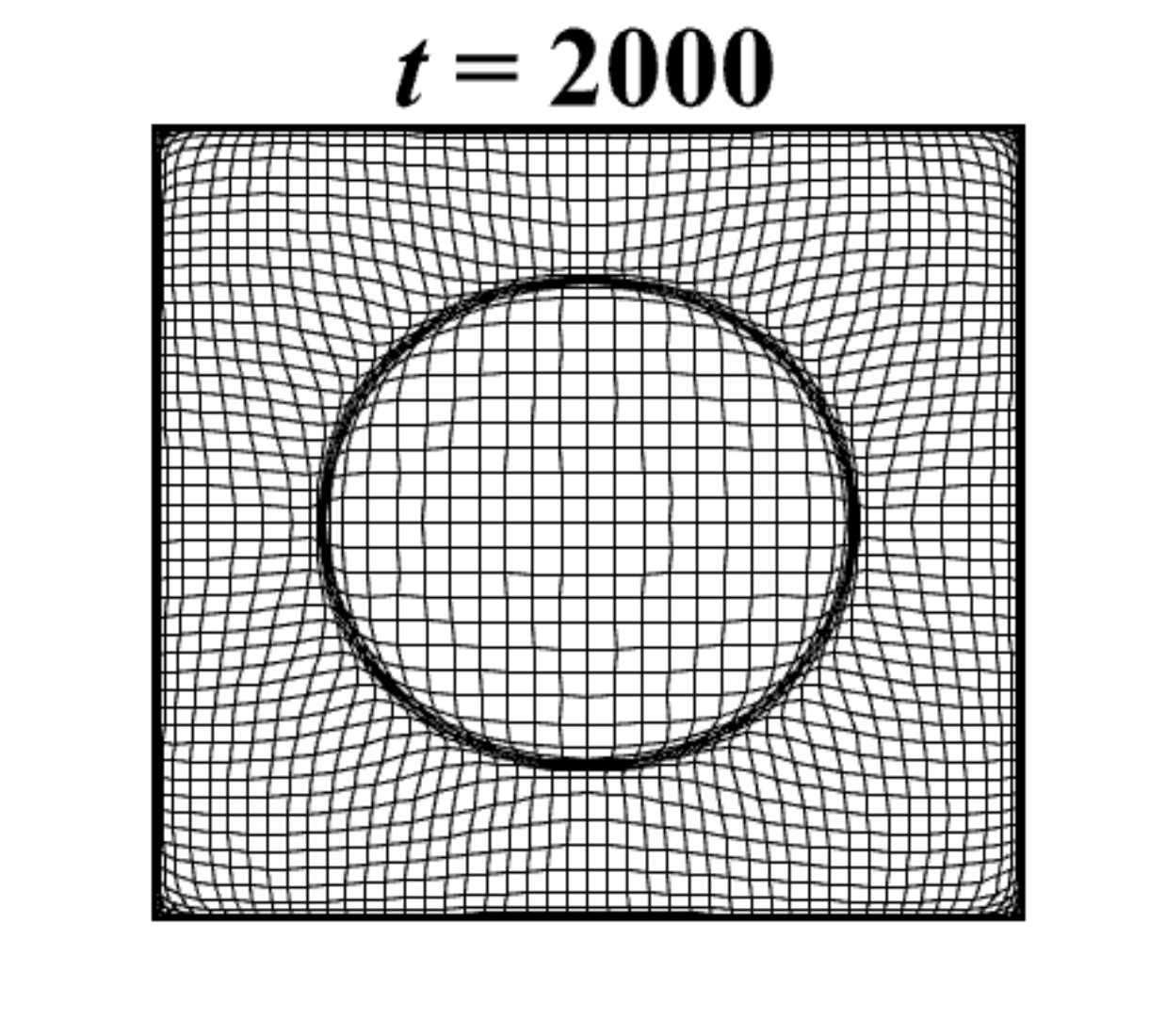}
    \includegraphics[width=0.48\linewidth]{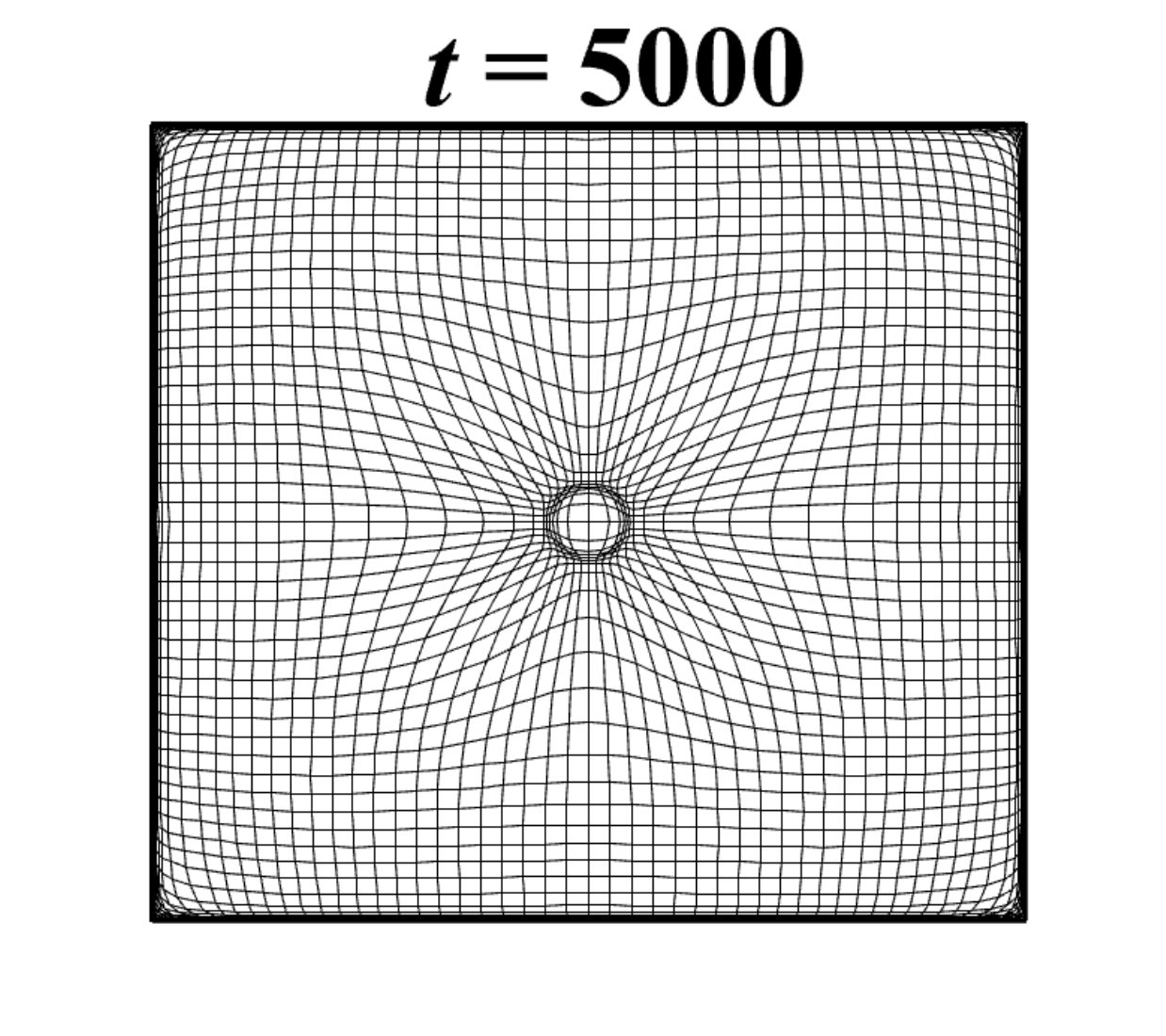}
\end{minipage}

    \caption{Shrinkage of a circular domain. Left: evolution computed by PBCOL on a $300^2$ grid at $t=0,1000,2000,$ and $5000$. Middle: evolution of the radius $R^2(t)$ by PBCOL {\rm(}$300^2$ grid{\rm)} and the moving-mesh method {\rm(}$65^2$ grid{\rm)}, together with the exact line. Right: evolution of the moving mesh corresponding to the moving mesh simulation at the same time instances as in \cite[Figure 3]{ShenYang2009}.}
\label{fig:singlebubble}
\end{figure}

 We next conduct a more challenging test for an initial data involving \(40\) randomly distributed circular domains:
 \[
u_0(x,y)=\tanh(-\tfrac{d(x,y)}{\sqrt{2}\eta})(1-x^4)(1-y^4),\quad d(x,y)=\min_{1\le k\le 40}\big\{\sqrt{(x-x_k)^2+(y-y_k)^2}-R_k\big\},
\]
 where the centers \((x_k,y_k)\) are randomly sampled in \([-0.7,0.7]^2\) and likewise for the radii \(R_k\in[0.05,0.13]\). In such a complicated situation, we find the moving-mesh spectral method 
 becomes unstable with so many interior interfaces to track and move the meshes towards. However, as shown in Figure \ref{fig:AC_coarsening_energy}, PBCOL with a $300^2$ grid produces very accurate results, and 
 the (numerical) free energy decreases monotonically throughout the simulation.  The total computational  time up to $T=800$ (with $\tau=0.01$ and $80000$ time steps) is about $7$ minutes.


\begin{figure}[!h]
    \centering
\includegraphics[width=0.9\linewidth]{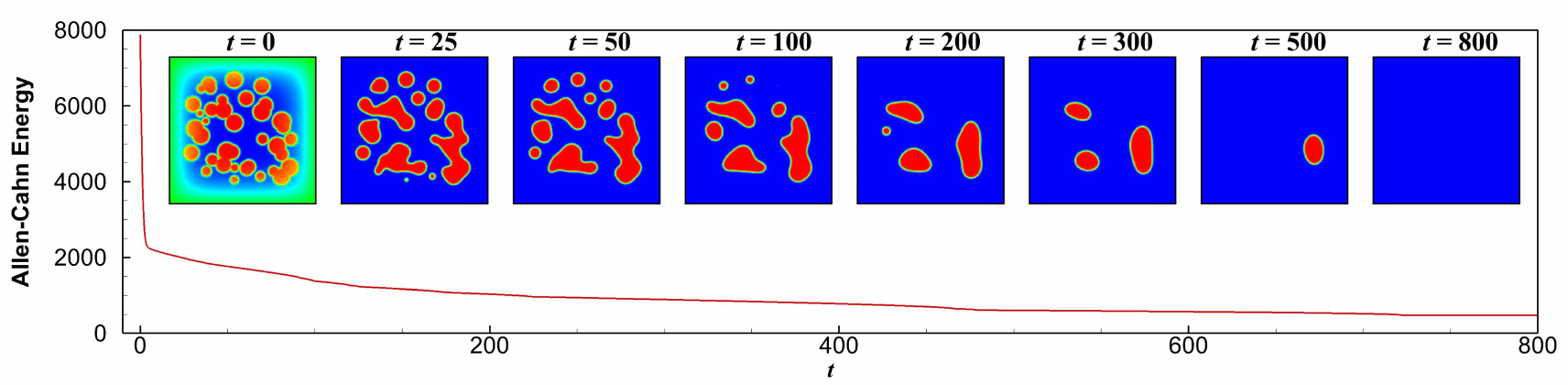}
    \caption{Long-time evolution of the Allen-Cahn phase field and the corresponding decay of the free energy.}
   \label{fig:AC_coarsening_energy}
\end{figure}

\section{Birkhoff preconditioners in three dimensions}\label{sec:extension}
In this section, we extend the two-dimensional preconditioning framework to three-dimensional elliptic problems. 
Similarly, we first construct the preconditioners for separable variable problems and then consider general variable-coefficient problems, followed by numerical experiments.

\subsection{Construction of Birkhoff preconditioners}\label{sec:3D_preconditioner}
We first consider 
\begin{equation}\label{eq:3D_separable_operator}
\mathcal L_3[u]
=
-a(x)u_{xx}
-b(y)u_{yy}
-c(z)u_{zz}
+q(x,y,z)u,
\quad
(x,y,z)\in \Omega=(-1,1)^3,
\end{equation}
where  the coefficients are given continuous functions such that
\(
0<\alpha\leq a(x),\,b(y),\,c(z)\leq\beta,
\) and $0\leq q_{\rm min}\leq q(x,y,z)\leq q_{\rm max}$.
Let
\[
\mathbf C_a=\operatorname{diag}\bigl(a(x_i)\bigr),\quad
\mathbf C_b=\operatorname{diag}\bigl(b(y_j)\bigr),\quad
\mathbf C_c=\operatorname{diag}\bigl(c(z_k)\bigr)
\]
be the corresponding coefficient matrices, and 
\(
\mathbf C_q
=
\operatorname{diag}
\bigl(q(x_i,y_j,z_k)\bigr)
\)
denote the three-dimensional diagonal evaluation matrix of \(q\).
With the same tensor-product ordering as in Section~\ref{sec: 2D_Preconditioner}, the
Lagrange-collocation and Birkhoff-collocation matrices are
\begin{subequations}\label{PLPB3D}
    \begin{align}
\mathbf A_{\rm L}^{(3)}
={}&
-\mathbf I_{N-1}\otimes\mathbf I_{N-1}
 \otimes(\mathbf C_a\mathbf D^{(2)})
-\mathbf I_{N-1}\otimes
 (\mathbf C_b\mathbf D^{(2)})\otimes\mathbf I_{N-1}
\label{eq:3D_AL}\\
&-(\mathbf C_c\mathbf D^{(2)})
 \otimes\mathbf I_{N-1}\otimes\mathbf I_{N-1}
+\mathbf C_q,\nonumber\\
\mathbf A_{\rm B}^{(3)}
={}&
-\mathbf B\otimes\mathbf B\otimes\mathbf C_a
-\mathbf B\otimes\mathbf C_b\otimes\mathbf B
-\mathbf C_c\otimes\mathbf B\otimes\mathbf B+\mathbf C_q
(\mathbf B\otimes\mathbf B\otimes\mathbf B).
\label{eq:3D_AB}
\end{align}
\end{subequations}

As in the two-dimensional case, we construct the right
preconditioners from the second-order principal parts of
\eqref{PLPB3D} as in \eqref{defn38}, that is,  
\begin{subequations}\label{3Dprecon}
\begin{align}
\mathbf P_{\rm L}^{(3)}
={}&
-\big[
\mathbf I_{N-1}\otimes\mathbf I_{N-1}
 \otimes(\mathbf C_a\mathbf D^{(2)})
+
\mathbf I_{N-1}\otimes
 (\mathbf C_b\mathbf D^{(2)})\otimes\mathbf I_{N-1}
\label{eq:3D_PL}\\
&\hspace{2.0cm}
+
(\mathbf C_c\mathbf D^{(2)})
 \otimes\mathbf I_{N-1}\otimes\mathbf I_{N-1}
\big]^{-1},
\nonumber\\
\mathbf P_{\rm B}^{(3)}
={}&
-\big[
\mathbf B\otimes\mathbf B\otimes\mathbf C_a
+
\mathbf B\otimes\mathbf C_b\otimes\mathbf B
+
\mathbf C_c\otimes\mathbf B\otimes\mathbf B
\big]^{-1}.
\label{eq:3D_PB}
\end{align}
\end{subequations}
The diagonalized representations follow  from the
one-dimensional construction in Theorem \ref{WinD2in_eq_D2inTWin}:
\[
\mathbf C_\nu\mathbf D^{(2)}
=
\mathbf V_\nu\mathbf\Sigma_\nu^{-1}\mathbf V_\nu^{-1},
\quad
\mathbf B\mathbf C_\nu^{-1}
=
\mathbf V_\nu\mathbf\Sigma_\nu\mathbf V_\nu^{-1}, \quad \nu=a,b,c,
\]
where
\(
\mathbf\Sigma_\nu
=
\operatorname{diag}
(\sigma_{\nu,1},\ldots,\sigma_{\nu,N-1}).
\)
Then as a direct extension of Theorem \ref{prop:explicit_preconditioners}, we have 
\begin{subequations}
\begin{align}
    &\mathbf P_{\rm L}^{(3)}
=
-(\mathbf V_c\otimes\mathbf V_b\otimes\mathbf V_a)
\mathbf\Lambda_{\rm L}^{(3)}
(\mathbf V_c^{-1}\otimes\mathbf V_b^{-1}\otimes\mathbf V_a^{-1}),\label{eq:3D_PL_diag}\\
&\mathbf P_{\rm B}^{(3)}
=
-
(\mathbf C_c^{-1}\mathbf V_c
 \otimes\mathbf C_b^{-1}\mathbf V_b
 \otimes\mathbf C_a^{-1}\mathbf V_a)
\mathbf\Lambda_{\rm B}^{(3)}
(\mathbf V_c^{-1}\otimes\mathbf V_b^{-1}\otimes\mathbf V_a^{-1}),\label{eq:3D_PB_diag}
\end{align}
\end{subequations}
where
\begin{equation}\label{eq:3D_LambdaLB}
\mathbf\Lambda_{\rm L}^{(3)}
=
\operatorname{diag}\big(
\tfrac{
\sigma_{a,i}\sigma_{b,j}\sigma_{c,k}
}{
\sigma_{b,j}\sigma_{c,k}
+\sigma_{a,i}\sigma_{c,k}
+\sigma_{a,i}\sigma_{b,j}
}
\big)_{i,j,k=1}^{N-1}, 
\quad \mathbf\Lambda_{\rm B}^{(3)}
=
\operatorname{diag}\big(
\tfrac{1}{
\sigma_{b,j}\sigma_{c,k}
+\sigma_{a,i}\sigma_{c,k}
+\sigma_{a,i}\sigma_{b,j}
}
\big)_{i,j,k=1}^{N-1}. 
\end{equation}
Therefore, the three-dimensional separable construction requires only the
three one-dimensional stable diagonalisations associated with
\(a, b\) and \(c\), while the remaining operations are diagonal
tensor-product operations.

For general variable-coefficient problem, i.e., \eqref{eq:general_elliptic} in three dimensions.
We construct the preconditioners  by taking  the diagonal principal coefficients using their
directional averages. Like \eqref{eq:general_avg_coeff}, we set
\begin{equation}\label{eq:3D_average_coeff}
a(x)=\bar a_{11}(x)
=
\frac{1}{4}
\int_{-1}^{1}\int_{-1}^{1}
a_{11}(x,y,z)\,{\rm d}y\,{\rm d}z,
\end{equation}
and likewise,  \(b(y)=\bar a_{22}(y)\) and \(c(z)=\bar a_{33}(z)\).
The corresponding one-dimensional coefficient matrices
\(\mathbf C_a\), \(\mathbf C_b\), and \(\mathbf C_c\) are then constructed
from these averaged coefficients and substituted directly into the
three-dimensional preconditioners \eqref{3Dprecon}. Thus, the
non-separable variable-coefficient case retains the same computational
structure as the separable construction above.

\begin{rmk}\label{Rmk:3D-theory}
The analysis of  eigenvalue distributions and conditioning, and estimates in two dimensions in Subsection \ref{Subsec:eigenanl} can be extended to three dimensions without significant difficulty.     \end{rmk}

\subsection{Numerical results}

We now assess the performance of the proposed preconditioned collocation methods, and focus on solving variable-coefficient problems resulting from coordinate transformations of the three-dimensional modified Helmholtz problems:
\begin{equation}\label{eq:3D_Helmholtz}
-\Delta u+\gamma u=f \quad \text{in } \; \widehat{\Omega}; \quad u=0 \quad \text{on }\; \partial\widehat{\Omega},\quad \gamma \ge 0,
\end{equation}
where $\widehat{\Omega}$ is  a tetrahedron or a curved hexahedron, as illustrated in Figure \ref{fig:Duffy_transformations3D}. 
We adopt the 3D Duffy 
and Gordon-Hall transformations to covert \eqref{eq:3D_Helmholtz} to the reference hexahedron $\Omega=(-1,1)^3.$


\begin{figure}[htbp]
    \centering
    \includegraphics[width=0.48\linewidth]{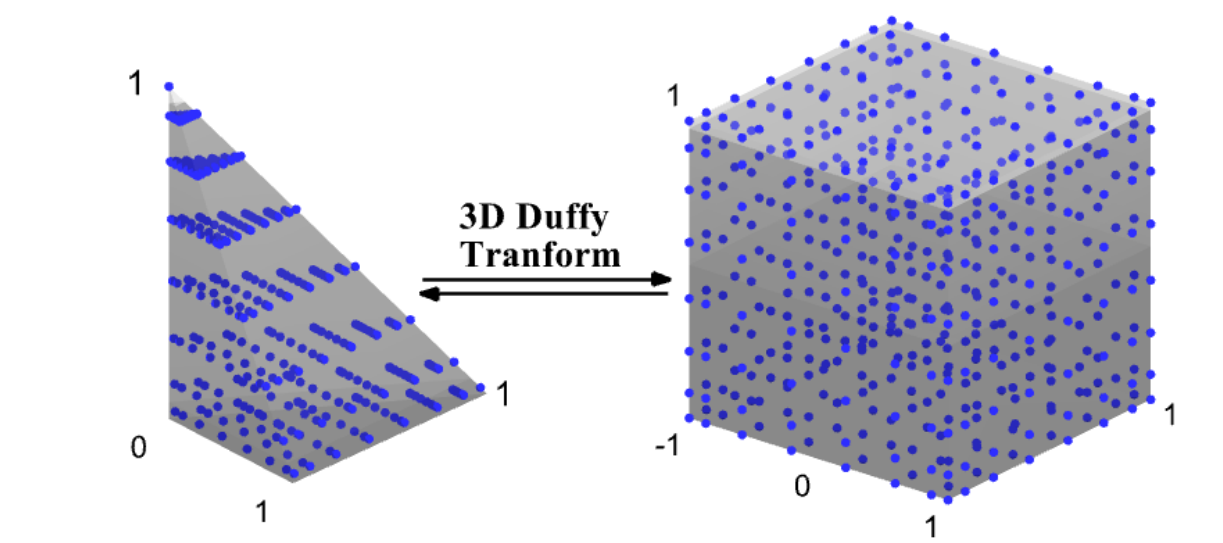} 
    \quad \includegraphics[width=0.48\linewidth]{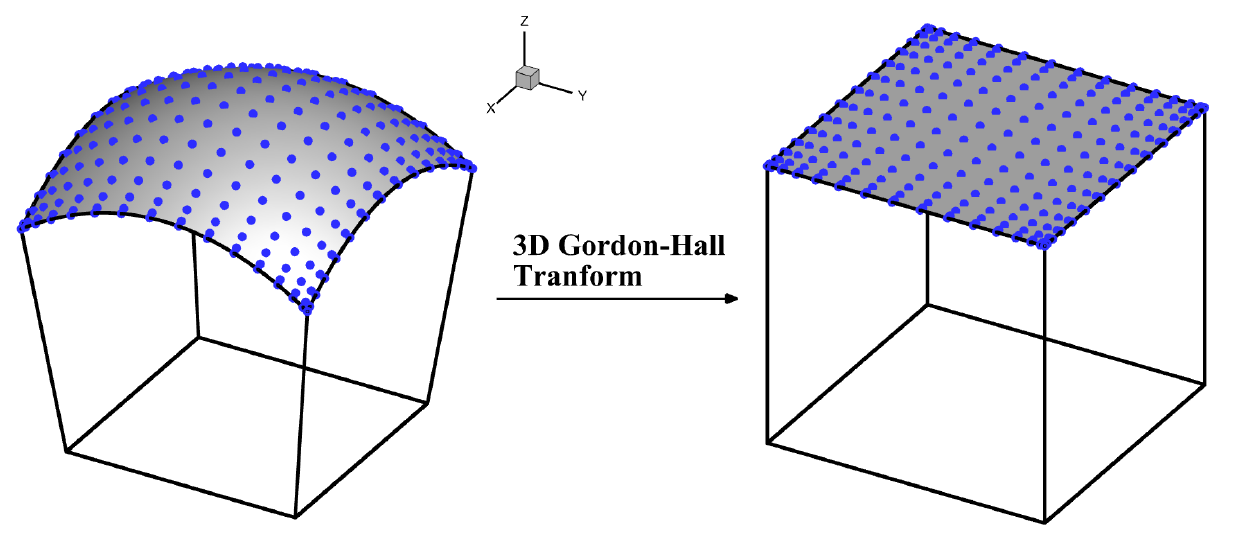} 
    \caption{Transformations from physical domains to the reference cube: the 3D Duffy transformation for a tetrahedron  (left) and the 3D Gordon-Hall transformation for a curved hexahedron (right).}
\label{fig:Duffy_transformations3D}
\end{figure}

We first consider \eqref{eq:3D_Helmholtz} on the tetrahedron
\(
\widehat{\Omega}=\operatorname{conv}\{(0,0,0),(1,0,0),(0,1,0),(0,0,1)\},
\)
with
\(
\gamma=100,f(x,y,z)=10^4\exp\big(\!-\tfrac{x^2+y^2+z^2}{0.05^2}\big).
\)
The reference solution is computed by PBCOL with $N=512$. 
Like \eqref{eq:Duffy_transformation},
the three-dimensional Duffy transformation reads
\begin{equation}\label{eq:3D_duffy}
\mathbf x(\boldsymbol\xi)
=
\begin{bmatrix}
\dfrac{(1+\xi_1)(1-\xi_2)(1-\xi_3)}{8}
,
\dfrac{(1+\xi_2)(1-\xi_3)}{4}
,
\dfrac{1+\xi_3}{2}
\end{bmatrix}^\intercal,
\quad \bm{\xi}=(\xi_1,\xi_2,\xi_3)^{\intercal}\in\Omega,
\end{equation}
which collapses the top surface of 
$\Omega$ to the vertex $(0,0,1)$ of 
$\widehat \Omega$ (see Figure
\ref{fig:Duffy_transformations3D}).
Following the same process as in
Subsection \ref{subsub:Duffy2D}, we can obtain the mapped variable coefficient problems on $\Omega$ and construct the collocation systems. Here, we omit the details.




\begin{figure}[htbp]
    \centering
\includegraphics[width=0.56\linewidth]{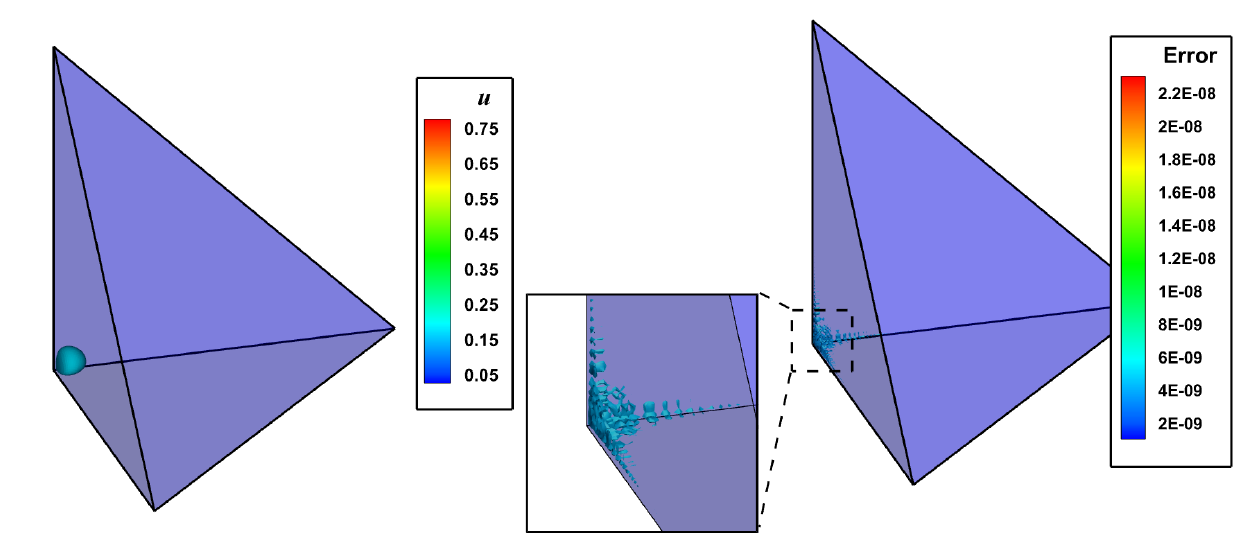} 
\quad     \includegraphics[width=0.3\linewidth]{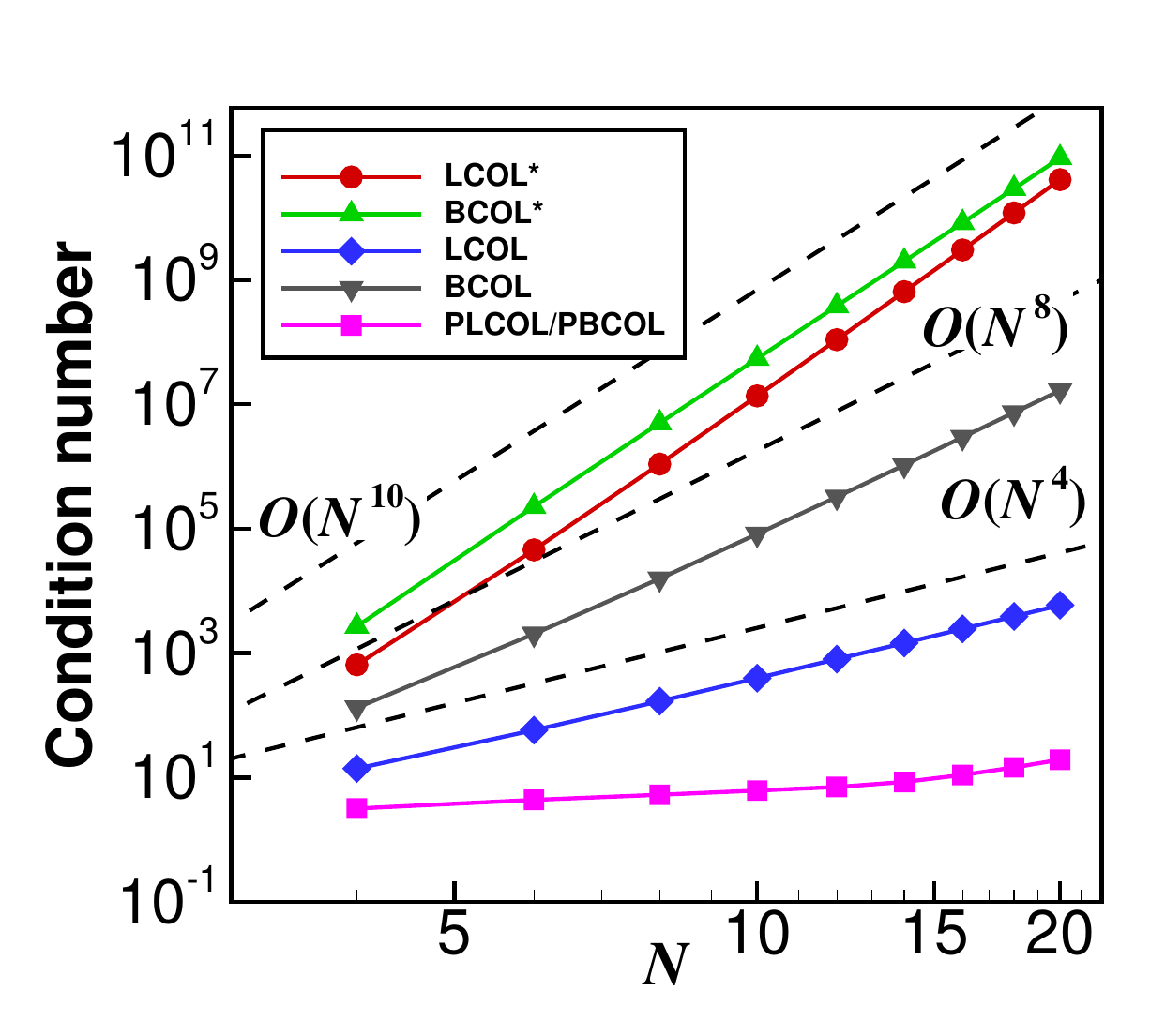} 
    \caption{Numerical results for the three-dimensional Duffy-transformed
problem. Left: solution isosurface at \(u=0.2\) for \(N=512\).
Middle: error isosurface at \(4\times10^{-9}\) for \(N=128\).
Right: condition-number growth of the six collocation systems.}
\label{fig:3Dduffy_angle}
\end{figure}

Figure~\ref{fig:3Dduffy_angle} (left) shows the numerical results for the
three-dimensional Duffy-transformed problem. The left panel shows the
isosurface of the numerical solution at \(u=0.2\) for \(N=512\), which
captures the localized corner singularity near the origin. The middle panel
displays the error isosurface at \(4\times10^{-9}\) for \(N=128\),
showing that the numerical error remains small even near the collapsed
region of the Duffy transformation. More importantly, the right panel shows that the condition numbers of $\mathrm{LCOL}^{\star}$ and $\mathrm{BCOL}^{\star}$ grow approximately as $\mathcal{O}(N^{10})$ and $\mathcal{O}(N^8)$, respectively, while those of LCOL and BCOL exhibit a much slower $\mathcal{O}(N^4)$ growth. By contrast, the condition numbers of PLCOL and PBCOL remain substantially smaller and increase only mildly with $N$, demonstrating the effectiveness of the proposed preconditioners for the three-dimensional Duffy-transformed problem.












\begin{table}[!htbp]
\centering
\caption{\footnotesize GMRES iteration counts, and CPU time (s) and relative $L^{\infty}$-errors/convergence order of PBCOL for \eqref{eq:3D_Helmholtz} on $\widehat{\Omega}$.}
\label{tab:duffy_3d_error_iter_cpu}
\footnotesize
\setlength{\tabcolsep}{3.0pt}
\sbox{\ReferenceTableBox}{
\begin{tabular}{c cc cc cc cc cc cccc}
\toprule
\multirow{2}{*}{$N$} 
& \multicolumn{2}{c}{LCOL$^\star$} 
& \multicolumn{2}{c}{BCOL$^\star$} 
& \multicolumn{2}{c}{LCOL} 
& \multicolumn{2}{c}{BCOL} 
& \multicolumn{2}{c}{PLCOL} 
& \multicolumn{4}{c}{PBCOL} \\

\cmidrule(lr){2-3}
\cmidrule(lr){4-5}
\cmidrule(lr){6-7}
\cmidrule(lr){8-9}
\cmidrule(lr){10-11}
\cmidrule(lr){12-15}

& Iter & Time
& Iter & Time
& Iter & Time
& Iter & Time
& Iter & Time
& Iter & Time & Error & Rate \\
\midrule

8
& 222  & 0.078
& 326  & 0.136
& 98   & 0.020
& 221  & 0.064
& 22   & 0.017
& 22   & 0.022
& 8.04e$-$01
& - \\

16
& 1350 & 8.722
& 2593 & 35.30
& 309  & 0.523
& 1259 & 7.932
& 24   & 0.022
& 24   & 0.034
& 6.96e$-$02
& 3.53 \\

32
& /    & /
& /    & /
& 1077 & 47.79
& /    & /
& 25   & 0.226
& 25   & 0.440
& 5.05e$-$03
& 3.79 \\

64
& /    & /
& /    & /
& 3868 & 2573
& /    & /
& 25   & 0.924
& 25   & 1.821
& 3.16e$-$04
& 4.00 \\

128
& /    & /
& /    & /
& /    & /
& /    & /
& 24   & 7.367
& 24   & 10.79
& 2.07e$-$05
& 3.93 \\

256
& /    & /
& /    & /
& /    & /
& /    & /
& 24   & 79.49
& 23   & 114.9
& 1.23e$-$06
& 4.07 \\

\bottomrule
\end{tabular}
}
\xdef\SharedTableScale{%
  \fpeval{\number\textwidth/\number\wd\ReferenceTableBox}%
}
\scalebox{\SharedTableScale}{\usebox{\ReferenceTableBox}}
\end{table}

Table~\ref{tab:duffy_3d_error_iter_cpu} reports the comparison similar to   
Table~\ref{tab:2d_triangle_omega90} in two dimensions. Indeed, we also observe the good performance of PLCOL and PBCOL, and particularly,   PBCOL captures and resolves the leading corner singularity with the expected order ${\mathcal O}(N^{-4})$ as in the last column of  Table~\ref{tab:2d_triangle_omega90}. 


Finally, we consider \eqref{eq:3D_Helmholtz} on a curved hexahedral domain $\widehat{\Omega}$ with
$f=\cosh(4x)\cosh(4y)\cosh(z)$ and $\gamma=1,$ where  
$\widehat{\Omega}$ is generated from the reference cube $\Omega=(-1,1)^3$ by the Gordon-Hall transformation (see Figure \ref{fig:Duffy_transformations3D} (right)):
\begin{equation}\label{eq:3D_GH}
\mathbf x
=
\mathbf T(\boldsymbol\xi)
=
\big(
r(\xi_3)\xi_1
,\;
r(\xi_3)\xi_2
,\;
\xi_3+
\tfrac{\alpha}{2}(1+\xi_3)
\big(
1-\tfrac{\xi_1^2+\xi_2^2}{2}
\big)\big)^\intercal, \;\;\; 
r(\xi_3)
=
\tfrac{1-\xi_3}{2}
+
\rho\tfrac{1+\xi_3}{2},
\end{equation}
where $\alpha,\rho>0$ are two tuning geometric parameters. 
In the test, we take
\(\alpha=0.85\) and \(
\rho=1.25.\)


Note that unlike the Duffy transformation, the Gordon-Hall mapping \eqref{eq:3D_GH} is regular, so we directly formulate the scheme and construct the preconditioners based on the transformed problem (see   
\eqref{eq:Transformed_PDEs_referenceDomain}) without multiplying the Jacobian (see \eqref{eq:scaled_transformed_PDE}).



 The reference solution is computed by  PBCOL  with \(N=512\). The left and middle panels in Figure~\ref{fig:3D_GordonHall}
display representative solution slices and the pointwise error
distribution, respectively, showing that the numerical solution remains
accurate over the curved domain. More importantly, the right panel indicates 
that the condition numbers of LCOL and BCOL grow approximately as
\(\mathcal{O}(N^4)\) and \(\mathcal{O}(N^8)\), respectively, whereas those of PLCOL and
PBCOL remain nearly constant. This demonstrates the effectiveness of the
proposed preconditioners for the nonseparable coefficients induced by the
Gordon-Hall transformation.

\begin{figure}[htbp]
    \centering
    \includegraphics[width=0.3\linewidth]{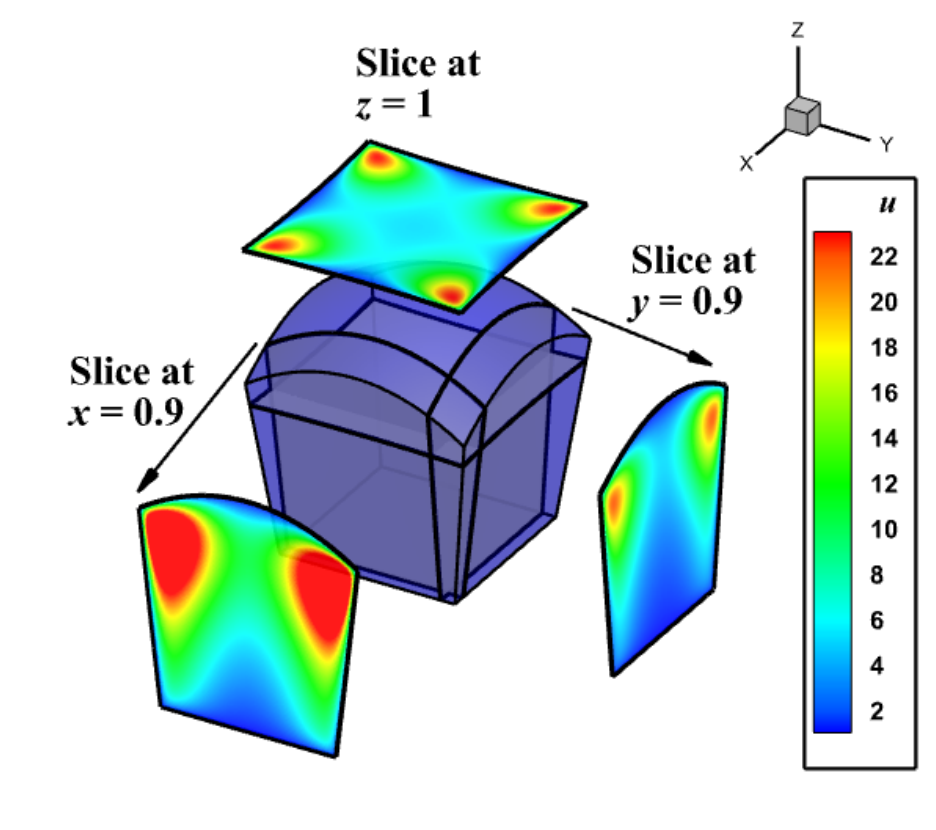}  \quad 
\includegraphics[width=0.3\linewidth]{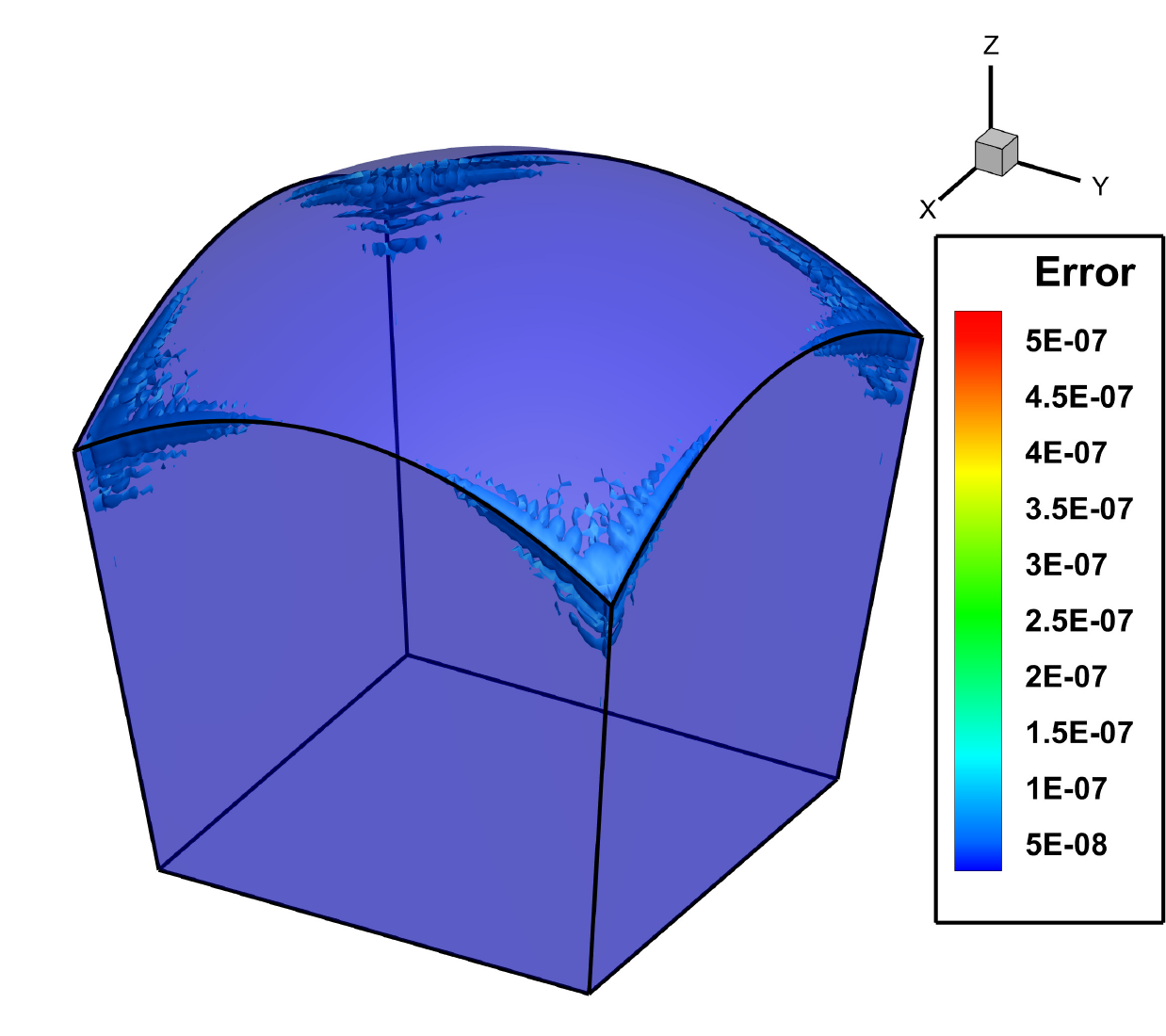}  
\quad \includegraphics[width=0.3\linewidth]{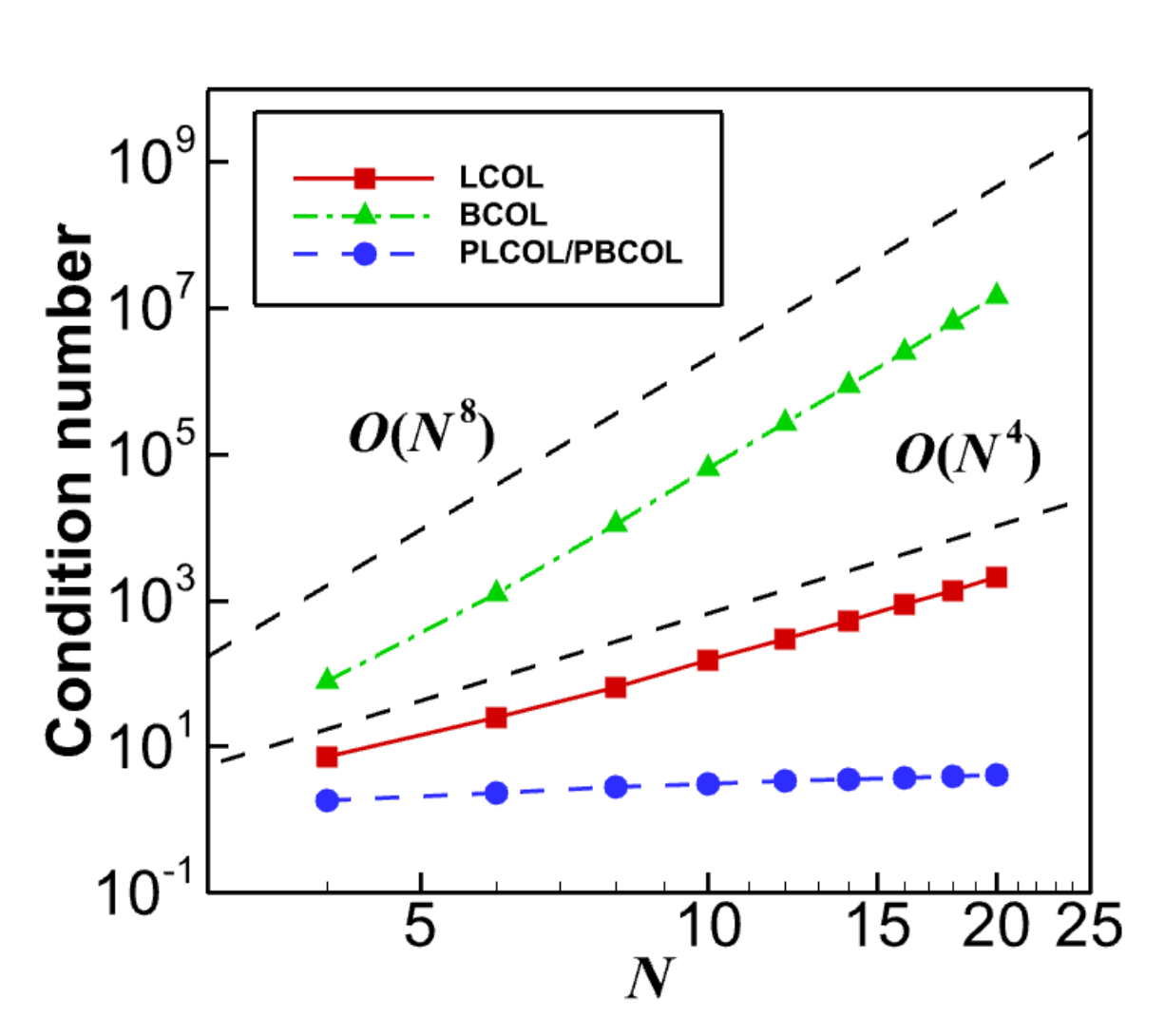}  
    \caption{
Numerical results for the three-dimensional Gordon-Hall problem.
Left: representative solution slices.
Middle: pointwise error distribution.
Right: condition-number growth.
}
\label{fig:3D_GordonHall}
\end{figure}

We provide in Table~\ref{tab:GH_3d_error_iter_cpu} a more quantitative comparison of the collocation schemes. As in the two-dimensional cases, the results again demonstrate the effectiveness of the proposed preconditioning techniques, which accurately resolve the corner singularity within the optimal approximability of  tensorial polynomial bases.   

\begin{table}[!htbp]
\centering
\caption{\footnotesize Errors, GMRES iteration counts, and CPU time (s) for the 3D Gordon-Hall problem.}
\label{tab:GH_3d_error_iter_cpu}
\footnotesize
\setlength{\tabcolsep}{3.0pt}

\scalebox{\SharedTableScale}{%
\begin{tabular}{c cc cc cc cccc}
\toprule
\multirow{2}{*}{$N$}
& \multicolumn{2}{c}{LCOL}
& \multicolumn{2}{c}{BCOL}
& \multicolumn{2}{c}{PLCOL}
& \multicolumn{4}{c}{PBCOL} \\

\cmidrule(lr){2-3}
\cmidrule(lr){4-5}
\cmidrule(lr){6-7}
\cmidrule(lr){8-11}

& Iter & Time
& Iter & Time
& Iter & Time
& Iter & Time & Error & Rate \\
\midrule

8
& 52   & 0.015
& 144  & 0.040
& 18   & 0.243
& 18   & 0.271
& 1.81e$-$01
& - \\

16
& 202  & 0.213
& 1098 & 5.649
& 23   & 0.504
& 23   & 0.521
& 1.88e$-$02
& 3.27 \\

32
& 763  & 25.13
& /    & /
& 25   & 1.204
& 25   & 1.354
& 1.74e$-$03
& 3.43 \\

64
& 2872 & 1394
& /    & /
& 27   & 3.171
& 27   & 3.838
& 1.83e$-$04
& 3.25 \\

128
& /    & /
& /    & /
& 28   & 14.21
& 28   & 18.29
& 1.73e$-$05
& 3.39 \\

256
& /    & /
& /    & /
& 32   & 135.3
& 28   & 152.6
& 1.98e$-$06
& 3.13 \\

\bottomrule
\end{tabular}%
}
\end{table}

\section{Conclusion}\label{sec:conclusion}
In this work, we resolve the long-standing issue of ill-conditioning of multidimensional Legendre collocation methods for second-order elliptic problems. The key to the success of the proposed approach is the stable diagonalisation of the underlying differentiation and integration matrices through a judiciously constructed similar symmetric matrix that properly incorporates the variable coefficients of the principal part.  We justify the effectiveness of the proposed preconditioning technique through spectral and conditioning analyses, and  demonstrate its robustness with extensive numerical tests on a variety of challenging elliptic-type problems.  

This work is our first attempt in this direction, with significant implications and some follow-up issues and aspects worthy of in-depth investigation, for instance, 
\vskip 3pt
\begin{itemize}
\item[(i)] Here, we only considered the second-order elliptic problems with Dirichlet boundary, but the idea and framework can be extended to other boundary conditions, and fourth-order problems (for which the usual collocation method suffers from even severer ill-conditioning). 
\vskip 3pt
\item[(ii)] The preconditioning  on triangles and tetrahedra sheds light on, and may provide a foundation for, effective preconditioning of spectral element methods on unstructured meshes \cite{KarniadakisSherwin2005,Khurana2026LOR}. Indeed, effective preconditioners are needed for  high-order element methods in various formulations (see e.g., \cite{Fortunato2024SurfacePDE,chalmers2018low,DiosadyMurman2019,bello2019scalable,Pablo2022scalable}), and our approach may offer a more natural high-order alternative to the low-order preconditioning strategies employed therein.  
 \end{itemize}
\vskip 3pt
\noindent We shall report further developments along these lines in future work.




\clearpage
\appendix
\renewcommand{\theHsection}{appendix.\Alph{section}}

\section{Proof of Theorem \ref{thm:general_spectral_bound}}\label{app:general_spectral_bound}

In this section, we provide the detailed proof of Theorem~\ref{thm:general_spectral_bound}. We first decompose the Lagrange-collocation operator into its second-order principal part, first-order part, and zeroth-order term. We then establish two auxiliary estimates in two lemmas: the first gives a uniform equivalence between the second-order principal part and the separable reference operator, while the second controls the first-order part in terms of the same reference operator. Finally, these estimates are combined with the stability assumption to derive the uniform spectral bounds for the preconditioned matrix.

We first introduce two auxiliary estimates needed in the proof.
In view of Proposition~\ref{prop:preconditioned_identity}, it is
sufficient to consider
\(\mathbf{A}_{\rm L}\mathbf{P}_{\rm L}\).
We decompose the Lagrange-collocation matrix in
\eqref{eq:general_AL} as
\(
\mathbf{A}_{\rm L}
=
\mathbf{A}_{{\rm L},2}
+
\mathbf{A}_{{\rm L},1}
+
\breve{\mathbf{S}},
\)
where
\begin{subequations}\label{supp-eq:AL_parts}
\begin{align}
\mathbf{A}_{{\rm L},2}
={}&
-\mathbf{C}_{11}
(\mathbf{I}_{N-1}\otimes\mathbf{D}^{(2)})
-2\mathbf{C}_{12}
(\mathbf{D}\otimes\mathbf{D})
-\mathbf{C}_{22}
(\mathbf{D}^{(2)}\otimes\mathbf{I}_{N-1}),
\label{supp-eq:AGL2}
\\
\mathbf{A}_{{\rm L},1}
={}&
\mathbf{Q}_1
(\mathbf{I}_{N-1}\otimes\mathbf{D})
+
\mathbf{Q}_2
(\mathbf{D}\otimes\mathbf{I}_{N-1}).
\label{supp-eq:AGL1}
\end{align}
\end{subequations}
The analysis relies on two auxiliary estimates. The first establishes
the uniform equivalence between the second-order principal part
\(\mathbf{A}_{{\rm L},2}\) and the separable reference operator
\(\mathbf{P}_{\rm L}^{-1}\), while the second controls the first-order
part \(\mathbf{A}_{{\rm L},1}\) in terms of the same reference operator.
Throughout this section, the superscript $*$ denotes complex conjugation,
while the superscript ${\rm H}$ denotes the conjugate transpose.

\subsection{Uniform equivalence between \(\mathbf{A}_{{\rm L},2}\) and \(\mathbf{P}_{\rm L}^{-1}\)}

We first establish a uniform equivalence between the full second-order
principal part \(\mathbf{A}_{{\rm L},2}\) and the separable reference
operator \(\mathbf{P}_{\rm L}^{-1}\). This estimate provides the main
control of the principal part in the subsequent spectral analysis.

\begin{lem}\label{lem:principal_uniform_equivalence}
For the general variable-coefficient problem
\eqref{eq:general_elliptic}, define
\begin{equation}\label{defR}
\mathbf A_0(x,y)
=
\operatorname{diag}
\bigl(
\bar a_{11}(x),\bar a_{22}(y)
\bigr),
\quad
\mathbf R(x,y)
=
\mathbf A_0^{-\frac12}
\mathbf A
\mathbf A_0^{-\frac12}.
\end{equation}
Assume that both \(\mathbf A\) and \(\mathbf A_0\) are uniformly
positive definite. Then there exist positive constants
\(c_1, c_2\), independent of \(N\), such that
\begin{equation}\label{eq:principal_uniform_equivalence}
c_1\|\mathbf P_{\rm L}^{-1}\boldsymbol v\|_{{\mathbf H}}
\leq
\|\mathbf A_{{\rm L},2}\boldsymbol v\|_{{\mathbf H}}
\leq
c_2\|\mathbf P_{\rm L}^{-1}\boldsymbol v\|_{{\mathbf H}},
\quad
\forall\,\boldsymbol v\in\mathbb C^{(N-1)^2},
\end{equation}
where
\(
{\mathbf H}
=
(\mathbf W\mathbf C_b^{-1})
\otimes
(\mathbf W\mathbf C_a^{-1})\) and \(
\|\boldsymbol v\|_{{\mathbf H}}
=
\sqrt{\boldsymbol v^{\rm H}{\mathbf H}\boldsymbol v}.
\)
\end{lem}

\begin{proof}
    By the definition of \(\mathbf R(x,y)\) in \eqref{defR}, we know that it is SPD, thus
if we denote the two eigenvalues of $\mathbf R$ as $\mu_1(x,y)$ and $\mu_2(x,y)$, then there exist positive constants
\(
m,M
\)
such that $m<\mu_1,\mu_2<M$ holds for all $(x,y)\in\Omega$. Here we define
\[
\varphi(x,y)
:=
\frac{\operatorname{tr}\mathbf R(x,y)}
{\|\mathbf R(x,y)\|_F^2},
\quad{\rm and}\quad
\rho_A
:=
\sup_{(x,y)\in\Omega}
\|\varphi\mathbf R-\mathbf I_2\|_F,
\]
where $\|\cdot\|_F$ is the Frobenius norm. Since \(\mathbf R(x,y)\) is SPD with eigenvalues \(\mu_1\) and \(\mu_2\), it is clear that
\begin{equation}\label{varphibound}
\operatorname{tr}\mathbf R=\mu_1+\mu_2,\quad \|\mathbf R(x,y)\|_F^2=\mu_1^2+\mu_2^2\quad \Rightarrow \quad \varphi=\tfrac{\mu_1+\mu_2}{\mu_1^2+\mu_2^2}.
\end{equation}
Then, the eigenvalues of the shifted matrix \(\varphi\mathbf R-\mathbf I_2\) are precisely \(\varphi\mu_1-1\) and \(\varphi\mu_2-1\) and
\[
\|\varphi\mathbf R-\mathbf I_2\|_F^2
=
(\varphi\mu_1-1)^2+(\varphi\mu_2-1)^2=\dfrac{(\mu_1-\mu_2)^2}{(\mu_1^2+\mu_2^2)},
\]
which leads to $\rho_A<1$. 

For
\(
\boldsymbol v\in\mathbb C^{(N-1)^2},
\)
let \(v_N\in\mathbb P_N\otimes\mathbb P_N\cap H_0^1(\Omega)\)
be the tensor-product Lagrange polynomial associated with
\(\boldsymbol v\) through
\eqref{eq:2d_Lagrange_Birkhoff_expansions}. We further define
\begin{equation}\label{defHoh}
\mathscr H_{0,h}v_N
=
\begin{bmatrix}
\bar a_{11}v_{N,xx}
&
\sqrt{\bar a_{11}\bar a_{22}}\,v_{N,xy}
\\
\sqrt{\bar a_{11}\bar a_{22}}\,v_{N,xy}
&
\bar a_{22}v_{N,yy}
\end{bmatrix},
\end{equation}
and for a matrix-valued grid function \(\mathbf G\), set
\[
\|\mathbf G\|_{{\mathbf H},F}^2
:=
\sum_{i,j=1}^{N-1}
\frac{\omega_i\omega_j}
{\bar a_{11}(x_i)\bar a_{22}(y_j)}
\|\mathbf G(x_i,y_j)\|_F^2.
\]
By the definition of \(\mathbf P_{\rm L}^{-1}\) in \eqref{def:PL_PB_var_coeff}, at each interior node
\((x_i,y_j)\), we have
\begin{equation}\label{invPLnodes}
(\mathbf P_{\rm L}^{-1}\boldsymbol v)_{ij}
=
-
a_i v_{N,xx}(x_i,y_j)
-
b_j v_{N,yy}(x_i,y_j),
\end{equation}
where \(
a_i:=\bar a_{11}(x_i),
b_j:=\bar a_{22}(y_j).
\)
By the definition of \({\mathbf H}\),
it is clear that \({\mathbf H}\) is diagonal, and its diagonal entry corresponding to the tensor-product
node \((x_i,y_j)\) is
\(
\frac{\omega_i\omega_j}{a_i b_j}.
\)
Therefore,
\begin{equation}\label{PLvnormH}
\begin{aligned}
\|\mathbf P_{\rm L}^{-1}\boldsymbol v\|_{{\mathbf H}}^2
&=
\sum_{i,j=1}^{N-1}
\frac{\omega_i\omega_j}{a_i b_j}
\left|
a_i v_{N,xx}(x_i,y_j)
+
b_j v_{N,yy}(x_i,y_j)
\right|^2\\
&=
\sum_{i,j=1}^{N-1}
\omega_i\omega_j
\Big(
\frac{a_i}{b_j}|v_{N,xx}|^2
+
\frac{b_j}{a_i}|v_{N,yy}|^2
+
2\operatorname{Re}
\bigl(v_{N,xx}v_{N,yy}^{*}\bigr)
\Big)(x_i,y_j).
\end{aligned}
\end{equation}
On the other hand, by the definition of $\mathscr H_{0,h}v_N$ in \eqref{defR} and \eqref{defHoh}, we obtain that at each interior node
\((x_i,y_j)\),
\[
\|\mathscr H_{0,h}v_N(x_i,y_j)\|_F^2
=
a_i^2|v_{N,xx}|^2
+
2a_ib_j|v_{N,xy}|^2
+
b_j^2|v_{N,yy}|^2,
\]
which leads to
\begin{equation}\label{HohvNH}
\|\mathscr H_{0,h}v_N\|_{{\mathbf H},F}^2
=
\sum_{i,j=1}^{N-1}
\omega_i\omega_j
\Big(
\frac{a_i}{b_j}|v_{N,xx}|^2
+
2|v_{N,xy}|^2
+
\frac{b_j}{a_i}|v_{N,yy}|^2
\Big)(x_i,y_j).
\end{equation}
Since
\(
v_N|_{\partial\Omega}=0,
\)
integration by parts twice gives
\begin{equation}\label{integrals}
\int_\Omega |v_{N,xy}|^2\,dxdy
=\Re
\int_\Omega
v_{N,xx}v_{N,yy}^{*}\,dxdy=
\Re
\sum_{i,j=1}^{N-1}
\omega_i\omega_j
v_{N,xx}(x_i,y_j)
v_{N,yy}^{*}(x_i,y_j),
\end{equation}
where the first equality follows from integration by parts and the homogeneous boundary condition, while the second follows from the exactness of the tensor-product LGL quadrature.
In addition, the boundary-node contributions of
\(|v_{N,xy}|^2\) are nonnegative, and therefore
\[
\sum_{i,j=1}^{N-1}
\omega_i\omega_j
|v_{N,xy}(x_i,y_j)|^2
\leq
\int_\Omega |v_{N,xy}|^2\,dxdy=\Re 
\sum_{i,j=1}^{N-1}
\omega_i\omega_j
v_{N,xx}(x_i,y_j)
v_{N,yy}^{*}(x_i,y_j).
\]
Consequently, by \eqref{PLvnormH}--\eqref{HohvNH}, we have
\begin{equation}\label{H0hvNHFineq}
    \|\mathscr H_{0,h}v_N\|_{{\mathbf H},F}
\leq
\|\mathbf P_{\rm L}^{-1}\boldsymbol v\|_{{\mathbf H}}.
\end{equation}

Moreover, we define \(
\mathbf\varphi
=
\operatorname{diag}
(
\varphi(x_i,y_j)
)_{i,j=1}^{N-1}.
\) By \eqref{defR} and \eqref{defHoh} one has
\[
\operatorname{tr}
\left(
(\varphi\mathbf R-\mathbf I_2)^\intercal
\mathscr H_{0,h}v_N
\right)
=
(\varphi a_{11}-\bar a_{11})v_{N,xx}
+
2\varphi a_{12}v_{N,xy}
+
(\varphi a_{22}-\bar a_{22})v_{N,yy}.
\]
According to the definition of \(\mathbf A_{{\rm L},2}\) in \eqref{supp-eq:AGL2}, at each interior LGL node one has
\[
(\mathbf A_{{\rm L},2}\boldsymbol v)_{ij}
=
-\bigl(
a_{11}v_{N,xx}
+
2a_{12}v_{N,xy}
+
a_{22}v_{N,yy}
\bigr)(x_i,y_j),
\]
Therefore, combining \eqref{invPLnodes}, we obtain that
\[
\begin{aligned}
(
\mathbf\varphi\mathbf A_{{\rm L},2}\boldsymbol v
-
\mathbf P_{\rm L}^{-1}\boldsymbol v
)_{ij}
=
-[
(\varphi a_{11}-\bar a_{11})v_{N,xx}
+
2\varphi a_{12}v_{N,xy}+
(\varphi a_{22}-\bar a_{22})v_{N,yy}
](x_i,y_j),
\end{aligned}
\]
which implies that $-(
\mathbf\varphi\mathbf A_{{\rm L},2}\boldsymbol v
-
\mathbf P_{\rm L}^{-1}\boldsymbol v
)_{ij}
=\operatorname{tr}
\left(
(\varphi\mathbf R-\mathbf I_2)^\intercal
\mathscr H_{0,h}v_N
\right).$
By the Frobenius Cauchy--Schwarz inequality, we have
\[
\left|
\operatorname{tr}
\left(
(\varphi\mathbf R-\mathbf I_2)^\intercal
\mathscr H_{0,h}v_N
\right)
\right|
\leq
\|\varphi\mathbf R-\mathbf I_2\|_F
\|\mathscr H_{0,h}v_N\|_F.
\]
By the definition of \(\rho_A\), it leads to
\begin{equation}\label{rhoAineq}
\left|
(
\mathbf\varphi\mathbf A_{{\rm L},2}\boldsymbol v
-
\mathbf P_{\rm L}^{-1}\boldsymbol v
)_{ij}
\right|
\leq
\rho_A
\|\mathscr H_{0,h}v_N(x_i,y_j)\|_F.
\end{equation}
Squaring both sides of \eqref{rhoAineq}, multiplying by the weight
\(
\frac{\omega_i\omega_j}{a_i b_j}
\),
and summing over \(i,j\), we obtain
\[
\begin{aligned}
\|
\mathbf\varphi\mathbf A_{{\rm L},2}\boldsymbol v
-
\mathbf P_{\rm L}^{-1}\boldsymbol v
\|_{\mathbf H}^2
\leq
\rho_A^2
\sum_{i,j=1}^{N-1}
\frac{\omega_i\omega_j}{a_i b_j}
\|\mathscr H_{0,h}v_N(x_i,y_j)\|_F^2
=
\rho_A^2
\|\mathscr H_{0,h}v_N\|_{\mathbf H,F}^2.
\end{aligned}
\]
Using \eqref{H0hvNHFineq} we finally obtain
\begin{equation}\label{rhoAineq2}
\|
\mathbf\varphi\mathbf A_{{\rm L},2}\boldsymbol v
-
\mathbf P_{\rm L}^{-1}\boldsymbol v
\|_{\mathbf H}
\leq
\rho_A
\|\mathbf P_{\rm L}^{-1}\boldsymbol v\|_{\mathbf H}.
\end{equation}
By \eqref{varphibound}, we find that $M^{-1}<\varphi<m^{-1}$, then the definition of $\mathbf\varphi$ leads to 
\begin{equation}\label{varphiALbound}
M^{-1}\|\mathbf{A}_{{\rm L},2}\bm{v}\|_{\mathbf{H}} \leq \|\mathbf{\varphi}\mathbf{A}_{{\rm L},2}\bm{v}\|_{\mathbf{H}} \leq m^{-1}\|\mathbf{A}_{{\rm L},2}\bm{v}\|_{\mathbf{H}}.
\end{equation}
Combining \eqref{rhoAineq2}--\eqref{varphiALbound} and the fact of $\rho_A<1$ we proved above, we obtain
\[
m(1-\rho_A)\|\mathbf{P}_{\rm L}^{-1}\bm{v}\|_{\mathbf{H}}
\leq
\|\mathbf{A}_{{\rm L},2}\bm{v}\|_{\mathbf{H}}
\leq
M(1+\rho_A)\|\mathbf{P}_{\rm L}^{-1}\bm{v}\|_{\mathbf{H}},
\]
which ends the proof.

\end{proof}

\subsection{Bound of the operator \(\mathbf{A}_{{\rm L},1}\)}\label{prooflemma2}

We next derive a relative bound for the first-order part
\(\mathbf{A}_{{\rm L},1}\) in terms of the reference operator
\(\mathbf{P}_{\rm L}^{-1}\) and the weighted discrete norm. This
estimate is used to control the lower-order contribution in the
spectral analysis.

\begin{lem}\label{lem:first_order_relative_bound}
For the general variable-coefficient problem
\eqref{eq:general_elliptic}, assume that
\[
0<\alpha_0
\leq
\bar a_{11}(x),\,\bar a_{22}(y)
\leq
\beta_0<\infty.
\] 
Let
\(
q_\infty
=
\|
\sqrt{q_1^2+q_2^2}
\|_{L^\infty(\Omega)}.
\)
Then, for every
\(
\boldsymbol v\in\mathbb C^{(N-1)^2},
\)
\begin{equation}\label{eq:first_order_relative_bound}
\|\mathbf A_{{\rm L},1}\boldsymbol v\|_{{\mathbf H}}
\leq
\frac{q_\infty}{\sqrt{\alpha_0}}
\|\mathbf P_{\rm L}^{-1}\boldsymbol v\|_{{\mathbf H}}^{\frac12}\, 
\|\boldsymbol v\|_{{\mathbf H}}^{\frac12}.
\end{equation}
\end{lem}

\begin{proof}
For
\(
\boldsymbol v\in\mathbb C^{(N-1)^2},
\) let
\(
v_N\in\mathbb P_N\otimes\mathbb P_N\cap H_0^1(\Omega)
\)
be the tensor-product Lagrange polynomial associated with
\(\boldsymbol v\) through
\eqref{eq:2d_Lagrange_Birkhoff_expansions}. At every interior
LGL node, by the definition of $\mathbf A_{\rm L}$ in \eqref{supp-eq:AGL1}, we have
\[
(\mathbf A_{{\rm L},1}\boldsymbol v)_{ij}
=
q_1(x_i,y_j)v_{N,x}(x_i,y_j)
+
q_2(x_i,y_j)v_{N,y}(x_i,y_j).
\]
Since \(\mathbf{H}=(\mathbf{W}\mathbf{C}_b^{-1})\otimes(\mathbf{W}\mathbf{C}_a^{-1})\),
and denoting \(a_i=\bar a_{11}(x_i)\) and \(b_j=\bar a_{22}(y_j)\) as the proof of Lemma \ref{lem:principal_uniform_equivalence}, we have
\begin{equation}\label{AL1vHnorm}
\|\mathbf{A}_{{\rm L},1}\bm{v}\|_{\mathbf{H}}^2
=
\sum_{i,j=1}^{N-1}
\frac{\omega_i\omega_j}{a_i b_j}
\left|
q_1(x_i,y_j)v_{N,x}(x_i,y_j)
+
q_2(x_i,y_j)v_{N,y}(x_i,y_j)
\right|^2.
\end{equation}
By the Cauchy--Schwarz inequality, it leads to
\begin{equation}\label{qCSineq}
|q_1v_{N,x}+q_2v_{N,y}|^2
\leq
(q_1^2+q_2^2)
\left(|v_{N,x}|^2+|v_{N,y}|^2\right)
\leq
q_\infty^2
\left(|v_{N,x}|^2+|v_{N,y}|^2\right).
\end{equation}
Therefore, combining \eqref{AL1vHnorm}--\eqref{qCSineq}, we obtain that
\[
\|\mathbf{A}_{{\rm L},1}\bm{v}\|_{\mathbf{H}}^2
\leq
q_\infty^2
\sum_{i,j=1}^{N-1}
\frac{\omega_i\omega_j}{a_i b_j}
\left(
|v_{N,x}|^2+|v_{N,y}|^2
\right).
\]
Using \(a_i,b_j\geq\alpha_0\), we have
\begin{equation}\label{AL1vHnormineq}
\|\mathbf{A}_{{\rm L},1}\bm{v}\|_{\mathbf{H}}^2
\leq
\frac{q_\infty^2}{\alpha_0}
\sum_{i,j=1}^{N-1}
\omega_i\omega_j
\left(
\frac{|v_{N,x}(x_i,y_j)|^2}{\bar a_{22}(y_j)}
+
\frac{|v_{N,y}(x_i,y_j)|^2}{\bar a_{11}(x_i)}
\right).
\end{equation}

On the other hand, from the definition of $\mathbf{P}_{\rm L}$ in \eqref{def:PL_PB_var_coeff} and
\(
\mathbf{H}
=
(\mathbf{W}\mathbf{C}_b^{-1})
\otimes
(\mathbf{W}\mathbf{C}_a^{-1}),
\)
we obtain
\[
\begin{aligned}
(\mathbf{P}_{\rm L}^{-1}\bm{v},\bm{v})_{\mathbf{H}}
={}&
-\bm{v}^{\rm H}
[
(\mathbf{W}\mathbf{C}_b^{-1})
\otimes
(\mathbf{W}\mathbf{D}^{(2)})
]\bm{v}
-
\bm{v}^{\rm H}
[
(\mathbf{W}\mathbf{D}^{(2)})
\otimes
(\mathbf{W}\mathbf{C}_a^{-1})
]\bm{v}.
\end{aligned}
\]
For each fixed \(y_j\), since \(v_N\) satisfies the homogeneous Dirichlet
boundary condition, \(v_N(\pm1,y_j)=0\). The exactness of the LGL quadrature and integration by parts implies that
\[
-\sum_{i=1}^{N-1}
\omega_i v_N^*(x_i,y_j)v_{N,xx}(x_i,y_j)
=
\int_{-1}^{1}|v_{N,x}(x,y_j)|^2\,{\rm d}x
=
\sum_{i=0}^{N}
\omega_i|v_{N,x}(x_i,y_j)|^2.
\]
Similarly, for each fixed \(x_i\),
\[
-\sum_{j=1}^{N-1}
\omega_j v_N^*(x_i,y_j)v_{N,yy}(x_i,y_j)
=
\sum_{j=0}^{N}
\omega_j|v_{N,y}(x_i,y_j)|^2.
\]
Therefore,
\begin{equation}\label{PLinvvineq}
(\mathbf{P}_{\rm L}^{-1}\bm{v},\bm{v})_{\mathbf{H}}
=
\sum_{i,j=0}^{N}
\omega_i\omega_j
\left(
\frac{|v_{N,x}(x_i,y_j)|^2}{\bar a_{22}(y_j)}
+
\frac{|v_{N,y}(x_i,y_j)|^2}{\bar a_{11}(x_i)}
\right).
\end{equation}
Since all the terms in the summation are nonnegative, combining \eqref{AL1vHnormineq}--\eqref{PLinvvineq} yields
\[
\begin{aligned}
\|\mathbf{A}_{{\rm L},1}\bm{v}\|_{\mathbf{H}}^2
\leq
\frac{q_\infty^2}{\alpha_0}
(\mathbf{P}_{\rm L}^{-1}\bm{v},\bm{v})_{\mathbf{H}}
\leq
\frac{q_\infty^2}{\alpha_0}
\|\mathbf{P}_{\rm L}^{-1}\bm{v}\|_{\mathbf{H}}
\|\bm{v}\|_{\mathbf{H}},
\end{aligned}
\]
where the last inequality uses the Cauchy--Schwarz inequality with respect to the \(\mathbf{H}\)-inner product. This ends the proof.
\end{proof}

\subsection{Detailed proof}

With the estimates for the second- and first-order parts established
above, we now prove the uniform spectral bounds for the preconditioned
collocation matrix.

\begin{proof}
    By Proposition~\ref{prop:preconditioned_identity}, it suffices to
consider \(\mathbf A_{\rm L}\mathbf P_{\rm L}\). We first consider the upper bound in \eqref{eq:general_preconditioned_annular_bound}. Arguing as in the proof of
Theorem~\ref{thm:separable_spectral_bound}, let
\(
\mathbf K_{\rm L}=\mathbf P_{\rm L}^{-1}
\)
and consider an eigenpair
\(
\mathbf A_{\rm L}\boldsymbol v
=
\lambda\mathbf K_{\rm L}\boldsymbol v.
\) 
By Lemmas~\ref{lem:principal_uniform_equivalence} and
\ref{lem:first_order_relative_bound}, we have
\begin{equation}\label{lemmasresult}
    \|\mathbf{A}_{{\rm L},2}\bm{v}\|_{\mathbf{H}}
\leq
C\|\mathbf{K}_{\rm L}\bm{v}\|_{\mathbf{H}},
\quad
\|\mathbf{A}_{{\rm L},1}\bm{v}\|_{\mathbf{H}}
\leq
\frac{q_\infty}{\sqrt{\alpha_0}}
\|\mathbf{K}_{\rm L}\bm{v}\|_{\mathbf{H}}^{\frac12}
\|\bm{v}\|_{\mathbf{H}}^{\frac12}.
\end{equation}
Here we also set 
\(
\widehat{\mathbf K}_{\rm L}
:=\mathbf H^{\frac12}\mathbf K_{\rm L}\mathbf H^{-\frac12}
\)
which is SPD and has the same spectrum as
\(\mathbf K_{\rm L}\). Let \(\bm{w}=\mathbf{H}^{\frac12}\bm{v}\). Then it is clear that
\[
\|\mathbf{K}_{\rm L}\bm{v}\|_{\mathbf{H}}
=
\|\mathbf{H}^{\frac12}\mathbf{K}_{\rm L}\bm{v}\|_2
=
\|\widehat{\mathbf{K}}_{\rm L}\bm{w}\|_2
\geq
\lambda_{\min}(\widehat{\mathbf{K}}_{\rm L})\|\bm{w}\|_2.
\]
Therefore, \eqref{spectralKL} and the fact of \(\|\bm{w}\|_2=\|\bm{v}\|_{\mathbf{H}}\) lead to
\begin{equation}\label{klhbound}
    \|\mathbf{K}_{\rm L}\bm{v}\|_{\mathbf{H}}
\geq
\frac{\alpha_0\pi^2}{2}
\|\bm{v}\|_{\mathbf{H}}.
\end{equation}
Thus combining \eqref{lemmasresult}--\eqref{klhbound}, we have
\begin{equation}\label{al1vbound}
    \|\mathbf{A}_{{\rm L},1}\bm{v}\|_{\mathbf{H}}
\leq
\frac{\sqrt{2}q_\infty}{\alpha_0\pi}
\|\mathbf{K}_{\rm L}\bm{v}\|_{\mathbf{H}}.
\end{equation}
Since \(\breve{\mathbf S}\) is a diagonal matrix composed of
\(s(x_i,y_j)\) and \(0<s\leq s_{\max}\), by \eqref{klhbound} we have
\begin{equation}\label{Svbound}
    \|\breve{\mathbf S}\bm{v}\|_{\mathbf{H}}
\leq
s_{\max}\|\bm{v}\|_{\mathbf{H}}.
\end{equation}
Recall that 
\(
\mathbf A_{\rm L}
=
\mathbf A_{{\rm L},2}
+
\mathbf A_{{\rm L},1}
+
\breve{\mathbf S},
\)
combining the above three estimates \eqref{lemmasresult}, \eqref{al1vbound} and \eqref{Svbound} yields
\[
\|\mathbf{A}_{\rm L}\bm{v}\|_{\mathbf{H}}
\leq
\left(
C+
\frac{\sqrt{2}q_\infty}{\alpha_0\pi}
+
\frac{2s_{\max}}{\alpha_0\pi^2}
\right)
\|\mathbf{K}_{\rm L}\bm{v}\|_{\mathbf{H}},
\]
which proves the upper bound in \eqref{eq:general_preconditioned_annular_bound}.

For the lower bound, we argue by contradiction.
Suppose that the eigenvalues of the preconditioned matrices are not
uniformly bounded away from the origin as $N$ increases.
Then there exists a sequence of polynomial degrees $N_k\to\infty$
and corresponding eigenpairs
\[
\mathbf A_{\rm L}^k \bm v_k=\lambda_k\mathbf K_{\rm L}^k\bm  v_k,
\quad |\lambda_k|\to0.
\]
normalized by
\(
\|\mathbf K_{\rm L}^{k}\boldsymbol v_k\|_{{\mathbf H}_k}
=1.
\)
Then one has
\(
\|\mathbf A_{\rm L}^{k}\boldsymbol v_k\|_{{\mathbf H}_k}
=
|\lambda_k|
\to0.
\)

For
\(
\boldsymbol v\in\mathbb C^{(N-1)^2},
\) let
\(
v_N\in\mathbb P_N\otimes\mathbb P_N\cap H_0^1(\Omega)
\)
be the tensor-product Lagrange polynomial associated with
\(\boldsymbol v\) through
\eqref{eq:2d_Lagrange_Birkhoff_expansions}.  Because the bound \eqref{klhbound} holds for any $k$, then by 
\(
\|\mathbf K_{\rm L}^{k}\boldsymbol v_k\|_{{\mathbf H}_k}
=1
\)
and the Cauchy--Schwarz inequality, we have
\begin{equation}\label{disenergybound}
\|\boldsymbol v_k\|_{{\mathbf H}_k}
\leq
\frac{2}{\alpha_0\pi^2},\quad {\rm and} \quad (\mathbf K_{\rm L}^{k}\boldsymbol v_k,\boldsymbol v_k
)_{{\mathbf H}_k}
\leq \frac{2}{\alpha_0\pi^2}.
\end{equation}
From \eqref{PLinvvineq}, we have
\[
(\mathbf{K}_{\rm L}^k\bm{v}_k,\bm{v}_k)_{\mathbf{H}_k}
=
\sum_{i,j=0}^{N_k}
\omega_i\omega_j
\left(
\frac{|v_{k,x}(x_i,y_j)|^2}{\bar a_{22}(y_j)}
+
\frac{|v_{k,y}(x_i,y_j)|^2}{\bar a_{11}(x_i)}
\right).
\]
Since \(\bar a_{11}(x),\bar a_{22}(y)\leq\beta_0\), it leads to
\[
(\mathbf{K}_{\rm L}^k\bm{v}_k,\bm{v}_k)_{\mathbf{H}_k}
\geq
\frac{1}{\beta_0}
\sum_{i,j=0}^{N_k}
\omega_i\omega_j
\left(
|v_{k,x}(x_i,y_j)|^2
+
|v_{k,y}(x_i,y_j)|^2
\right).
\]
Thus by \eqref{disenergybound} it follows that
\[
\sum_{i,j=0}^{N_k}
\omega_i\omega_j
\left(
|v_{k,x}(x_i,y_j)|^2
+
|v_{k,y}(x_i,y_j)|^2
\right)
\leq \frac{2\beta_0}{\alpha_0\pi^2}.
\]
Moreover, by the uniform equivalence between the tensor-product LGL discrete norm and the continuous \(L^2\)-norm, there exists a constant \(C_{\rm LGL}\) independent of \(N_k\), such that
\[
\|v_{k,x}\|_{L^2(\Omega)}^2
\leq
C_{\rm LGL}
\sum_{i,j=0}^{N_k}
\omega_i\omega_j
|v_{k,x}(x_i,y_j)|^2,
\quad
\|v_{k,y}\|_{L^2(\Omega)}^2
\leq
C_{\rm LGL}
\sum_{i,j=0}^{N_k}
\omega_i\omega_j
|v_{k,y}(x_i,y_j)|^2.
\]
Consequently, there exists a constant $C$ independent of \(N_k\) such that
\begin{equation}\label{poincare}
\|\nabla v_k\|_{L^2(\Omega)}^2
=
\|v_{k,x}\|_{L^2(\Omega)}^2
+
\|v_{k,y}\|_{L^2(\Omega)}^2
\leq C.
\end{equation}
Thus, by the Poincar\'e inequality,
\(\{v_k\}\) is uniformly bounded in \(H_0^1(\Omega)\).
Hence, after passing to a subsequence, one has
\[
v_k\rightharpoonup v
\quad\text{in }H_0^1(\Omega),
\quad
v_k\to v
\quad\text{in }L^2(\Omega).
\]

We then define the standard tensor-product LGL discrete inner product by
\begin{equation}\label{LGLinner}
\langle \bm f,\bm g\rangle_{N_k}
=
\sum_{i,j=0}^{N_k}
\omega_i\omega_j f(x_i,y_j)g^*(x_i,y_j),
\end{equation}
where $\bm f,\bm g$ are the vector of the values of \(f,g\) at the tensor-product LGL nodes.
Let
\(
\mathbf{C}_k
=
\mathbf{C}_b^k\otimes\mathbf{C}_a^k.
\)
Since
\(
\mathbf{H}_k
=
(\mathbf{W}(\mathbf{C}_b^k)^{-1})
\otimes
(\mathbf{W}(\mathbf{C}_a^k)^{-1}),
\)
we have
\(
\mathbf{H}_k\mathbf{C}_k
=
\mathbf{W}\otimes\mathbf{W}.
\)
Let \(\phi\) be an arbitrary function in \(C_0^\infty(\Omega)\)
and let \(\bm{\phi}_k\) denote the vector of the values of \(\phi\) at the
interior tensor-product LGL nodes. 
Then by the definitions of \(\mathbf{H}_k\) and \(\mathbf{C}_k\), we have
\begin{equation}\label{ALvkHk}
(
\mathbf{A}_{\rm L}^k\bm{v}_k,
\mathbf{C}_k\bm{\phi}_k
)_{\mathbf{H}_k}
=
\bm{\phi}_k^*(\mathbf{W}\otimes\mathbf{W})(\mathbf{A}_{\rm L}^k\bm{v}_k)
=
\sum_{i,j=1}^{N_k-1}
\omega_i\omega_j
(\mathbf{A}_{\rm L}^k\bm{v}_k)_{ij}
\phi^*(x_i,y_j).
\end{equation}
Since \(\phi\) vanishes at all boundary LGL
nodes,
then combining the definition of the discrete LGL inner product \eqref{LGLinner} and \eqref{ALvkHk}, it gives
\begin{equation}\label{NkHkeq}
\langle
\mathbf{A}_{\rm L}^k\bm{v}_k,\bm{\phi}_k
\rangle_{N_k}
=
(
\mathbf{A}_{\rm L}^k\bm{v}_k,
\mathbf{C}_k\bm{\phi}_k
)_{\mathbf{H}_k}.
\end{equation}
Moreover, we have
\[
\begin{aligned}
\|\mathbf{C}_k\bm{\phi}_k\|_{\mathbf{H}_k}^2
&=
\sum_{i,j=1}^{N_k-1}
\omega_i\omega_j
\bar a_{11}(x_i)\bar a_{22}(y_j)
|\phi(x_i,y_j)|^2
\leq
\beta_0^2
\sum_{i,j=0}^{N_k}
\omega_i\omega_j
|\phi(x_i,y_j)|^2
\leq
4\beta_0^2
\|\phi\|_{L^\infty(\Omega)}^2.
\end{aligned}
\]
Recall the assumption 
\(
\|\mathbf A_{\rm L}^{k}\boldsymbol v_k\|_{{\mathbf H}_k}
=
|\lambda_k|
\to0.
\)
Therefore, by \eqref{NkHkeq} and using the Cauchy--Schwarz inequality, we have
\begin{equation}\label{ALvkinf}
\left|
\langle \mathbf{A}_{\rm L}^k\bm{v}_k,\bm{\phi}_k\rangle_{N_k}
\right|
\leq
\|\mathbf{A}_{\rm L}^k\bm{v}_k\|_{\mathbf{H}_k}
\|\mathbf{C}_k\bm{\phi}_k\|_{\mathbf{H}_k}
\leq
2\beta_0
\|\phi\|_{L^\infty(\Omega)}
\|\mathbf{A}_{\rm L}^k\bm{v}_k\|_{\mathbf{H}_k}
\longrightarrow 0.
\end{equation}

For any continuous function \(g\) on \(\overline{\Omega}\), define the
tensor-product LGL interpolation operator \(\mathcal{I}_k\) by
\[
\mathcal{I}_k g(x,y)
=
\sum_{i,j=0}^{N_k}
g(x_i,y_j)
\ell_i(x)\ell_j(y)
\in
\mathbb{P}_{N_k}\otimes\mathbb{P}_{N_k}.
\]
Since \(\phi\in C_0^\infty(\Omega)\),
the functions \(a_{11}\phi\), \(a_{12}\phi\), and \(a_{22}\phi\)
vanish on \(\partial\Omega\). By the discrete summation-by-parts
identities, we have
\begin{equation}\label{a11inte}
-\langle
a_{11}v_{k,xx},\phi
\rangle_{N_k}
=
\left\langle
v_{k,x},
\partial_x\mathcal{I}_k(a_{11}\phi)
\right\rangle_{N_k},
\quad
-\langle
a_{22}v_{k,yy},\phi
\rangle_{N_k}
=
\left\langle
v_{k,y},
\partial_y\mathcal{I}_k(a_{22}\phi)
\right\rangle_{N_k}.
\end{equation}
For the mixed derivative term, applying the summation-by-parts
identity in the two coordinate directions gives
\begin{equation}\label{a12inte}
-2\langle
a_{12}v_{k,xy},\phi
\rangle_{N_k}
=
\left\langle
v_{k,x},
\partial_y\mathcal{I}_k(a_{12}\phi)
\right\rangle_{N_k}
+
\left\langle
v_{k,y},
\partial_x\mathcal{I}_k(a_{12}\phi)
\right\rangle_{N_k}.
\end{equation}
Therefore, using \eqref{a11inte}--\eqref{a12inte}, we have
\[
\begin{aligned}
\langle
\mathbf{A}_{\rm L}^k\bm{v}_k,\bm{\phi}_k
\rangle_{N_k}
={}&
\left\langle
v_{k,x},
\partial_x\mathcal{I}_k(a_{11}\phi)
+
\partial_y\mathcal{I}_k(a_{12}\phi)
\right\rangle_{N_k}
+
\left\langle
v_{k,y},
\partial_x\mathcal{I}_k(a_{12}\phi)
+
\partial_y\mathcal{I}_k(a_{22}\phi)
\right\rangle_{N_k}
\\
&+
\left\langle
q_1v_{k,x}+q_2v_{k,y}+sv_k,
\phi
\right\rangle_{N_k}.
\end{aligned}
\]
Thus, the right-hand side is precisely the discrete weak form
associated with \(\mathcal{L}[v_k]\).
Hence, \eqref{ALvkinf} leads to
\(
\langle \mathcal{L}[v_k],\phi\rangle_{N_k}\longrightarrow0
\)
for every \(\phi\in C_0^\infty(\Omega)\). Moreover, by the standard consistency of the tensor-product LGL
interpolation and quadrature, together with the weak convergence
\(v_k\rightharpoonup v\) in \(H_0^1(\Omega)\) and the strong convergence
\(v_k\to v\) in \(L^2(\Omega)\), passing to the limit in the above
discrete weak identity yields the corresponding continuous weak
identity for \(\mathcal{L}[v]=0\). Then the stability assumption
\eqref{eq:continuous_H2_stability} therefore implies \(v=0\).

By the definition of \(
\mathbf{H}_k=
(\mathbf{W}(\mathbf{C}_b^k)^{-1})
\otimes
(\mathbf{W}(\mathbf{C}_a^k)^{-1}),\) 
and the lower bounds
\(\bar a_{11},\bar a_{22}\geq\alpha_0\), we have
\[
\begin{aligned}
\|\bm{v}_k\|_{\mathbf{H}_k}^2
&=
\sum_{i,j=1}^{N_k-1}
\frac{\omega_i\omega_j}
{\bar a_{11}(x_i)\bar a_{22}(y_j)}
|v_k(x_i,y_j)|^2
\leq
\frac{1}{\alpha_0^2}
\sum_{i,j=0}^{N_k}
\omega_i\omega_j
|v_k(x_i,y_j)|^2.
\end{aligned}
\]
Here the boundary terms vanish since \(v_k\in H_0^1(\Omega)\).
By the uniform equivalence between the tensor-product LGL discrete norm
and the continuous \(L^2\)-norm, we obtain
\[
\sum_{i,j=0}^{N_k}
\omega_i\omega_j
|v_k(x_i,y_j)|^2
\leq
C_{\rm LGL}\|v_k\|_{L^2(\Omega)}^2\longrightarrow0,
\]
where \(C_{\rm LGL}\) is independent of \(N_k\). Consequently,
\(
\|\bm{v}_k\|_{\mathbf{H}_k}\longrightarrow0.
\)
By Lemma~\ref{lem:first_order_relative_bound} and the normalization
\(
\|\mathbf{K}_{\rm L}^{k}\bm{v}_k\|_{\mathbf{H}_k}=1,
\)
we have
\begin{equation}\label{AL1bound}
\|\mathbf{A}_{{\rm L},1}^{k}\bm{v}_k\|_{\mathbf{H}_k}
\leq
\frac{q_\infty}{\sqrt{\alpha_0}}
\|\bm{v}_k\|_{\mathbf{H}_k}^{\frac12}
\longrightarrow0.
\end{equation}
Moreover, since \(0<s\leq s_{\max}\), it is clear that
\begin{equation}\label{Sbound}
\|\breve{\mathbf{S}}^{k}\bm{v}_k\|_{\mathbf{H}_k}
\leq
s_{\max}\|\bm{v}_k\|_{\mathbf{H}_k}
\longrightarrow0.
\end{equation}
Finally, by the eigenvalue relation
\(
\mathbf{A}_{\rm L}^{k}\bm{v}_k
=
\lambda_k\mathbf{K}_{\rm L}^{k}\bm{v}_k
\)
and the normalization above,
\begin{equation}\label{ALbound}
\|\mathbf{A}_{\rm L}^{k}\bm{v}_k\|_{\mathbf{H}_k}
=
|\lambda_k|
\|\mathbf{K}_{\rm L}^{k}\bm{v}_k\|_{\mathbf{H}_k}
=
|\lambda_k|
\longrightarrow0.
\end{equation}
Therefore, combining \eqref{AL1bound}--\eqref{ALbound}, by the definition of 
\(
\mathbf{A}_{{\rm L},2}^{k}
=
\mathbf{A}_{\rm L}^{k}
-
\mathbf{A}_{{\rm L},1}^{k}
-
\breve{\mathbf{S}}^{k},
\)
we have
\[
\|\mathbf A_{{\rm L},2}^{k}\boldsymbol v_k\|_{{\mathbf H}_k}
\leq\|\mathbf{A}_{{\rm L},1}^{k}\bm{v}_k\|_{\mathbf{H}_k}
+
\|\breve{\mathbf{S}}^{k}\bm{v}_k\|_{\mathbf{H}_k}
+
\|\mathbf{A}_{\rm L}^{k}\bm{v}_k\|_{\mathbf{H}_k}
\longrightarrow0.
\]
On the other hand, Lemma~\ref{lem:principal_uniform_equivalence}
implies
\[
\|\mathbf A_{{\rm L},2}^{k}\boldsymbol v_k\|_{{\mathbf H}_k}
\geq
c_1\|\mathbf K_{\rm L}^{k}\boldsymbol v_k\|_{{\mathbf H}_k}
=c_1,
\]
which is a contradiction. This proves the lower bound and hence
\eqref{eq:general_preconditioned_annular_bound}.

\end{proof}

\section{Proof of Theorem~\ref{thm:uniform_condition_number}}
\label{app:uniform_condition_number}

In this section, we provide the detailed proof of the \(N\)-uniform Euclidean \(2\)-norm condition-number bound stated in Theorem~\ref{thm:uniform_condition_number}. We first establish a one-dimensional pointwise estimate for the weighted eigenvectors of the variable-coefficient differentiation matrices in a lemma. We then combine this estimate with the tensor-product eigenstructure of the two-dimensional operator to derive \(N\)-uniform bounds for \(\mathbf P_{\rm L}\) and \(\mathbf A_{\rm L}^{-1}\), which lead to the desired condition-number estimate.

\subsection{A pointwise estimate for the weighted eigenvectors}

We first establish a one-dimensional pointwise estimate for the weighted
eigenvectors of \(-\mathbf C_\nu\mathbf D^{(2)}\), \(\nu=a,b\).
By Lemma \ref{WinD2in_eqA}, we know that
\(
\mathbf H_\nu
(-\mathbf C_\nu\mathbf D^{(2)})
=
-\mathbf W\mathbf D^{(2)}
\)
is symmetric positive definite, 
\(-\mathbf C_\nu\mathbf D^{(2)}\) admits an
\(\mathbf H_\nu\)-orthonormal eigenbasis with positive eigenvalues. This estimate will be the key ingredient in the subsequent
two-dimensional condition-number analysis.

\begin{lemma}\label{lem:pointwise_eigenvector_estimate}
For \(\nu=a,b\), let
\(
\mathbf H_\nu=\mathbf W\mathbf C_\nu^{-1},
\)
and let
\(\{\bm\phi_k^\nu\}_{k=1}^{N-1}\) be an
\(\mathbf H_\nu\)-orthonormal eigenbasis of
\(-\mathbf C_\nu\mathbf D^{(2)}\), satisfying
\[
-\mathbf C_\nu\mathbf D^{(2)}\bm\phi_k^\nu
=
\lambda_k^\nu\bm\phi_k^\nu,
\quad
(\bm\phi_k^\nu)^{\intercal}
\mathbf H_\nu
\bm\phi_\ell^\nu
=
\delta_{k\ell},
\quad
\lambda_k^\nu>0.
\]
Then, for each interior LGL point \(x_i\),
\[
\sum_{k=1}^{N-1}
\frac{|\phi_k^\nu(i)|^2}{\lambda_k^\nu}
\le
\frac{1-x_i^2}{2},
\quad
1\le i\le N-1,
\quad
\nu=a,b.
\]
\end{lemma}

\begin{proof}
For any \(\bm v\in\mathbb R^{N-1}\), let
\(p\in\mathbb P_N^0\) be the corresponding Lagrange interpolant, satisfying
\(
p(x_i)=v_i,
1\le i\le N-1.
\)
Since
\(
\mathbf H_\nu
(-\mathbf C_\nu\mathbf D^{(2)})
=
-\mathbf W\mathbf D^{(2)},
\)
it follows from \eqref{uWDv} that
\[
\bm v^{\intercal}\mathbf H_\nu
\bigl(-\mathbf C_\nu\mathbf D^{(2)}\bigr)\bm v
=
-\bm v^{\intercal}\mathbf W\mathbf D^{(2)}\bm v
=
\int_{-1}^1 |p'(x)|^2\,dx.
\]
For any \(x\in(-1,1)\), using \(p(-1)=p(1)=0\), we have
\[
p(x)
=
\frac{1-x}{2}\int_{-1}^x p'(t)\,dt
-
\frac{1+x}{2}\int_x^1 p'(t)\,dt.
\]
Applying the Cauchy--Schwarz inequality gives
\[
\begin{aligned}
|p(x)|^2
&\le
\left[
\frac{(1-x)^2}{4}(1+x)
+
\frac{(1+x)^2}{4}(1-x)
\right]
\int_{-1}^1 |p'(t)|^2\,dt
=
\frac{1-x^2}{2}
\int_{-1}^1 |p'(t)|^2\,dt.
\end{aligned}
\]
Therefore, at each interior LGL point \(x_i\),
\begin{equation}\label{supineq}
\sup_{\bm v\ne\bm 0}
\frac{|v_i|^2}
{\bm v^{\intercal}\mathbf H_\nu
(-\mathbf C_\nu\mathbf D^{(2)})\bm v}
\le
\frac{1-x_i^2}{2}.
\end{equation}

We next express the supremum in terms of the eigenvectors.
Expanding \(\bm v\) in the
\(\mathbf H_\nu\)-orthonormal eigenbasis gives
\[
\bm v
=
\sum_{k=1}^{N-1}v_k^\nu\bm\phi_k^\nu,
\quad
v_k^\nu
=
(\bm\phi_k^\nu)^{\intercal}\mathbf H_\nu\bm v.
\]
By the orthonormality and the eigenvalue relation,
\[
\bm v^{\intercal}\mathbf H_\nu
(-\mathbf C_\nu\mathbf D^{(2)})\bm v
=
\sum_{k=1}^{N-1}
\lambda_k^\nu |v_k^\nu|^2.
\]
Meanwhile, denoting by \(v_i\) the \(i\)-th component of \(\bm v\), which satisfies \(v_i=p(x_i)\), we have
\(\displaystyle
v_i
=
\sum_{k=1}^{N-1}
v_k^\nu\phi_k^\nu(i).
\)
Hence, by the Cauchy--Schwarz inequality, it leads to
\[
|v_i|^2
\le
\left(
\sum_{k=1}^{N-1}
\lambda_k^\nu |v_k^\nu|^2
\right)
\left(
\sum_{k=1}^{N-1}
\frac{|\phi_k^\nu(i)|^2}{\lambda_k^\nu}
\right).
\]
It follows that
\[
\sup_{\bm v\ne\bm 0}
\frac{|v_i|^2}
{\bm v^{\intercal}\mathbf H_\nu
(-\mathbf C_\nu\mathbf D^{(2)})\bm v}
\le
\sum_{k=1}^{N-1}
\frac{|\phi_k^\nu(i)|^2}{\lambda_k^\nu}.
\]
On the other hand, taking
\[
v_k^\nu
=
\frac{\phi_k^\nu(i)}{\lambda_k^\nu},
\quad
1\le k\le N-1,
\]
attains equality. Therefore,
\begin{equation}\label{supeq}
\sup_{\bm v\ne\bm 0}
\frac{|v_i|^2}
{\bm v^{\intercal}\mathbf H_\nu
(-\mathbf C_\nu\mathbf D^{(2)})\bm v}
=
\sum_{k=1}^{N-1}
\frac{|\phi_k^\nu(i)|^2}{\lambda_k^\nu}.
\end{equation}
Combining \eqref{supineq}--\eqref{supeq} yields
\[
\sum_{k=1}^{N-1}
\frac{|\phi_k^\nu(i)|^2}{\lambda_k^\nu}
\le
\frac{1-x_i^2}{2},
\]
which completes the proof.
\end{proof}

\subsection{Detailed proof}

We now prove Theorem~\ref{thm:uniform_condition_number}.

\begin{proof}
By Proposition~\ref{prop:preconditioned_identity}, it suffices to
consider \(\mathbf A_{\rm L}\mathbf P_{\rm L}\).
Under the assumption, it follows from
\eqref{AL-var-coeff} and \eqref{def:PL_PB_var_coeff} that
\begin{subequations}\label{AL1}
\begin{align}
\mathbf K_{\rm L}
&=
\mathbf P_{\rm L}^{-1}
=
\mathbf I_{N-1}\otimes
(-\mathbf C_a\mathbf D^{(2)})
+
(-\mathbf C_b\mathbf D^{(2)})
\otimes\mathbf I_{N-1},
\label{AL1a}
\\
\mathbf A_{\rm L}
&=
\mathbf K_{\rm L}
+
s\mathbf I_{(N-1)^2}.
\label{AL1b}
\end{align}
\end{subequations}
With the notation in
Lemma~\ref{lem:pointwise_eigenvector_estimate}, we define
\(
\mathbf H
=
\mathbf H_b\otimes\mathbf H_a.
\)
Since
\(\{\bm\phi_n^a\}_{n=1}^{N-1}\) and
\(\{\bm\phi_m^b\}_{m=1}^{N-1}\)
are respectively \(\mathbf H_a\)- and
\(\mathbf H_b\)-orthonormal, the tensor-product vectors
\(
\bm\phi_m^b\otimes\bm\phi_n^a,
1\le m,n\le N-1,
\)
form an \(\mathbf H\)-orthonormal basis of
\(\mathbb R^{(N-1)^2}\).
Moreover,
\begin{equation}\label{eigenKL}
\mathbf K_{\rm L}
(\bm\phi_m^b\otimes\bm\phi_n^a\bigr)
=
(\lambda_m^b+\lambda_n^a)
(\bm\phi_m^b\otimes\bm\phi_n^a\bigr).
\end{equation}
For any \(\bm f\in\mathbb R^{(N-1)^2}\), let
\(
\bm u
=
\mathbf P_{\rm L}\bm f
=
\mathbf K_{\rm L}^{-1}\bm f.
\)
Expanding \(\bm f\) in the above \(\mathbf H\)-orthonormal basis gives
\[
\bm f
=
\sum_{m,n=1}^{N-1}
\widehat f_{mn}
(\bm\phi_m^b\otimes\bm\phi_n^a),
\]
where
\(
\widehat f_{mn}
=
(\bm\phi_m^b\otimes\bm\phi_n^a)^{\intercal}
\mathbf H\bm f.
\)
By the \(\mathbf H\)-orthonormality, we have
\[
\sum_{m,n=1}^{N-1}
|\widehat f_{mn}|^2
=
\|\bm f\|_{\mathbf H}^2,
\quad
\|\bm f\|_{\mathbf H}^2
=
\bm f^{\intercal}\mathbf H\bm f.
\]
Therefore, it leads to
\[
\bm u
=
\sum_{m,n=1}^{N-1}
\frac{\widehat f_{mn}}
{\lambda_m^b+\lambda_n^a}
\bigl(\bm\phi_m^b\otimes\bm\phi_n^a\bigr).
\]
Denoting by \(u_{ij}\) the component of \(\bm u\)
corresponding to the tensor-product node \((x_i,y_j)\), we have
\[
u_{ij}
=
\sum_{m,n=1}^{N-1}
\frac{\widehat f_{mn}}
{\lambda_m^b+\lambda_n^a}
\phi_m^b(j)\phi_n^a(i).
\]
Hence, by the Cauchy--Schwarz inequality,
\[
|u_{ij}|^2
\le
\left(
\sum_{m,n=1}^{N-1}
\frac{
|\phi_m^b(j)|^2
|\phi_n^a(i)|^2}
{(\lambda_m^b+\lambda_n^a)^2}
\right)
\|\bm f\|_{\mathbf H}^2.
\]
Since
\(
(\lambda_m^b+\lambda_n^a)^2
\ge
4\lambda_m^b\lambda_n^a,
\)
Lemma~\ref{lem:pointwise_eigenvector_estimate} yields
\begin{equation}\label{uijbound}
\begin{aligned}
|u_{ij}|^2
&\le
\frac{1}{4}
\left(
\sum_{m=1}^{N-1}
\frac{|\phi_m^b(j)|^2}{\lambda_m^b}
\right)
\left(
\sum_{n=1}^{N-1}
\frac{|\phi_n^a(i)|^2}{\lambda_n^a}
\right)
\|\bm f\|_{\mathbf H}^2
\le
\frac{(1-x_i^2)(1-y_j^2)}{16}
\|\bm f\|_{\mathbf H}^2.
\end{aligned}
\end{equation}
Summing \eqref{uijbound} over all interior tensor-product nodes gives
\[
\begin{aligned}
\|\bm u\|_2^2
&=
\sum_{i,j=1}^{N-1}|u_{ij}|^2
\le
\frac{1}{16}
\left(
\sum_{i=1}^{N-1}(1-x_i^2)
\right)
\left(
\sum_{j=1}^{N-1}(1-y_j^2)
\right)
\|\bm f\|_{\mathbf H}^2
\le
\frac{(N-1)^2}{16}
\|\bm f\|_{\mathbf H}^2.
\end{aligned}
\]
Since
\(
\mathbf H
=
(\mathbf W\mathbf C_b^{-1})
\otimes
(\mathbf W\mathbf C_a^{-1}),
\)
assumption~\eqref{var-coeff-assumption2} implies
\[
\|\bm f\|_{\mathbf H}^2
\le
\frac{\omega_{\max}^2}{\alpha^2}
\|\bm f\|_2^2,
\]
where
\(\displaystyle
\omega_{\max}
=
\max_{1\le i\le N-1}\omega_i.
\)
Recalling the LGL weight estimate used in
Corollary~\ref{Proposition:Cond_S_V},
we have
\(
\omega_{\max}\le \frac{c_2}{N},
\)
where $c_2$ is defined in Corollary~\ref{Proposition:Cond_S_V} and independent of $N$.
Consequently, by \(
\bm u
=
\mathbf P_{\rm L}\bm f
=
\mathbf K_{\rm L}^{-1}\bm f
\), we have
\begin{equation}\label{PLnormbound}
\|\mathbf P_{\rm L}\bm f\|_2
=
\|\bm u\|_2
\le
\frac{(N-1)\omega_{\max}}{4\alpha}
\|\bm f\|_2\quad \Rightarrow \quad \|\mathbf P_{\rm L}\|_2
\le
\frac{c_2}{4\alpha}.
\end{equation}
Since \(s\ge0\), it is clear that
\(
\lambda_m^b+\lambda_n^a+s
\ge
\lambda_m^b+\lambda_n^a.
\)
Therefore, the same argument used to derive \eqref{PLnormbound} yields
\begin{equation}\label{ALinvnorm}
\|\mathbf A_{\rm L}^{-1}\bm f\|_2
\le
\frac{(N-1)\omega_{\max}}{4\alpha}
\|\bm f\|_2\quad \Rightarrow \quad 
\|\mathbf A_{\rm L}^{-1}\|_2
\le
\frac{c_2}{4\alpha}.
\end{equation}

It remains to estimate the condition number of
\(\mathbf A_{\rm L}\mathbf P_{\rm L}\).
By \eqref{AL1} it is clear that
\(
\mathbf A_{\rm L}\mathbf P_{\rm L}
=
\mathbf I_{(N-1)^2}
+
s\mathbf P_{\rm L}.
\)
Therefore, by \eqref{PLnormbound} we have
\begin{equation}\label{ALPLbound}
\|\mathbf A_{\rm L}\mathbf P_{\rm L}\|_2
\le
1+s\|\mathbf P_{\rm L}\|_2
\le
1+\frac{s c_2}{4\alpha}.
\end{equation}
For the inverse, using \eqref{AL1b}, we have
\[
\begin{aligned}
(\mathbf A_{\rm L}\mathbf P_{\rm L})^{-1}
=
(\mathbf A_{\rm L}-s\mathbf I_{(N-1)^2})
\mathbf A_{\rm L}^{-1}
=
\mathbf I_{(N-1)^2}
-
s\mathbf A_{\rm L}^{-1}.
\end{aligned}
\]
Therefore, \eqref{ALinvnorm} leads to
\begin{equation}\label{ALPLinvnorm}
\|(\mathbf A_{\rm L}\mathbf P_{\rm L})^{-1}\|_2
\le
1+s\|\mathbf A_{\rm L}^{-1}\|_2
\le
1+\frac{s c_2}{4\alpha}.
\end{equation}
Hence, combining \eqref{ALPLbound}--\eqref{ALPLinvnorm} completes the proof.

\end{proof}

\bibliographystyle{siamplain}
\bibliography{refpapers}


\end{document}